\documentclass[12pt]{amsart}

\usepackage{amsmath,amsthm,amscd,amsfonts,amssymb,epic,eepic,bbm,tikz-cd,mathtools,mathrsfs,longtable,multicol}
\usepackage[pagebackref,colorlinks=true,linkcolor=blue,citecolor=blue]{hyperref}

\allowdisplaybreaks
\newtheorem{thm}{Theorem}[section]
\newtheorem{cor}[thm]{Corollary}

\newtheorem{lem}[thm]{Lemma}

\newtheorem{prop}[thm]{Proposition}
\theoremstyle{remark}
\newtheorem*{rem}{Remark}

\newcounter{remarkscounter}

\numberwithin{equation}{section}
\newcommand{\A}{\mathbb{A}}
\newcommand{\GL}{\mathrm{GL}}
\newcommand{\gl}{\mathfrak{gl}}
\newcommand{\SL}{\mathrm{SL}}
\newcommand{\ZZ}{\mathbb{Z}}

\newcommand{\QQ}{\mathbb{Q}}

\newcommand{\lto}{\longrightarrow}

\newcommand{\OO}{\mathcal{O}}
\newcommand{\CC}{\mathbb{C}}
\newcommand{\RR}{\mathbb{R}}
\newcommand{\GG}{\mathbb{G}}

\newcommand{\Sp}{\mathrm{Sp}}

\newcommand{\F}{\mathcal{F}}
\newcommand{\quash}[1]{}
\newcommand{\ord}{\mathrm{ord}}
\newcommand\norm[1]{\left\lVert#1\right\rVert}

\theoremstyle{definition}

\newenvironment{psmatrix}
  {\left(\begin{smallmatrix}}
  {\end{smallmatrix}\right)}

\renewcommand{\bar}{\overline}

\numberwithin{equation}{subsection}

\newcommand{\one}{\mathbbm{1}}

\allowdisplaybreaks
\begin{document}

\title[The Fourier transform for triples of quadratic spaces]{The Fourier transform for triples of quadratic spaces}

\author{Jayce R. Getz}
\address{Department of Mathematics\\
Duke University\\
Durham, NC 27708}
\email{jgetz@math.duke.edu}

\author{Chun-Hsien Hsu}
\address{Department of Mathematics\\
University of Chicago\\
Chicago, IL 60637}
\email{chunhsien@uchicago.edu}


\thanks{The authors are thankful for partial support provided by the first author's NSF grant DMS 1901883. 
Any opinions, findings, and conclusions or recommendations expressed in this material are those of the authors and do not necessarily reflect the views of the National Science Foundation.  The first author also thanks D.~Kazhdan for travel support under his ERC grant AdG 669655.  Part of this paper was written while the first author was on sabbatical at the Postech Mathematics Institute in South Korea.  He thanks the institute and the Postech Math Department for their hospitality and excellent working conditions.}

\begin{abstract}
Let $V_1,V_2,V_3$ be a triple of even dimensional vector spaces over a number field $F$ equipped with nondegenerate quadratic forms $\mathcal{Q}_1,\mathcal{Q}_2,\mathcal{Q}_3$, respectively.  Let $
Y  \subset \prod_{i=1}^3 V_i
$
be the closed subscheme consisting of $(v_1,v_2,v_3)$ such that $\mathcal{Q}_1(v_1)=\mathcal{Q}_2(v_2)=\mathcal{Q}_3(v_3)$.   One has a Poisson summation formula for this scheme under suitable assumptions on the functions involved, but the relevant Fourier transform was previously only defined as a correspondence.  In the current paper we employ a novel global-to-local argument to prove that this Fourier transform is well-defined on the Schwartz space of $Y(\A_F).$  To execute the global-to-local argument, we introduce boundary terms and thereby extend the Poisson summation formula to a broader class of test functions. This is the first time a summation formula with boundary terms has been proven for a spherical variety that is not a Braverman-Kazhdan space.
\end{abstract}

\maketitle

\setcounter{tocdepth}{1}
{\small
\tableofcontents
}

\section{Introduction}\label{sec:intro}

Let $d_1,d_2,d_3$ be three positive even integers, let  $V_i:=\GG_a^{d_i}$, $V:=\oplus_{i=1}^3V_i$, and let $F$ be a number field.  For each $i$, let $\mathcal{Q}_i$ be a nondegenerate quadratic form on $V_i(F)$.  
Let
 $Y \subset V$ be the closed subscheme whose points in an $F$-algebra $R$ are given by 
\begin{align*}
Y(R):&=\{ (v_1,v_2,v_3) \in V(R):\mathcal{Q}_1(v_1)=\mathcal{Q}_2(v_2)=\mathcal{Q}_3(v_3)\}.
\end{align*}
In \cite{Getz:Liu:Triple} the first author and Liu proved a Poisson summation formula for this space.  The space $Y$ is a spherical variety, and hence the summation formula confirms a special case of  conjectures of Braverman and Kazhdan, later investigated by L.~Lafforgue and Ng\^o, and extended to spherical varieties by Sakellaridis
\cite{BK:basic:affine,BK-lifting,BK:normalized,LafforgueJJM,Ngo:Hankel,SakellaridisSph}.  
It is the first  summation formula for a spherical variety that is not a Braverman-Kazhdan space.  Here a Braverman-Kazhdan space is the affine closure of $[P,P] \backslash G$ where $G$ is a reductive group and $P <  G$ is a parabolic subgroup.

In this paper we prove that the Fourier transform on $Y,$ originally defined as a correspondence, descends to an automorphism of the Schwartz space.  Let us be more precise.  Let $X^\circ:=[P,P] \backslash \mathrm{Sp}_6,$ where $P < \mathrm{Sp}_6$ is the Siegel parabolic, and let $X:=\overline{\mathrm{Pl}(X^{\circ})}$ be the corresponding Braverman-Kazhdan space (see \eqref{X}). As explained in \S \ref{sec:BK}, the Schwartz space $\mathcal{S}(X(\A_F))$ is defined and comes equipped with a Fourier transform $\mathcal{F}_X: \mathcal{S}(X(\A_F)) \to \mathcal{S}(X(\A_F))$.  We define $\mathcal{S}(X(\A_F) \times V(\A_F))$ using the conventions in \S \ref{sec:gen:Schwartz}.  
For notational simplicity, we let
$$
\mathcal{F}_X:\mathcal{S}(X(\A_F) \times V(\A_F)) \lto \mathcal{S}(X(\A_F) \times V(\A_F))
$$
be the automorphism given on pure tensors by $\mathcal{F}_X(f_1 \otimes f_2)=\mathcal{F}_X(f_1) \otimes f_2$. Let $Y^{\mathrm{sm}} \subset Y$ be the smooth locus. 
In this paper we introduce the \textbf{Schwartz space}
$$
\mathcal{S}(Y(\A_F)):=\mathrm{Im}\big(I:\mathcal{S}(X(\A_F) \times V(\A_F)) \lto C^\infty(Y^{\mathrm{sm}}(\A_F))\big)
$$
where $I$ is defined as in \eqref{Is:global:intro} below.  
The Poisson summation formula in \cite{Getz:Liu:Triple} relies not on a Fourier transform from $\mathcal{S}(Y(\A_F))$ to itself, but a correspondence
\medskip
\begin{center}
\begin{tikzcd}[column sep=large]
\mathcal{S}(X(\A_F)\times V(\A_F))
\arrow[d, two heads, "I"]
\arrow[r, "\mathcal{F}_X"] &\mathcal{S}(X(\A_F) \times V(\A_F)) \arrow[d,two heads, "I"]\\
\mathcal{S}(Y(\A_F)) \arrow[r,dotted]
& \mathcal{S}(Y(\A_F)).
\end{tikzcd}
\end{center} 
 In the current paper we prove the following theorem:
\begin{thm} \label{thm:FY:intro}  Assume $Y^{\mathrm{sm}}(\A_F)$ is nonempty.  
There is a unique $\CC$-linear isomorphism $\mathcal{F}_Y:\mathcal{S}(Y(\A_F)) \to \mathcal{S}(Y(\A_F))$ such that 
$
I \circ \mathcal{F}_X=\mathcal{F}_Y \circ I.$
\end{thm}
\noindent In other words, the dotted arrow in the diagram above can be replaced by $\mathcal{F}_Y$ and the resulting diagram is commutative.  Theorem \ref{thm:FY:intro} follows from Theorem \ref{thm:FY} below.  We prove in Proposition \ref{prop:L2} below that $\mathcal{S}(Y(F_v))$ is contained in $L^2(Y(F_v))$ (with respect to an appropriate measure) for all places $v$ of $F.$  
As an application of Theorem \ref{thm:FY:intro}, in a follow-up paper \cite{Getz:Hsu:Leslie} with S.~Leslie, we give an explicit formula for $\mathcal{F}_Y$ and prove that it extends to a unitary operator in the non-Archimedean case.  This will constitute a sound setup for harmonic analysis on $Y(F_v)$. 
We refer to  \cite{Getz:Quadrics,Gurevich:Kazhdan,Kobayashi:Mano} for analogous work when $Y$ is replaced by the zero locus of a single quadratic form.  We would like to emphasize again that the setting of the current paper is a significant leap from the setting of these other papers because our space is not a Braverman-Kazhdan space.

We prove Theorem \ref{thm:FY:intro} via global-to-local argument.  Though global-to-local arguments using summation formulae such as the trace formula are common in the literature, the authors do not know of another example where such a technique is used to define a unitary operator.  The argument is contained in \S \ref{sec:appendix}.  To execute it, we prove a more flexible version of the Poisson summation formula of \cite{Getz:Liu:Triple} that involves boundary terms.  We also develop Fourier and harmonic analysis on $\mathcal{S}(Y(\A_F))$ to the point that we can take limits of functions. This work is of independent interest.

In remainder of the introduction, we state the Poisson summation formula we prove in this paper. Before  stating it in full generality, we highlight a special case.  Fix a nontrivial additive character $\psi:F \backslash \A_F \to \CC^\times.$  For $1 \leq i \leq 3$, let
 $$
\rho_i:=\rho_{i,\psi}:\SL_2(\A_F) \times \mathcal{S}(V_i(\A_F)) \lto \mathcal{S}(V_i(\A_F))
 $$
 be the Weil representation attached to $\psi$ and the $\mathcal{Q}_i.$  
 For a place $v$ of $F,$ let
 \begin{align}
  \label{intro:Siv}
  \begin{split}\mathcal{S}_{iv}:&=\{f \in \mathcal{S}(V_i(F_v)): \rho_i(g)f(0)=0 \textrm{ for all }g \in \SL_2(F_v)\},\\
    \mathcal{S}_{0v}:&=\mathcal{S}_{1v}\otimes \mathcal{S}_{2v}\otimes \mathcal{S}_{3v}.\end{split}
 \end{align}
From the definition of the Weil representation and the Bruhat decomposition of $\SL_2(F_v)$ we have
\begin{align}
\mathcal{S}_{iv}=\left\{f \in \mathcal{S}(V_i(F_v)): f(0)=0\textrm{ and } \int_{V_i(F_v)} \psi(tQ_i(w))f(w)dw=0 \textrm{ for all }t \in F_v\right\}.
\end{align}

Restrictions of elements of $\mathcal{S}_{0v}$ to $Y^{\mathrm{sm}}(F_v)$ are elements of $\mathcal{S}(Y(F_v))$ by  Lemma \ref{lem:smooth}.  We check in Lemma \ref{lem:nontriv} below that $\mathcal{S}_{0v}$ is nonzero for finite $v \nmid 2$ (in fact, infinite-dimensional).  
 \begin{thm}  \label{thm:PS}
 Let $f \in \mathcal{S}(Y(\A_F))$.  Assume that there are finite places $v_1,v_2$ of $F$ such that $f=f_{v_1}f_{v_2}f^{v_1v_2}$ where $f_{v_1}$ and $\mathcal{F}_Y(f_{v_2})$ are restrictions of elements of $\mathcal{S}_{0v_1}$ and $\mathcal{S}_{0v_2}$, respectively.   Then 
 $$
 \sum_{y \in Y^{\mathrm{sm}}(F)} f(y)=\sum_{y \in Y^{\mathrm{sm}}(F)} \mathcal{F}_{Y}(f)(y).
 $$
 \end{thm}
 \noindent This is similar to the main theorem of \cite{Getz:Liu:Triple}, but Theorem \ref{thm:PS} has the additional benefit that the hypotheses and conclusion are given intrinsically in terms of the space $\mathcal{S}(Y(\A_F))$ and not extrinsically in terms of the map $I:\mathcal{S}(X(\A_F) \times V(\A_F)) \to \mathcal{S}(Y(\A_F)).$
 
\subsection{The boundary terms} \label{ssec:bt}
We now describe our main summation formula more precisely.    Possibly confusing (but useful) notational conventions on Schwartz spaces are given in \S \ref{sec:gen:Schwartz}. 
Let $G$ be the image of $\SL_2^3$ under a natural embedding $\SL_2^3 \to \mathrm{Sp}_6$ (see \eqref{Gembed}). The quasi-affine scheme $X^{\circ}=[P,P]\backslash \mathrm{Sp}_6$  admits a  natural right $G$-action.

Over a field of characteristic zero, there are five orbits in $X^{\circ}$ under the action of $G.$ 
We fix representatives $\gamma_b,\gamma_0:=\mathrm{Id},\gamma_1,\gamma_2,\gamma_3$ as in \S \ref{sec:groups:orbits} and let $G_\gamma$ be the  stabilizer of $\gamma$ in $G.$  The subscript $b$ stands for basepoint; $\gamma_b$ is a representative for the open orbit.  

We have a Weil representation 
$$
\rho:=\rho_{1} \otimes \rho_2 \otimes \rho_3:G(\A_F) \times \mathcal{S}(V(\A_F)) \lto \mathcal{S}(V(\A_F)).
$$
We will require the following assumption on $f=f_1\otimes f_2 \in \mathcal{S}(X(\A_F) \times V(\A_F))$: There are finite places $v_1,v_2$ of $F$ such that
\begin{align}\label{cptsupport0}
    f_1&=f_{v_1}f_{v_2}f^{v_1v_2} \textrm{ and } f_{v_1} \in C_c^\infty(X^{\circ}(F_{v_1})),\textrm{ } \mathcal{F}_X(f_{v_2}) \in C_c^\infty(X^{\circ}(F_{v_2})),\\
\label{awayfromVcirc0}
    \rho(g)f_2(v)&=0\,\,\textrm{for $v\not\in V^\circ(F)$,\, for all $g\in G(\A_F)$.}
\end{align}
Here $V^\circ$ is the open subscheme of $V$ consisting of triples $(v_1,v_2,v_3)$ with each $v_i \neq 0$.  The origin of condition \eqref{awayfromVcirc0} is explained below \eqref{Theta}.  We point out that if $f_2=f_{2v}f_2^v$ for some place $v$ of $F$ with $f_{2v} \in \mathcal{S}_{0v}$ then $f_2$ satisfies \eqref{awayfromVcirc0}.

\begin{rem}
A similar condition on $f_2$ was assumed in \cite{Getz:Liu:Triple}.  We warn the reader that in loc.~cit. the assertion that (5.0.7) implies (5.0.5) is false.  Fortunately, this claim is never used in loc.~cit.
\end{rem}

Let $\Phi\in \mathcal{S}(\mathbb{A}_F^2)$ and let $N_2 < \SL_2$ be the subgroup of upper triangular unipotent matrices. For
$f=f_1\otimes f_2 \in \mathcal{S}(X(\A_F) \times V(\A_F))$ we define
\begin{align} \label{Is:global:intro}\begin{split}
I(f)\left(\xi \right)&:=\int_{G_{\gamma_b}(\A_F) \backslash G(\A_F)} f_1\left(\gamma_bg\right) \rho\left(g\right)f_2(\xi)\,dg\,\quad \text{for  } \xi \in Y^{\mathrm{sm}}(\A_F),\\
I_{0}(f)\left(\xi \right)&:=\int_{N_2^3(\A_F) \backslash G(\A_F)} f_1\left(g\right) \rho\left(g\right)f_2(\xi)\,dg\, \quad \text{for  } \xi \in \widetilde{Y}_0(\A_F). \end{split}
\end{align}
Here $\widetilde{Y}_0$ (and $\widetilde{Y}_i$ for $1 \leq i \leq 3$) is defined as in \S\ref{sec:groups:orbits}.   
For $\xi \in \widetilde{Y}_i(\A_F)$ and $s\in \CC,$  we set
\begin{align*}
&I_{i}(f\otimes \Phi)(\xi,s)\\&:=
\int_{G_{\gamma_i}(\A_F) \backslash G(\A_F)}f_1\left(\gamma_{i}g\right) \left(
\int_{N_2(\A_F) \backslash \SL_2(\A_F)}\int_{\A_F^\times}
\rho\left(\Delta_i(h)g\right)f_2(\xi )\Phi(x(0,1)hp_i(g))|x|^{2s}d^{\times}xdh\right)dg
\end{align*} 
where $\Delta_i$ is defined as in \eqref{Deltai} and $p_i$ is defined as in \eqref{pi}.  In \S\ref{sec:groups:orbits},  certain quotient schemes $Y_i$ of $\widetilde{Y}_i$ are also defined. Roughly, $Y_0(F)$ is a quotient of
$$
\widetilde{Y}_0(F):=\{(v_1,v_2,v_3) \in V^\circ (F): \mathcal{Q}_1(v_1)=\mathcal{Q}_2(v_2)=\mathcal{Q}_3(v_3)=0\}
$$
by an action of $(F^\times)^2$
and $Y_i$ is the product of the zero locus of $\mathcal{Q}_i$ in $V_i^\circ$ and the quasi-projective scheme cut out of $\mathbb{P}(V_{i-1}^\circ \times V_{i+1}^\circ)$ by $\mathcal{Q}_{i-1}-\mathcal{Q}_{i+1}$, where the indices are understood ``modulo $3$'' in the obvious sense.  Here $\mathbb{P}(V_{i-1}^{\circ} \times V_{i+1}^{\circ})$ is the image of $V_{i-1}^{\circ} \times V_{i+1}^\circ$ in $\mathbb{P}(V_{i-1} \times V_{i+1}).$

The main summation formula  proved in this paper is the following theorem:

\begin{thm} \label{thm:main:intro}
Assume that 
$$
(f=f_1\otimes f_2,\Phi) \in \mathcal{S}(X(\A_F) \times V(\A_F)) \times \mathcal{S}(\A_F^2)
$$ 
where $f$ satisfies \eqref{cptsupport0} and  \eqref{awayfromVcirc0}, and $\widehat{\Phi}(0) \neq 0.$  Let $\kappa_F=\frac{2}{\mathrm{Vol}(F^\times \backslash (\A_F^\times)^1)}.$ One has 
\begin{align*}
&\sum_{\xi \in Y^{\mathrm{sm}}(F)} I(f)(\xi)
+\sum_{\xi \in   Y_0(F)}I_0(f)(\xi)+\frac{\kappa_F}{\widehat{\Phi}(0)}
\mathrm{Res}_{s=1}\sum_{i=1}^3 \sum_{\xi \in Y_i(F)}I_i(f\otimes \Phi)(\xi,s)\\
&=\sum_{\xi \in Y^{\mathrm{sm}}(F)} I(\mathcal{F}_X(f))(\xi)+\sum_{\xi \in Y_0(F)} I_0(\mathcal{F}_X(f))(\xi)+
\frac{\kappa_F}{\widehat{\Phi}(0)}\mathrm{Res}_{s=1}\sum_{i=1}^3 \sum_{\xi \in Y_i(F)}I_i(\F_X(f)\otimes \Phi)(\xi,s).
\end{align*} 
\end{thm}
\noindent Here $\widehat{\Phi}$ denotes the Fourier transform of $\Phi$ normalized as in \eqref{2dFT},
and $(\A_F^\times)^1<\A_F^\times$ is the subgroup of ideles of norm $1$.  When we speak of \textbf{boundary terms} in the paper, we mean the summands in Theorem \ref{thm:main:intro} involving $I_0$ and $I_i$.  The proof of Theorem \ref{thm:main:intro} is a refinement of the proof of the main theorem of \cite{Getz:Liu:Triple}.  Briefly, one substitutes a triple of $\Theta$-functions into the integral representation of the triple product $L$-function due to Garrett \cite{GarrettTripleAnnals}.  We make use of the adelic reformulation of Piatetski-Shapiro and Rallis \cite{PSRallisTriple}.  The boundary terms correspond to the $\SL_2^3$-orbits in $X^\circ$ that are not open.  In \cite{Getz:Liu:Triple} assumptions were made to eliminate these terms.  At the suggestion of the referee, we point out that Theorem \ref{thm:main:intro} implies that the same formula is valid if $f$ is replaced by a $\CC$-linear combination of functions $f$ satisfying conditions \eqref{cptsupport0} and \eqref{awayfromVcirc0}; in particular, one can allow the places $v_1,v_2$ to vary.

Of course it would be desirable to remove assumptions \eqref{cptsupport0} and  \eqref{awayfromVcirc0}.  To obtain an identity without these assumptions will    require the addition of boundary terms on both sides of the formula.  Thus the statement of the Poisson summation formula for general test functions will be more intricate.

The proof would also be far more technical as we now explain.  There is no assumption on support in the main theorem of \cite{Getz:Liu:BK}, but the boundary terms (i.e. those not given by the evaluation of a Schwartz function on a point of $X^\circ(F)$) are given in terms of residues of Eisenstein series.  It seems wise to wait until there is an explicit geometric understanding of these terms before attempting to remove \eqref{cptsupport0}.  

On the other hand, assumption \eqref{awayfromVcirc0} implies that the $\Theta$-function 
\begin{align} \label{Theta}
\Theta_f(g):=\sum_{\xi \in V(F)}\rho(g)f(\xi)
\end{align}
is cuspidal as a function of $g \in \SL_2^3(F) \backslash \SL_2^3(\A_F).$  Removing this assumption would probably require using Arthur truncation as in \cite{Getz:Quadrics}.  Even in the simpler situation of the zero locus of a single quadratic form in loc.~cit., one has to come to grips with a host of additional complications.  Explicitly, truncation introduces new analytic difficulties, the lack of invariance of the truncation requires attention, and at the end one has to introduce a whole family of boundary terms indexed by all lower dimensional quadratic forms in the Witt class of the original quadratic form.  We also suspect that using Arthur truncation it might be possible to rewrite the terms involving residues at $s=1.$

In any case, all of the possible refinements above, though interesting, are not necessary to define the Schwartz space of $Y$ and prove that the Fourier transform $\mathcal{F}_Y$ is well-defined.  For this purpose, we use the fact that the summation formula above allows us to treat $f_1$ such that $f_1$ and $\mathcal{F}_{X}(f_1)$ are not supported on the open $\SL_2^3(F)$-orbit in $X^{\circ}(F).$  The corresponding summation formula in \cite{Getz:Liu:Triple} was limited to functions satisfying this additional assumption.

\subsection{Outline}
We now outline the sections of this paper.  
We state conventions for Schwartz spaces in \S \ref{sec:gen:Schwartz}.  
In
 \S \ref{sec:BK} we recall and refine certain results from harmonic analysis on Braverman-Kazhdan spaces. In particular we prove a Plancherel formula for  $\mathcal{S}(X(F_v))$  (see Proposition \ref{prop:cont}).  

After this, we discuss the geometric preliminaries necessary for the study of $Y$ in \S \ref{sec:groups:orbits}.  We turn in \S \ref{sec:loc:func} to the definition of the local integrals necessary to prove our main summation formula, Theorem \ref{thm:main:intro}.   We give a definition of the Schwartz space of $Y(F_v)$ in \S \ref{ssec:Schwartz:Y}.  

Theorem \ref{thm:main:intro} is proved in \S \ref{sec:summation}. This summation formula is given in terms of infinite sums of Eulerian integrals, that is, integrals that factor along the places of $F$ (or residues of such expressions).  
The local integrals are computed in the unramified case in \S \ref{sec:unr}.
The proof of Theorem \ref{thm:main:intro} depends on bounds on local integrals that are deferred to  \S \ref{sec:bound:na} and \S\ref{sec:bound:arch}. We discuss the $L^2$-theory in \S \ref{sec:l2theory}, and prove in particular that $S(Y(F_v)) <L^1(Y(F_v)) \cap L^2(Y(F_v))$ with respect to a natural measure.  In \S \ref{sec:appendix} we construct the isomorphism $\mathcal{F}_Y$ as described above and prove Theorem \ref{thm:PS}.
We have appended a list of symbols for the reader's convenience.

\section*{Acknowledgements}

In response to the paper \cite{Getz:Liu:Triple}
several people including  Y.~Sakellaridis and Z.~Yun  asked the first author about whether it was possible to define $\mathcal{F}_Y$; he thanks them for this question. He also thanks D.~Kazhdan for many interesting conversations on material related to the topic of this paper, and H.~Hahn for her help with editing and her constant encouragement.  The authors thank D.~Kazhdan for suggesting that a discussion of the boundary terms be added to the introduction, T.~Ikeda for answering a question on his paper \cite{Ikeda:poles:triple}, and J.-L.~Colliot-Thélène for pointing out his result with Sansuc in \cite{Colliot1982} (see Theorem \ref{thm:CTS}).  We also acknowledge the comments of the anonymous referee; they helped improve the exposition.

\section{Schwartz spaces} 
\label{sec:gen:Schwartz}
This work involves several Schwartz spaces. Let $F$ be a local field. For a quasi-affine scheme $X$ of finite type over $F,$ let $X^{\mathrm{sm}} \subset X$ be the smooth locus.  Any Schwartz space $\mathcal{S}(X(F))$ we discuss will be a space of functions on $X^{\mathrm{sm}}(F)$.  Functions in the Schwartz space need not be defined on all of $X(F)$ if $X$ is not smooth.  We will not define Schwartz spaces of general quasi-affine schemes of finite type over $F.$  In fact obtaining a good definition for general spherical varieties is an important open problem \cite{SakellaridisSph}.  
In this subsection we explain the definition for smooth quasi-affine schemes and how to form Schwartz spaces of products $X(F) \times Y(F)$ given that the Schwartz spaces of each factor have been defined.     We have modeled our approach on the treatment of the smooth case in \cite{AG:Nash}.

Suppose $F$ is non-Archimedean.  If $X$ is smooth, we set
$
\mathcal{S}(X(F)):=C_c^\infty(X(F)).$
More generally, if we have already defined $\mathcal{S}(X(F))$ and $\mathcal{S}(Y(F))$, we set
$
\mathcal{S}(X(F) \times Y(F)):=\mathcal{S}(X(F)) \otimes \mathcal{S}(Y(F))
$
(the algebraic tensor product).  

Now assume that $F$ is Archimedean. If $X$ is smooth,  we define $\mathcal{S}(X(F))=\mathcal{S}(\mathrm{Res}_{F/\RR}X(\RR))$ as in \cite[Remark 3.2]{Elazar:Shaviv}.  By \cite[\S 2.2]{Elazar:Shaviv}, the Schwartz space of a real algebraic variety and the Schwartz space of its underlying Nash manifold defined in \cite{AG:Nash} may be canonically identified.
In any case $\mathcal{S}(X(F))$ is a  Fr\'echet space.  It is defined as a quotient of a nuclear space by a closed subspace, and hence is nuclear. In general,  suppose that we have defined Schwartz spaces $\mathcal{S}(X(F))$ and $\mathcal{S}(Y(F))$ that are additionally Fr\'echet spaces. 
We then define 
$
\mathcal{S}(X(F) \times Y(F)):=\mathcal{S}(X(F)) \widehat{\otimes} \mathcal{S}(Y(F))
$
where the hat denotes the (completed) projective topological tensor product.  It is also a Fr\'echet space.

We warn the reader that in \cite{Elazar:Shaviv} there is a definition of a Schwartz space for any quasi-affine scheme of finite type over the real numbers.  In the smooth case their definition coincides with ours.  In the limited situations in which we define Schwartz spaces of nonsmooth schemes, our definitions do not coincide with theirs because functions in our Schwartz spaces need not extend to the singular set.  

Finally we discuss the adelic setting.  Let $X$ be a quasi-affine scheme of finite type over a number field $F.$  Then for all finite sets $S$ of places of $F,$ $X(\A_F^S)$ is defined as a topological space \cite{Conrad_adelic_points}.  
 Assume $\mathcal{S}(X(F_v))$ is defined for all places $v$ and a \textbf{basic function} $b_{X,v}\in \mathcal{S}(X(F_v))$ is chosen for almost all $v.$   
 If $S$ contains all infinite places, $\mathcal{S}(X(\A_F^S))$ will always be a restricted tensor product
$\otimes_{v \not\in S}' \mathcal{S}(X(F_v))$
with respect to  $b_{X,v}$; if $S$ is a set of infinite places of $F,$ then 
$
\mathcal{S}(X(F_S)):=\widehat{\otimes}_{v\in S}\mathcal{S}(X(F_v)) 
$  is the (completed) projective topological tensor product. For general $S$, we put
 $$
 \mathcal{S}(X(\A_F^S)):=\mathcal{S}(X(F_{\infty-S})) \otimes \mathcal{S}(X(\A_F^{\infty \cup S})).
 $$

\section{Braverman-Kazhdan spaces}
\label{sec:BK}
In this section we study Schwartz spaces of certain Braverman-Kazhdan spaces.  In particular we endow these Schwartz spaces with a Fr\'echet structure in the Archimedean case.  This refinement is necessary for continuity arguments.  

Let $\mathrm{Sp}_{2n}$ denote the symplectic group on a $2n$-dimensional vector space, and let $P < \mathrm{Sp}_{2n}$, $M < P$ denote the usual Siegel parabolic and Levi subgroup.  More specifically, for $\ZZ$-algebras $R$, set
\begin{align}\begin{split}\label{eq:levi}
\mathrm{Sp}_{2n}(R):&=\left\{g \in \GL_{2n}(R): g^t\begin{psmatrix} & I_n \\ -I_n & \end{psmatrix}g:=\begin{psmatrix} & I_n \\ -I_n & \end{psmatrix}\right\},\\
M(R):&=\{\begin{psmatrix} A & \\ & A^{-t} \end{psmatrix}: A \in\GL_n(R)\},\\
N(R):&=\{\begin{psmatrix} I_n &  Z\\ & I_n \end{psmatrix} : Z \in \gl_n(R), Z^t=Z\}, \end{split}
\end{align}
and $P=MN.$  Apart from this section, we will only use the $n=3$ case, but since it is no more difficult to treat the general case, we include it.
We define a character
\begin{align} \label{omega} \begin{split}
\omega:M(R) &\lto R^\times\\
\begin{psmatrix} A & \\ & A^{-t} \end{psmatrix} &\longmapsto \det A. \end{split}
\end{align}
Let 
\begin{align} \label{XP}
    X^{\circ}:=[P,P] \backslash \mathrm{Sp}_{2n}.
\end{align}
Let $\mathrm{GSp}_{2n}$ denote the group of similitudes and let 
$$
\nu:\mathrm{GSp}_{2n} \lto \GG_m
$$
denote the similitude norm.
We note that there is a left action
\begin{align}
    \label{act}
    \begin{split}
    M^{\mathrm{ab}}(R) \times \mathrm{GSp}_{2n}(R) \times X^{\circ}(R) &\lto X^{\circ}(R)\\
    (m,g,x) &\longmapsto m\begin{psmatrix} I_n & \\ & \nu(g) I_n \end{psmatrix}xg^{-1}.
    \end{split}
\end{align}
One has the Pl\"ucker embedding
\begin{align}\label{pluckerembed}
    \mathrm{Pl}: X^{\circ} \lto \wedge^{n}\GG_a^{2n}
\end{align}
given by taking the wedge product of the last $n$ rows of a representative $g \in \Sp_{2n}(R)$ for $[P,P](R)g$.  We denote by $X$ the closure of $\mathrm{Pl}(X^{\circ})$:
\begin{align} \label{X}
    X:=\bar{\mathrm{Pl}(X^{\circ})}.
\end{align}
It is an affine variety (in fact a spherical variety, for many more details see \cite[\S 7.2]{WWLi:Zeta}).  As explained in loc.~cit., $X$ is the affine closure of $X^{\circ}$.

\subsection{The local Schwartz spaces} \label{ssec:loc:Schwartz}
Let $F$ be a local field.  Let 
$$
\delta_P(p):=|\det \left(\mathrm{Ad}(p):\mathrm{Lie}\,P \to \mathrm{Lie}\,P \right)|
$$
be the modular quasi-character of $P(F).$ The Schwartz space $\mathcal{S}(X(F))$ of $X(F)$ was defined in \cite[\S 5]{Getz:Hsu:Leslie}, where it was denoted $\mathcal{S}(X_P(F)).$  
For $f \in C^\infty(X^{\circ}(F))$ and $g \in \Sp_{2n}(F)$, and a quasi-character $\chi:F^\times\to\CC$, let
\begin{align} \label{Mellin}
f_{\chi_s}(g):=\int_{M^{\mathrm{ab}}(F)}\delta_P(m)^{1/2}\chi_s(\omega(m))f(m^{-1}g)dm
\end{align}
be its Mellin-transform. Here $\chi_s:=\chi|\cdot|^s$ and $\omega$ is defined as in \eqref{omega}. When this integral is well-defined, either because it converges absolutely or is meromorphically continued from a half-plane of absolute convergence, it is a section of 
$$
I(\chi_s):=\mathrm{Ind}_{P}^{\mathrm{Sp}_{2n}}(\chi_s\circ \omega).
$$ 
  Here the induction is normalized and taken in the category of smooth representations. 

Let $\psi:F \to \CC^\times$ be a nontrivial additive character.
\begin{thm}\cite[Theorem 4.4]{Getz:Liu:BK}, \cite[\S 5.3]{Getz:Hsu:Leslie} \label{thm:44}
There is a linear automorphism 
\begin{align*} 
\mathcal{F}_X:=\mathcal{F}_{X,\psi}:\mathcal{S}(X(F)) \lto \mathcal{S}(X(F)).
\end{align*}
For $f \in \mathcal{S}(X(F))$, the Fourier transform 
$\mathcal{F}_X(f)$ is the unique function in $\mathcal{S}(X(F))$ such that 
\begin{align*} 
\mathcal{F}_X(f)_{\chi_s}=M_{w_0}^*(f_{\bar{\chi}_{-s}})
\end{align*}
for all (unitary) characters $\chi$ and all $s \in \CC$.  \qed
\end{thm}
\noindent Here \begin{align} \label{Mw0psi}
    M_{w_0}^*:=M_{w_0,\psi}^*:=\Big(\gamma(s+\tfrac{1-n}{2},\chi,\psi)\prod_{r=1}^{\lfloor n/2\rfloor}\gamma(2s-n+2r,\chi^2,\psi)\Big)M_{w_0}:I(\chi_s) \lto I(\bar{\chi}_{-s})
\end{align}
is the normalized intertwining operator of \cite[(3.5)]{Getz:Liu:BK}. The Tate $\gamma$-factors depend on a choice of Haar measure on $F$ which we always take to be the self-dual measure with respect to $\psi$.  In \cite[Corollary 6.10]{Getz:Hsu:Leslie} one finds an explicit formula for $\mathcal{F}_X.$  We point out that $
    M_{w_0}=\iota_{w_0} \circ \mathcal{R}_{P|P^{\mathrm{op}}}
   $
    in the notation of loc.~cit.

Let $F$ be an Archimedean local field. For real numbers $A < B $, $p(x) \in \CC[x]$ and meromorphic functions $f:\CC \to \CC$, let 
\begin{align} \label{semi:normf}
\begin{split}
V_{A,B}:&=\{s \in \CC:A \leq \mathrm{Re}(s) \leq B\},\\
|f|_{A,B,p}:&=\mathrm{sup}_{s \in V_{A,B}}|p(s)f(s)|.
\end{split}
\end{align}
To complete our discussion of $\mathcal{S}(X(F))$ we must endow it with a topology.  Recall the $L$-factors $a_w(s,\eta)$ indexed by $w \in \{\mathrm{Id},w_0\}$ from \cite{Getz:Liu:BK}. 
Consider the Lie algebra
$$
\mathfrak{g}:=\mathrm{Lie}(M^{\mathrm{ab}}(F) \times  \Sp_{2n}(F)),
$$
viewed as a real Lie algebra.  It acts on $C^\infty(X^{\circ}(F))$ via the differential of the action \eqref{act} and hence we obtain an action of $U(\mathfrak{g})$, the universal enveloping algebra  of the complexification of $\mathfrak{g}$. Let  $\widehat{K}_{\GG_m}$ be a set of representatives for the (unitary) characters of $F^\times$ modulo equivalence, where $\chi$ is equivalent to $\chi'$ if and only if $\chi=|\cdot|^{it}\chi'$ for some $t \in \RR$. 
For all real numbers $A<B$, $w \in \{\mathrm{Id},w_0\}$, $D \in U(\mathfrak{g}),$ any polynomials $p_{w} \in \CC[s]$ such that $p_{w}(s)a_w(s,\eta)$ has no poles for all $(s,\eta) \in V_{A,B} \times \widehat{K}_{\GG_m}$ and compact subsets $\Omega \subset X^{\circ}(F)$, let
\begin{align} \label{seminorm}
|f|_{A,B,w,p_w,\Omega,D}:=\sum_{\eta \in \widehat{K}_{\GG_m}}\mathrm{sup}_{g \in \Omega}|M_w(D.f)_{\eta_s}(g)|_{A,B,p_w}.
\end{align}
By definition of the Schwartz space \cite[\S 5]{Getz:Hsu:Leslie} this is a seminorm on $\mathcal{S}(X(F))$ and the collection of these seminorms as $A,B,p_w,\Omega,D$ vary  gives $\mathcal{S}(X(F))$ the structure of a locally convex space.    

One could probably rewrite the seminorm in \eqref{seminorm} in terms of asymptotics toward the origin in place of Mellin transforms using work of Igusa (see \cite{JLL:Harmonic,CH:nonarch} for non-Archimedean analogues).  The expressions one would likely obtain are not any simpler or more conceptual from our point of view.  We do not know whether it is possible to remove dependence on the intertwining operator $M_{w_0}$ in the definition of the topology.

\begin{lem} 
Assume $F$ is Archimedean. The space $\mathcal{S}(X(F))$ is a Fr\'echet space.
\end{lem}

\begin{proof}
We first observe that we can replace the family of seminorms with a countable subfamily inducing the same topology.  More specifically we can choose a countable basis of $U(\mathfrak{g})$, and restrict the $(A,B)$ to lie in the set
 $\{(-N,N): N \in \ZZ_{> 0}\}$.  Since the poles of $a_w(s,\eta)$ can only occur at points in $\tfrac{1}{2}\ZZ$ (see \cite[(3.4)]{Getz:Liu:BK}), we can similarly restrict our attention to a countable family of $p_{w}$.  Finally we can restrict attention to a countable family of $\Omega$ by simply choosing a countable family of compact subsets of $X^{\circ}(F)$ whose union is $X^{\circ}(F)$.
 
  The countable family of seminorms described above is separating by Mellin inversion (see \cite[Lemma 4.3]{Getz:Liu:BK}).  It follows that $\mathcal{S}(X(F))$ is Hausdorff and metrizable.  It is also clear that it is complete.  
\end{proof}

\noindent We point out that when $F$ is non-Archimedean we do not endow $\mathcal{S}(X(F))$ with a topology.

  Recall that we have a left action \eqref{act} of $M^{\mathrm{ab}} \times \mathrm{GSp}_{2n}$ on $X^{\circ}.$ 
  This yields an action on functions: for a function $f$ on $X^{\circ}(F)$ and $(m,g,x)\in M^{\mathrm{ab}}(F)\times \mathrm{GSp}_{2n}(F) \times X^{\circ}(F) $
 \begin{align}\label{LRaction}
     L(m)R(g)f(x):=f\left(m^{-1}\begin{psmatrix} I_n & \\ & \nu(g)^{-1} I_n \end{psmatrix}xg\right).
 \end{align}
 Using the formula for $\mathcal{F}_{X}$ from \cite[Corollary 6.10]{Getz:Hsu:Leslie} one deduces the following lemma:

\begin{lem} \label{lem:BKF} If $(m,g) \in M^{\mathrm{ab}}(F) \times \mathrm{GSp}_{2n}(F)$ and $f \in \mathcal{S}(X(F))$, then $L(m)R(g)f \in \mathcal{S}(X(F))$.  Moreover,
\begin{align*}
\mathcal{F}_X(L(m)R(g)f)=|\nu(g)|^{n(n+1)/2}\delta^{-1}_P(m)L(m^{-1})R(\nu(g)^{-1}g)\mathcal{F}_X(f).
\end{align*} \qed
\end{lem}

\noindent One checks the following lemma:
\begin{lem} \label{lem:smooth:act}
The action of $M^{\mathrm{ab}}(F) \times \mathrm{GSp}_{2n}(F)$ on $\mathcal{S}(X(F))$ is smooth.  When $F$ is Archimedean, it is continuous with respect to the Fr\'echet topology on $\mathcal{S}(X(F))$. \qed
\end{lem}

We observe that the inclusions (see \S \ref{ssec:containment})
$$
C_c^\infty(X^{\circ}(F)) \lto \mathcal{S}(X^{\circ}(F)) \lto \mathcal{S}(X(F))
$$
are continuous in the Archimedean case, where we give $C_c^\infty(X^{\circ}(F))$ the usual topology for compactly supported smooth functions on a real manifold and $\mathcal{S}(X^{\circ}(F))$ the topology explained in \S \ref{sec:gen:Schwartz}.  

It is useful to explicitly state and prove a refinement of \cite[Lemmas 5.1 and 5.7]{Getz:Liu:BK}.
The group $\mathrm{Sp}_{2n}$ acts on $\wedge^{n} \GG_a^{2n}$ via its action on $\GG_a^{2n}$. Let $K <\mathrm{Sp}_{2n}(F)$ be  a maximal compact subgroup. We assume $K=\mathrm{Sp}_{2n}(\OO)$ when $F$ is non-Archimedean.  Here $\OO$ is the ring of integers of $F.$ For $F$ Archimedean, choose a $K$-invariant bilinear form $(\cdot,\cdot)$ on $\wedge^n F^{2n}$  and set $|x|:=(x,x)^{[F:\RR]/2}$.  For $F$ non-Archimedean, the standard basis of $F^{2n}$ induces a canonical basis on $\wedge^n F^{2n}$ (given by wedge products of the standard basis of $F^{2n}$ in increasing order).  Define the norm $|x|$ on $\wedge^n F^{2n}$ to be the  maximum norm with respect to the induced basis. We claim that this norm is invariant under $\mathrm{Sp}_{2n}(\OO).$   Indeed, for $x \in \wedge^{n} F^{2n}-\{0\}$ we have $|x|=q^{-k}$ where $k$ is the unique integer such that $\varpi^{-k}x \in \wedge^{n}\OO^{2n}-\varpi \wedge^{n}\OO^{2n}.$  It follows that $|x|$ is invariant under $\GL(\wedge^n\OO^{2n}),$ and the action of $\mathrm{Sp}_{2n}(\OO)$ on $\wedge^{n} F^{2n}$ is the restriction of the action of $\GL(\wedge^n\OO^{2n})$ on $\wedge^{n}F^{2n}.$

In either the Archimedean or non-Archimedean case, we set
\begin{align}\label{plucknorm}
|g|:=|\mathrm{Pl}(g)|
\end{align}
where $\mathrm{Pl}:X^{\circ} \to \wedge^n \GG_a^{2n}$ is the Pl\"ucker embedding.  Note that writing $g=mk$ where $(m,k)\in M^{\mathrm{ab}}(F)\times K$, one has $|g|=|mk|=|\omega(m)|^{-1}$.

\begin{lem} \label{lem:seminorm}
Let $0\leq \beta<\tfrac{1}{2}$.  If $F$ is non-Archimedean, then any $f \in \mathcal{S}(X(F))$ satisfies
$$
|f(g)| \ll_{f,\beta} |g|^{-\tfrac{n+1}{2}+\beta}.
$$
Moreover $f(g)=0$ for $|g|$ sufficiently large in a sense depending on $f.$
If $F$ is Archimedean, then for each $N \in \ZZ_{\geq 0}$ and $D \in U(\mathfrak{g})$ there is a continuous seminorm $\nu_{D,N,\beta}$ on $\mathcal{S}(X(F))$ such that for $f \in \mathcal{S}(X(F))$ one has
$$
|D.f(g)| \leq \nu_{D,N,\beta}(f)|g|^{-N-\tfrac{n+1}{2}+\beta}.
$$
\end{lem}

\begin{proof}
Assume first that $F$ is non-Archimedean.  When $\beta=0$ the lemma is just \cite[Lemma 5.1]{Getz:Liu:BK} and by inspecting the proof one obtains the refined estimate stated in the current lemma.

  For the Archimedean assertion, using Mellin inversion \cite[Lemma 4.3]{Getz:Liu:BK}, we write
\begin{align} \label{contour}
(\omega\bar{\omega}(m))^{-N} D.f(mk)=\delta^{1/2}_P(m)\sum_{\eta \in \widehat{K}_{\GG_m}} \int_{i\RR+\sigma}(D.f)_{\eta_{s+2N[F:\RR]^{-1}}}(k)\eta_s(\omega(m))\frac{ds}{4\pi^{[F:\RR]} i}
\end{align}
for $\sigma$ sufficiently large.
The factor $a_{\mathrm{Id}}(s,\chi)$ is holomorphic in the half plane $\mathrm{Re}(s)>-\tfrac{1}{2}$ and hence so is $a_{\mathrm{Id}}(s+\tfrac{2N}{[F:\RR]},\chi)$.
Thus using the fact that the seminorms \eqref{seminorm} are finite, we can shift the contour to $\sigma=-\beta$ in \eqref{contour} to see that it is bounded by
\begin{align*}
\delta^{1/2}_P(m)&\left|\sum_{\eta \in \widehat{K}_{\GG_m}} \int_{i\RR-\beta}(D.f)_{\eta_{s+2N[F:\RR]^{-1}}}(k)\eta_{s}(\omega(m))\frac{ds}{4\pi^{[F:\RR]} i}\right|\\
&\leq \frac{\delta^{1/2}_P(m)|\omega(m)|^{-\beta}}{4\pi^{[F:\RR]}}2\left(|f|_{A,B,\mathrm{Id},1,K,D} +|f|_{A,B,\mathrm{Id},s^2,K,D}\right)
\end{align*}
where $A:=-\beta-\varepsilon+\tfrac{2N}{[F:\RR]}$ and $B:=-\beta+\varepsilon+\tfrac{2N}{[F:\RR]}$ for some $\varepsilon<\tfrac{1}{2}-\beta$.
Since $|mk|=|\omega(m)|^{-1}$ and $\delta_P(m)=|m|^{-(n+1)}$, we deduce the lemma in the Archimedean case. 
\end{proof}

To prove Proposition \ref{prop:topology} below, we require a more precise version of \cite[Lemma 3.3]{Getz:Liu:BK}.   
Assume for the moment that $F$ is Archimedean and let
\begin{align} \label{mu}
\mu(z):=\frac{z}{(z\bar{z})^{1/2}}
\end{align}
where in the denominator we mean the positive square root. Then any character of $F^\times$ can be written uniquely in the form
$\chi=|\cdot|^{it}\mu^{\alpha}$ where $t \in \RR$, $\alpha \in \{0,1\}$ when $F$ is real, and $\alpha \in \ZZ$ when $F$ is complex. 
\begin{lem} \label{lem:33}
Assume $F$ is Archimedean.  Let $A<B$ be real numbers, and for $w \in \{\mathrm{Id},w_0\}$ let $p_w,p_w' \in \CC[x]$ be polynomials such that $p_w(s)a_w(s,\mu^{\alpha})$ and $p_w'(s)a_w(-s,\overline{\mu^{\alpha}})$ are holomorphic and nonvanishing for all $\alpha$ as above and all $s \in V_{A,B}$. Then the quotients 
$$
\frac{p_{\mathrm{Id}}'(s)a_{\mathrm{Id}}(-s,\overline{\mu^{\alpha}})}{p_{w_0}(s)a_{w_0}(s,\mu^{\alpha})}, \quad \frac{p_{w_0}(s)a_{w_0}(s,\mu^\alpha)}{p_{\mathrm{Id}}'(s)a_{\mathrm{Id}}(-s,\overline{\mu^{\alpha}})},\quad \frac{p_{\mathrm{Id}}(s)a_{\mathrm{Id}}(s,\mu^{\alpha})}{p_{w_0}'(s)a_{w_0}(-s,\overline{\mu^{\alpha}})}, \quad \frac{p_{w_0}'(s)a_{w_0}(-s,\overline{\mu^{\alpha}})}{p_{\mathrm{Id}}(s)a_{\mathrm{Id}}(s,\mu^{\alpha})}
$$
are bounded in $V_{A,B}$ by polynomials in $s.$  \qed
\end{lem}

\noindent The key difference between this lemma and \cite[Lemma 3.3]{Getz:Liu:BK} is the uniformity in $\alpha$ of the bound.  However, the proof of \cite[Lemma 3.3]{Getz:Liu:BK} actually yields the stronger assertion of Lemma \ref{lem:33}.

Let
$$
\mathcal{S}(X(F),K) < \mathcal{S}(X(F))
$$
be the subspace of $K$-finite vectors. It is dense \cite[\S 4.4.3.1]{Warner:Semisimple:I}.    We prove the following lemma for use in the proof of Theorem \ref{thm:BK:PS}:

\begin{prop} \label{prop:topology}  If $F$ is Archimedean then the Fourier transform 
$$
\mathcal{F}_X:\mathcal{S}(X(F)) \lto \mathcal{S}(X(F))
$$
is continuous.
\end{prop}
\begin{proof}
To prove the proposition it suffices to show that the seminorms
\begin{align} \label{seminorm2}
f \longmapsto |\mathcal{F}_X(f)|_{A,B,w,p_w,\Omega,D}=\sum_{\eta \in \widehat{K}_{\GG_m}}\mathrm{sup}_{g \in \Omega}|M_w(D.\mathcal{F}_X(f))_{\eta_s}(g)|_{A,B,p_w}
\end{align}
are continuous on $\mathcal{S}(X(F)).$ 
Let $w \in \{\mathrm{Id},w_0\}$.  
Assume $
D=D_1 \otimes D_2$
where $D_1 \in U(\mathrm{Lie}(M^{\mathrm{ab}}(F))$ and  $D_2 \in U(\mathfrak{sp}_{2n}(F))$. By \cite[Lemma 5.9]{Getz:Liu:BK} and Theorem \ref{thm:44},
\begin{align} \label{this}
|M_w(D.\mathcal{F}_{X}(f))_{\mu^{\alpha}_s}(k)|_{A,B,p_w}=|M_w\mathcal{F}_{X}(D_2.f)_{\mu^{\alpha}_s}(g)|_{A,B,p_{\alpha,w}}
=|M_wM_{w_0}^*(D_2.f_{(\overline{\mu^{\alpha}})_{-s}})(g)|_{A,B,
p_{\alpha,w}}.
\end{align}
Here $p_{\alpha,w}(s)$ is a polynomial function of $s$ and $\alpha$ divisible by $p_{w}(s)$ that depends on $D_1.$
In \eqref{this} we used the fact that the Mellin transform commutes with $D_2$ since $D_2$ is induced by the action of $\mathrm{Sp}_{2n}(F)$ on $X^\circ(F)$ on the right, whereas the Mellin transform is given by integrating over the action of $M^{\mathrm{ab}}(F)$ on $X^{\circ}(F)$ on the left.

  Using the argument of \cite[Lemma 3.4]{Getz:Liu:BK}, but with Lemma \ref{lem:33} replacing \cite[Lemma 3.3]{Getz:Liu:BK}, the quantity \eqref{this} is bounded by a constant depending on $A,B$
  times 
  \begin{align*}
  \max(1,|\alpha|)^N|M_{w'}(D_2f_{(\bar{\mu^{\alpha}})_{-s}})(g)|_{A,B,p_{w'}}
  \end{align*}
  for some $N$
  and an appropriate $p_{w'}$ independent of $\alpha$, where 
  \begin{align*}
  w'&=\begin{cases} w_0 &\textrm{ if }w=\mathrm{Id},\\
  \mathrm{Id}&\textrm{ if }w=w_0.
  \end{cases}
  \end{align*}
  This in turn is dominated by 
  \begin{align*}
      |M_{w'}(D'D_2f_{(\bar{\mu^{\alpha}})_{-s}})(g)|_{A,B,p_{w'}}
  \end{align*}
  for an appropriate differential operator $D'$ (see \cite[Lemma 5.9]{Getz:Liu:BK}).  Thus the seminorm \eqref{seminorm2} is dominated by 
  \begin{align*}
      f \longmapsto \sum_{\eta \in \widehat{K}_{\GG_m}}\mathrm{sup}_{g \in \Omega}|M_{w'}(D'D_2f_{(\bar{\mu^{\alpha}})_{-s}})(g)|_{A,B,p_{w'}}.
  \end{align*}
  But this is a continuous seminorm on $\mathcal{S}(X(F))$ by definition.
\end{proof}

Let
\begin{align} \label{1:def}
m(x):=\left(\begin{smallmatrix} x^{-1} & & & \\ & I_{n-1} & & \\ & & x & \\ & & & I_{n-1} \end{smallmatrix}\right).
\end{align}
Assume now that $F$ is non-Archimedean.  Then $|m(x)|=|x|.$
By the Iwasawa decomposition, a $\CC$-vector space basis for $C_c^\infty(X^{\circ}(F))^{\mathrm{Sp}_{2n}(\OO)}$ is given by the functions
\begin{align} \label{1k}
\one_{k}:=\one_{[P,P](F)m(\varpi^{k}) \Sp_{2n}(\OO)}
\end{align}
for $k \in \ZZ$.
The space $\mathcal{S}(X(F))^{\Sp_{2n}(\OO)}$ contains $C_c^\infty(X^{\circ}(F))^{\Sp_{2n}(\OO)}$ but it is larger.  It contains, for example, 
the basic function 
\begin{align} \label{bX}
b_{X}:=\sum_{(j_1,\dots,j_{\lfloor n/2 \rfloor},k) \in \ZZ^{\lfloor n/2 \rfloor+1}_{\geq 0}} q^{2j_1+4j_2+\dots+2\lfloor n/2 \rfloor j_{\lfloor n/2 \rfloor}}\one_{k+2j_1+\dots+2j_{\lfloor n/2\rfloor}}.
\end{align}
One has $\mathcal{F}_X(b_X)=b_X$ \cite[Lemma 5.4]{Getz:Liu:BK} provided that $\psi$ is unramified.

It will be convenient to isolate another family of functions in this space.  
For $c \in \ZZ$, let
\begin{align*}
\one_{\geq c}:=\sum_{k \geq c} \one_k=\one_{\varpi^c X(\OO)}.
\end{align*}
\begin{lem} \label{lem:onec:in}
One has $\one_{ \geq c} \in \mathcal{S}(X(F))^{\Sp_{2n}(\OO)}$.
\end{lem}

\begin{proof}
Recall the action $L$ defined in \eqref{LRaction}. One has $L(m(\varpi)^{c})\one_{\geq 0}=\one_{\geq c}$, so 
by Lemma \ref{lem:BKF} it suffices to show $\one_{\geq 0} \in \mathcal{S}(X(F))^{\mathrm{Sp}_{2n}(\OO)}$.  Since
$$
\Big(\prod_{j=1}^{\lfloor n/2 \rfloor}(1-q^{2j}L(m(\varpi)^2))\Big)b_X=\one_{\geq 0},
$$
we can apply Lemma \ref{lem:BKF} again to deduce the result.
\end{proof}

The usual Schwartz space of $F$ is dense in $L^2(F,dx)$ and the Fourier transform extends to an isometry of $L^2(F,dx)$.  We now prove analogues of these statements in the current setting.  We choose a positive right $\mathrm{Sp}_{2n}(F)$-invariant Radon measure on $X^{\circ}(F)$ (it is unique up to scaling).  
Since $X^{\circ}(F) \subset X(F)$ is open and dense we extend by zero to obtain a measure on $X(F)$ and we can speak of $L^2(X(F))$.  
\begin{prop} \label{prop:cont}
One has $\mathcal{S}(X(F)) < L^2(X(F))$.  The Fourier transform $\mathcal{F}_X$ extends to an isometry of $L^2(X(F))$. 
For $f,f_1,f_2 \in L^2(X(F))$ one has
\begin{align} \label{inverse}
\mathcal{F}_X \circ \mathcal{F}_X& = L(m(-1)^{n+1})\,,\\
 \label{complex:conj}
    \overline{\mathcal{F}_X(f)}&=\mathcal{F}_X(L(m(-1))\bar{f})\,,\\
\label{Planch}
\int_{X(F)} \overline{f_1(x)}\mathcal{F}_X(f_2)(x) dx&=\int_{X(F)}\overline{\mathcal{F}_X(f_1)(x)}(L(m(-1)^{n+1})f_2)(x)dx\,.
\end{align}
\end{prop}

Before giving the proof we recall two lemmas.  The first is an identity that was stated with a typo in \cite[(1.2.3)]{Ikeda:poles:triple}:

\begin{lem} \label{lem:n1} 
Assume that $\chi:F^\times \to \CC^\times$ is a  character and that $n=1$.
The operator
\begin{align*}
    M_{w_0}^* \circ M_{w_0}^*:I(\chi_s) \lto I(\chi_s)
\end{align*}
    is the identity.  \qed
\end{lem}

\noindent This is well-known and may be obtained via a standard argument.

\begin{lem} \label{lem:correction}
 For any $n$ and any character $\chi:F^\times \to \CC^\times$, the operator 
 $$
 M_{w_0}^*\circ M_{w_0}^*:I(\chi_s) \lto I(\chi_s)
 $$
 is multiplication by $\chi(-1)^{n+1}$.
\end{lem}

\begin{proof}
This was stated incorrectly in  \cite[Lemma 1.1]{Ikeda:poles:triple}.  The source of the error is the typo in \cite[(1.2.3)]{Ikeda:poles:triple}.  Upon correcting the typo using Lemma \ref{lem:n1}, the same argument proves the current lemma.
\end{proof}

\begin{rem}
The typos in \cite{Ikeda:poles:triple} corrected in Lemma \ref{lem:n1} and \ref{lem:correction} do not affect the main results of \cite{Ikeda:poles:triple} or their proofs.  Moreover they do not affect \cite{Getz:Liu:BK}, which makes use of results in  \cite{Ikeda:poles:triple}, except for the statement of \cite[Lemma 4.6]{Getz:Liu:BK}.  The correct statement is in Proposition \ref{prop:cont} above.  
\end{rem}

\begin{proof}[Proof of Proposition \ref{prop:cont}]
The inclusion $\mathcal{S}(X(F)) <L^2(X(F))$ is an easy consequence of the Iwasawa decomposition and Lemma \ref{lem:seminorm}.  

For $f \in \mathcal{S}(X(F))$, assertion \eqref{inverse} is a consequence of  Theorem \ref{thm:44} and Lemma \ref{lem:correction}.
Taking the complex conjugate of the identity of Theorem \ref{thm:44}, for $\sigma \geq 0$ one has
\begin{align*}
\overline{\mathcal{F}_X(f)}_{\chi_{\sigma+it}}&=\overline{\mathcal{F}_X(f)_{\bar{\chi}_{\sigma-it}}}=\overline{M_{w_0}^*(f_{\chi_{-\sigma+it}})}\\
&=\overline{\gamma(-\sigma+it-\tfrac{n-1}{2},\chi,\psi)}\prod_{r=1}^{\lfloor n/2\rfloor}\overline{\gamma(2(-\sigma+it)-n+2r,\chi^2,\psi)}\overline{M_{w_0}(f_{\chi_{-\sigma+it}})}\\
&=\gamma(-\sigma-it-\tfrac{n-1}{2},\bar{\chi},\overline{\psi})\prod_{r=1}^{\lfloor n/2\rfloor}\gamma(2(-\sigma-it)-n+2r,\overline{\chi}^2,\overline{\psi})M_{w_0}(\overline{f}_{\bar{\chi}_{-\sigma-it}})\\
&=\chi(-1)M_{w_0}^*(\overline{f}_{\bar{\chi}_{-\sigma-it}})\,.
\end{align*}
Thus by Theorem \ref{thm:44}, we deduce assertion \eqref{complex:conj}.

Let $P^{\mathrm{op}}<\mathrm{Sp}_{2n}$ be the parabolic subgroup opposite to $P$ with respect to $M.$  In \cite{BK:normalized} Braverman and Kazhdan defined an isometry
$$
\mathcal{F}_{P|P^{\mathrm{op}}}:L^2(X(F)) \lto L^2([P^{\mathrm{op}},P^{\mathrm{op}}] \backslash \mathrm{Sp}_{2n}(F)).
$$ 
One has $\mathcal{F}_X =\iota_{w_0} \circ \mathcal{F}_{P|P^{\mathrm{op}}}$ by \cite[(5.24)]{Getz:Hsu:Leslie},
where $\iota_{w_0}$ is the isometry
\begin{align*}
\iota_{w_0}:L^2([P^{\mathrm{op}},P^{\mathrm{op}}]  \backslash \mathrm{Sp}_{2n}(F)) &\tilde{\lto} L^2(X(F))\\
f &\longmapsto \left(x \mapsto f(w_0^{-1}x)\right).
\end{align*}
It follows that $\mathcal{F}_X$ is an isometry.  The Plancherel formula \eqref{Planch} follows  from   \eqref{inverse}, the unitarity of $\mathcal{F}_X,$ and a standard argument using a polarization identity.
\end{proof}

\subsection{The summation formula}

We now revert to the global setting. Let $F$ be a number field. Recall that $\mathcal{S}(X(\mathbb{A}_F))$ is defined in \S \ref{sec:gen:Schwartz} as the restricted tensor product of $\mathcal{S}(X(F_v))$ with respect to the basic functions $b_{X,v}$.  Let $\psi:F\backslash \A_F \to \CC^\times$ be a nontrivial character. We have $\mathcal{F}_{X,\psi_v}(b_{X,v})=b_{X,v}$ if $F_v$ is non-Archimedean and $\psi_v$ is unramified. Thus we have a global Fourier transform 
$$
\mathcal{F}_X:=\mathcal{F}_{X,\psi}:=\otimes_{v}\mathcal{F}_{X,\psi_v}:\mathcal{S}(X(\A_F)) \lto \mathcal{S}(X(\A_F)).
$$

By Proposition \ref{prop:schwartz:inside} below $\mathcal{S}(X^{\circ}(F_v)) < \mathcal{S}(X(F_v))$ for all places $v.$ 
\begin{thm} \label{thm:BK:PS} Assume that for some finite places $v_1,v_2$ (not necessarily distinct) one has
$f=f_{v_1}f^{v_1}$ and $\mathcal{F}_X(f)=\mathcal{F}_X(f_{v_2})\mathcal{F}_X(f^{v_2})$ with $f_{v_1} \in C_c^\infty(X^{\circ}(F_{v_1}))$ and $\mathcal{F}_X(f_{v_2}) \in C_c^\infty(X^{\circ}(F_{v_2})).$
Then 
\begin{align*}
\sum_{\gamma \in X^\circ(F)}f(\gamma)=\sum_{\gamma \in X^\circ (F)} \mathcal{F}_X(f)(\gamma).
\end{align*}
\end{thm}

\begin{proof} 
We may assume $f=f_\infty f^\infty$ with $f_\infty \in \mathcal{S}(X(F_\infty))$ and $f^\infty \in \mathcal{S}(X(\A_F^\infty))$.  
Let $K_\infty < \Sp_{2n}(F_\infty)$ be a maximal compact subgroup and let $\mathcal{S}(X(F_\infty),K_\infty) <\mathcal{S}(X(F_\infty))$ be the space of $K_\infty$-finite functions. 
Assume first that $f_\infty \in \mathcal{S}(X(F_\infty),K_\infty)$.  Then the stated identity follows from \cite[Theorem 1.1]{Getz:Liu:BK} and \cite[Theorem 10.1]{Getz:Liu:Triple}.

We now argue by continuity to deduce the identity in general.  Consider the linear form
\begin{align} \label{lin:form} \begin{split}
\mathcal{S}(X(F_\infty)) &\lto \CC\\
 f_\infty&\longmapsto \sum_{\gamma\in X^\circ (F)} f_\infty(\gamma)f^\infty(\gamma)- \sum_{\gamma\in X^\circ(F)} \mathcal{F}_X(f_\infty)(\gamma)\mathcal{F}_X(f^\infty)(\gamma). \end{split}
\end{align}
The Fourier transform is continuous by Proposition \ref{prop:topology}.  Thus following the proof of \cite[Lemma 6.4]{Getz:Liu:BK}, by replacing Lemma 5.7 in loc.~ cit. with Lemma \ref{lem:seminorm}, we see that the sums defining \eqref{lin:form} are absolutely convergent and \eqref{lin:form} is continuous.  It vanishes on the dense subspace $\mathcal{S}(X(F_\infty),K_\infty) <\mathcal{S}(X(F_\infty))$ and hence is identically zero.
\end{proof}

We remark that Theorem \ref{thm:BK:PS} was already proved in \cite{BK:normalized}, but with a different definition of the Schwartz space.  At the non-Archimedean places the two definitions yield the same space of functions \cite{CH:nonarch}.  At the Archimedean places this is less clear.  In any case, it is easier to just prove the theorem directly than to rigorously check the compatibility of the two definitions.

\subsection{Containment of Schwartz spaces} \label{ssec:containment}

In this subsection we prove the following proposition:
\begin{prop} \label{prop:schwartz:inside}
One has $\mathcal{S}(X^{\circ}(F))<\mathcal{S}(X(F))$.  In the Archimedean case the inclusion is continuous.  
\end{prop}
\noindent 
In the non-Archimedean case this is \cite[Proposition 4.7]{Getz:Liu:BK}.  

\begin{rem}In the Archimedean case the weaker statement that $K_\infty$-finite functions in $C_c^\infty(X^{\circ}(F))$ are contained in $\mathcal{S}(X(F))$ was asserted in \cite[Proposition 4.7]{Getz:Liu:BK}.  This is true, but the proof is incomplete.  It relies on \cite[Lemma A.2]{Getz:Liu:BK}, which is false.  Happily, this does not affect the rest of the results in \cite{Getz:Liu:BK} because the false assertion in \cite[Lemma A.2]{Getz:Liu:BK} is not used elsewhere in the paper.  
\end{rem}

\quash{
Let $T \leq \mathrm{Sp}_{2n}$ be the maximal torus of diagonal matrices, let $W_{\mathrm{Sp}_{2n}}$ be the Weyl group of $T$ in $\mathrm{Sp}_{2n}$ and let $W_M$ denote the Weyl group of $T$ in $M$.   Denote by $B \leq P$ the unique Borel subgroup consisting of matrices whose upper left $n \times n$ block is upper triangular.  We let $\Phi_{\mathrm{Sp}_{2n}}$ be the set of roots of $T$ in $\mathrm{Sp}_{2n}$ and we define positive roots using $B.$  
We let
 $W_n$ be the complete set of representatives for $W_{\mathrm{Sp}_{2n}} / W_M$ obtained by choosing the unique element of minimal length in each coset as follows (see \cite[\S 1]{Ikeda:poles:triple} and \cite[Part A, Lemma 5.1]{GPSR:LNM}). For each subset $\iota=\{i_1, i_2, \ldots, i_k\}$ of $\{1,2,\ldots,n\}$
let 
\begin{align} \label{J}
J := \{j_1, j_2 \ldots, j_{n-k}\} = \{1,2,\ldots, n\} - \iota\,,
\end{align}
 where $i_1<\dots<i_k$ and $j_1 <\dots <j_{n-k}$.
Define an element $w_\iota$ of $W_{\mathrm{Sp}_{2n}}$ by 
\begin{align*}
\begin{matrix}
& t_1 \mapsto t_{j_1}, &\ldots, &t_{n-k} \mapsto t_{j_{n-k}}\,,\\
& t_{n-k+1} \mapsto t^{-1}_{i_k}, & \ldots,& t_n \mapsto t^{-1}_{i_1}\,,
\end{matrix}
\end{align*}
where 
\begin{align*}\left(\begin{smallmatrix}
t_1 &&&&&\\
&\ddots &&&&\\
&&t_n &&&\\
&&&t_1^{-1}&&\\
&&&&\ddots &\\
&&&&&t_n^{-1}
\end{smallmatrix} \right)\in T(F)\,.
\end{align*}
In particular $w_0:=w_{\{1,\dots,n\}}$ is the long Weyl element.  For each $0 \leq k \leq n-1$  let $\iota_k=\{k+1,\dots,n\}$ and let $\iota_{n}=\emptyset$.  }
Let 
$$
w_k:=\begin{psmatrix} I_k & 0 &0 & 0 \\  0 & 0 &0 & - \beta_{n-k}\\ 0 & 0 & I_k & 0     \\0 & \beta_{n-k} &0 & 0\end{psmatrix}
$$
where $\beta_{n-k} \in \GL_{n-k}(\ZZ)$ is the antidiagonal matrix.

\begin{lem} \label{lem:cover} Let $C=[P,P] \backslash Pw_0Pw_0^{-1} \subset X^\circ.$  Then $Cg$ is open in $X^\circ$ for all $g \in \mathrm{Sp}_{2n}(F)$ and 
\begin{align*} 
\bigcup_{0 \leq k \leq n} \bigcup_{\tau \in \mathfrak{S}_n} C(F)w_k\begin{psmatrix}\tau & \\ & \tau \end{psmatrix} =X^\circ(F).
\end{align*}
Here the inner union is over the group of permutation matrices in $\GL_n(\ZZ).$
\end{lem}

\begin{proof}
It is well known that the big Bruhat cell $ Pw_0P$ is open in $\mathrm{Sp}_{2n}$, hence the same is true of $Cg$ in  $X^\circ$ for all $g \in \mathrm{Sp}_{2n}(F).$ 
We have an isomorphism $\phi:\GL_n \times \mathrm{Sym}^{2}(\GG_a^n) \to N\backslash Pw_0Pw_0^{-1}$ given on points in an $F$-algebra $R$ by 
\begin{align*}  \begin{split}
\phi:\GL_n(R) \times \mathrm{Sym}^2(R^n) &\lto (N\backslash Pw_0Pw_0^{-1})(R)\\
(A,Z) &\longmapsto \begin{psmatrix} A^{-t} & \\ & A \end{psmatrix}w_0 \begin{psmatrix} I_n & -\beta_n Z\beta_n \\ & I_n \end{psmatrix} w_0^{-1}=\begin{psmatrix} * & *\\  AZ & A \end{psmatrix}. \end{split}
\end{align*}
It follows that $C(F)$ may be characterized as the subset of $X^\circ(F)$ consisting of all classes of the form
\begin{align} \label{class}
\begin{psmatrix} * & *\\  B & A \end{psmatrix}
\end{align}
where $\begin{psmatrix} B & A \end{psmatrix} \in M_{n,2n}(F)$ and $\det A \neq 0.$  This in turn implies that
$$
C(F)w_{k}
$$
is the subset of $X^{\circ}(F)$ consisting of classes of the form \eqref{class} where the $n \times n$ matrix formed by the columns corresponding to
$\{i:k+1 \leq i \leq n+k\}$ is invertible.

To complete the proof, we claim that 
after multiplying by an element of $M(F)$ on the left (i.e., performing Gauss-Jordan elimination on $A$) and $\begin{psmatrix} \tau & \\ & \tau \end{psmatrix}$ with $\tau \in \mathfrak{S}_n$ on the right, a general class \eqref{class} lies in $C(F)w_k$
for some $k.$ Indeed, after multiplying on the left and right as just explained, we may assume  
$$
\begin{psmatrix} B & A\end{psmatrix} =
\begin{psmatrix}   u & D & 0 & 0\\ 
* & *& I_{k} &v \\
\end{psmatrix}
$$
for some $0 \leq k \leq n$. Here $u,v^t\in M_{n-k,k}(F)$ and  $D \in M_{n-k,n-k}(F).$  We claim $D$ is invertible so that $\begin{psmatrix} B & A\end{psmatrix}$ lies in $C(F)w_{k}$. Since the rows of $\begin{psmatrix} B & A \end{psmatrix}$ span a Lagrangian subspace, the $(i,j)$th entry of $u$ is $-1$ times the dot product of the $i$th row of $D$ and the $j$th row of $v$. In other words, $u=-Dv^t.$ Thus if $D$ is singular, then the rows of $\begin{psmatrix}
    u & D
\end{psmatrix}$ are linearly dependent, which is a contradiction.

\end{proof}

\begin{proof}[Proof of Proposition \ref{prop:schwartz:inside}]  We can and do assume $F$ is Archimedean.
If $f \in \mathcal{S}(X^{\circ}(F))$, it is easy to see that the integral defining $f_{\chi_s}$ is absolutely convergent for all $\chi:F^\times \to \CC^\times$ and $s \in \CC$.  Thus $f_{\chi_s}$ is a good section by \cite[Lemma 1.3]{Ikeda:poles:triple}.  We have to verify that for all real numbers $A<B$, $w \in \{\mathrm{Id},w_0\}$, $D \in U(\mathfrak{g}),$ any polynomials $p_{w} \in \CC[s]$ such that 
\begin{align} \label{pw}
    p_{w}(s)a_w(s,\eta) \textrm{ has no poles for all }(s,\eta) \in V_{A,B} \times \widehat{K}_{\GG_m}
\end{align}    
and compact subsets $\Omega \subset X^{\circ}(F)$, one has
$|f|_{A,B,w,p_w,\Omega,D}<\infty$
and $|f|_{A,B,w,p_w,\Omega,D}$ is continuous with respect to the topology on $\mathcal{S}(X^{\circ}(F))$.  
Since $U(\mathfrak{g})$ acts continuously on $\mathcal{S}(X^{\circ}(F))$, it suffices to verify this for $D=\mathrm{Id}$.  Note that we do not require $p_w(s)$ to depend on $\eta.$  In fact by the explicit description of Archimedean local $L$-functions we see that $a_w(s,\eta)$ can only have a pole in $V_{A,B}$ for finitely many $\eta.$

We start by reducing to an estimate involving a single $\eta$.  
Let $D_1$ (and $\bar{D}_1$ when $F$ is complex) be the generators of $U(\mathrm{Lie}(M^{\mathrm{ab}}(F)))$ given in \cite[(4.2) and (4.3)]{Getz:Liu:BK}, respectively. 
Every element of $\widehat{K}_{\GG_m}$ is in the equivalence class of $\mu^{\alpha}$ for $\alpha=0, 1$ when $F$ is real and $\alpha \in \ZZ$ when $F$ is complex.  Here $\mu$ is defined as in \eqref{mu}.  If $F$ is complex,
\begin{align*}
   & |M_w(D_1^N\overline{D}_1^{N'}.f)_{\mu^{\alpha }_s}(g)|_{A,B,p_{w}}\\&=\left|\left(\frac{\alpha}{2}+it+s+\tfrac{n+1}{2} \right)^N\left(-\frac{\alpha}{2}+it+s+\tfrac{n+1}{2} \right)^{N'}M_w(f)_{\mu^\alpha_s}(g)\right|_{A,B,p_{w}}
\end{align*}
by \cite[Lemma 5.9]{Getz:Liu:BK}.  This provides us with an estimate on $M_w(f)_{\mu^\alpha_s}(g)$ as a function of $\alpha$.  Using this estimate, we see that to prove $|f|_{A,B,w,p_w,\Omega,1}$ is finite for all $f \in \mathcal{S}(X^{\circ}(F))$, it suffices to prove that for each $p_w$ satisfying \eqref{pw}  there is a continuous seminorm $\nu$ on $\mathcal{S}(X^{\circ}(F))$  such that 
$$
\mathrm{sup}_{g \in \Omega}|M_wf_{\mu^\alpha_s}(g)|_{A,B,p_{w}}\leq \nu(f)
$$
for all $f \in \mathcal{S}(X^{\circ}(F))$ and $\alpha$.  Here and below the seminorm $\nu$ is allowed to depend on $A,B,w,p_w,\Omega$.  
In fact, it is enough to show that there is a continuous seminorm $\nu$ on $\mathcal{S}(X^{\circ}(F))$  that 
\begin{align} \label{seminorm:stuff}
    |M_wf_{\mu^\alpha_s}(w_0)|_{A,B,p_{w}}\leq \nu(f).
\end{align}
Indeed, let $\Omega' \subset \mathrm{Sp}_{2n}(F)$ be a compact set whose projection to $X^{\circ}(F)$ is $\Omega$.  Then assuming we have a seminorm $\nu$ as just described, we have
$$
\mathrm{sup}_{g \in \Omega}|M_wf_{\mu^\alpha_s}(g)|_{A,B,p_{w}}=\mathrm{sup}_{g \in w_0^{-1}\Omega'}|M_w(R(g)f)_{\mu^\alpha_s}(w_0)|_{A,B,p_{w}}\leq \sup_{g \in w_0^{-1}\Omega'}\nu(R(g)f)
$$
and $\sup_{g \in w_0^{-1}\Omega'}\nu(R(g)f)$ is another continuous seminorm.

Since \eqref{seminorm:stuff} is clear for $w=\mathrm{Id}$, we are left with the $w=w_0$ case.  We will roughly follow the strategy of \cite[\S 4]{PSRallisTriple}, but we cannot immediately reduce to functions having support in the big cell as in \cite[Lemma 4.1]{PSRallisTriple}.  Indeed, the proof of loc.~cit.~uses the irreducibility of certain principal series, and the space $\mathcal{S}(X(F))$ is reducible as a representation of $\mathrm{Sp}_6(F).$  

Let $C$ be the image of $P w_0Pw_0^{-1}$ in $X^{\circ}$.  
Choose  a tempered partition of unity subordinate to the cover of $X^{\circ}(F)$ given by Lemma \ref{lem:cover}, that is, choose tempered functions $t_{k,\tau}$ supported in $C(F)w_k\begin{psmatrix}\tau & \\ & \tau \end{psmatrix}$ such that $\sum_{k=0}^n\sum_{\tau \in \mathfrak{S}_n}t_{k,\tau}=1$  and $t_{k,\tau}f \in \mathcal{S}(C(F)w_k\begin{psmatrix} \tau & \\ & \tau \end{psmatrix})$ for all $f \in \mathcal{S}(X^{\circ}(F))$ \cite[Proposition 3.14]{Elazar:Shaviv}. 
Then
$$
|M_{w_0}f_{\mu^\alpha_s}(w_0)|_{A,B,p_w}\le \sum_{k=0}^n\sum_{\tau \in \mathfrak{S}_n}|M_{w_0}(t_{k,\tau}f)_{\mu^\alpha_s}(w_0)|_{A,B,p_{w}}.
$$
Hence we can and do assume that $f$ is supported in $C(F)w_k\begin{psmatrix} \tau & \\ & \tau \end{psmatrix}$ for some fixed $0 \leq k \leq n$ and some $\tau \in \mathfrak{S}_n.$ 
Now 
\begin{align*}
M_{w_0}f_{\mu^\alpha_s}(w_0)&=\int_{N(F)}\int_{M^{\mathrm{ab}}(F)}\delta_P(m)^{1/2}\mu^\alpha_{s}(\omega(m))f(m^{-1}w_0^{-1}nw_0)dmdn\\
&=\int_{\mathrm{Sym}^2(F^n)}\int_{(\SL_n\backslash \GL_n)(F)}
\mu^\alpha_{s+(n+1)/2}(\det A)f\begin{psmatrix} * & * \\ AZ & A \end{psmatrix}dA dZ.
\end{align*}
The notation is a reminder that the image of an element of $\mathrm{Sp}_{2n}(F)$ in $X^{\circ}(F)$ depends only on the bottom $n$ rows of the matrix. 

Write
$
Z=\tau^{-1} \begin{psmatrix} u &x\\ x^t & y \end{psmatrix}\tau$
where $(u,x,y)\in \mathrm{Sym}^2(F^k)\times M_{k \times (n-k)}(F) \times \mathrm{Sym}^2(F^{n-k})$, and let
$$
A'=A'(x,y):=\begin{psmatrix} I_k & x\beta_{n-k}\\ 0 & y\beta_{n-k} \end{psmatrix} 
.$$
Then if $A'$ is invertible, we have
\begin{align*}
\begin{psmatrix} * & *\\  AZ & A \end{psmatrix}&=\begin{psmatrix} (A\tau^{-1})^{-t} & \\ & A\tau^{-1} \end{psmatrix}\begin{psmatrix} *  & * & * & *\\ u & 0 & I_k & x\beta_{n-k}\\x^t & -\beta_{n-k} & 0 & y\beta_{n-k}
\end{psmatrix}w_k\begin{psmatrix} \tau & \\ & \tau \end{psmatrix}\\&=\begin{psmatrix} (A\tau^{-1}A')^{-t} & \\ & A\tau^{-1}A'  \end{psmatrix}\begin{psmatrix} * & * & * & *\\ u-xy^{-1}x^t & xy^{-1}\beta_{n-k} & I_k & 0 \\ \beta_{n-k}y^{-1}x^t & -\beta_{n-k}y^{-1}\beta_{n-k}  & 0 & I_{n-k}\end{psmatrix}w_k\begin{psmatrix} \tau & \\ & \tau \end{psmatrix}.
\end{align*}
Thus 
\begin{align*}
    &M_{w_0}f_{\mu_s^\alpha}(w_0)\\&=\int\mu_{s+(n+1)/2}^\alpha(\det A)\\
    &\times f\left(\begin{psmatrix} (A\tau^{-1}A')^{-t} & \\ & A\tau^{-1}A'  \end{psmatrix}\begin{psmatrix} * & * & * & *\\ u-xy^{-1}x^t & xy^{-1}\beta_{n-k} & I_k & 0 \\ \beta_{n-k}y^{-1}x^t & -\beta_{n-k}y^{-1}\beta_{n-k}  & 0 & I_{n-k}\end{psmatrix}w_k \begin{psmatrix} \tau & \\ & \tau \end{psmatrix}\right) dAdudxdy\\
    &=\int\mu_{s+(n+1)/2}^\alpha(\det A \tau)f\left(\begin{psmatrix} A^{-t} & \\ & A  \end{psmatrix}\begin{psmatrix} * & * & * & *\\ u & x& I_k & 0 \\ x^t & y^{-1}  & 0 & I_{n-k}\end{psmatrix}w_k\begin{psmatrix} \tau & \\ & \tau \end{psmatrix} \right) \frac{|y|^{k}dAdudxdy}{\mu_{s+(n+1)/2}^{\alpha}(\det \beta_{n-k} y) }\\
    &=\int\mu_{s+(n+1)/2}^\alpha(\det A\tau )f\left(\begin{psmatrix} A^{-t} & \\ & A  \end{psmatrix}\begin{psmatrix} * & * & * & *\\ u & x& I_k & 0 \\ x^t & y & 0 & I_{n-k}\end{psmatrix}w_k\begin{psmatrix} \tau & \\ & \tau \end{psmatrix} \right) \mu_{s-(n+1)/2}^{\alpha}(\det \beta_{n-k} y)dAdudxdy,
\end{align*}
where the integrals are over $(\SL_n\backslash\GL_n)(F) \times \mathrm{Sym}^2(F^k) \times M_{k \times (n-k)}(F) \times \mathrm{Sym}^2(F^{n-k})$.  Here we have used that $d(y^{-1})=|\det y|^{-n+k-1}dy$.
  Now consider the differential operator 
\begin{align*}
    \partial:=\det \left( \partial_{ij}\right)
\end{align*}
where $(\partial_{ij})$  is the unique symmetric $(n-k) \times (n-k)$ matrix of partial differential operators satisfying
$$
\partial_{ij}=\begin{cases} \frac{\partial}{\partial z_{ii}} &\textrm{ if }i=j,\\
\frac{1}{2}\frac{\partial}{\partial z_{ij}}&\textrm{ if } i>j. \end{cases}
$$
When $F$ is complex, we view these as holomorphic differential operators.  Then for $y \in \mathrm{Sym}^2(F^{n-k})$ we have 
$$
\partial (\det y)^s=\prod_{i=0}^{n-k-1}\left(s+\frac{i}{2} \right)
(\det y)^{s-1}
$$
(see, e.g.~\cite[Theorem 2.2]{Caracciolo:Sokal:Sportiello}).

Applying integration by parts $m$ times we have
\begin{align*}
   &p_{m,\alpha}(s) \int \mu_{s-(n+1)/2}^\alpha(\det y )\mu_{s+(n+1)/2}^\alpha(\det A)f\left(\begin{psmatrix} A^{-t} & \\ & A  \end{psmatrix}\begin{psmatrix} * & * & * & *\\ u & x & I_k & 0 \\ x & y  & 0 & I_{n-k}\end{psmatrix}w_k \begin{psmatrix} \tau & \\ & \tau \end{psmatrix} \right)dAdudxdy\\
   &=\int \mu_{s-(n+1)/2+m}^\alpha(\det y)\mu_{s+(n+1)/2}^\alpha(\det A)(\partial\overline{\partial})^mf\left(\begin{psmatrix} A^{-t} & \\ & A  \end{psmatrix}\begin{psmatrix} * & * & * & *\\ u & x & I_k & 0 \\ x & y  & 0 & I_{n-k}\end{psmatrix}w_k \begin{psmatrix} \tau & \\ & \tau \end{psmatrix} \right)dAdudxdy
\end{align*}
where $p_{m,\alpha}(s) \in \CC[s]$ has zeros only in $\tfrac{1}{2}\ZZ$. Here by convention $\overline{\partial}$ is the identity operator when $F$ is real, and we are letting $\partial$ and $\overline{\partial}$ act on $f$ viewed as a function of $y \in \mathrm{Sym}^{n-k}(F)$.   
 We observe that the bottom integral converges absolutely for $\mathrm{Re}(s)+m > \tfrac{n+1}{2}$, and thus provides us with a holomorphic continuation of $p_{m,\alpha}(s)M_{w_0}f_{\mu_s^{\alpha}}(w_0)$ to this range.  Moreover, if $A+m>\tfrac{n+1}{2}$, then for all $p \in \CC[s]$ one has
 \begin{align} \label{bound:stuff}
    |M_{w_0}f_{\mu^\alpha_s}(w_0)|_{A,B,pp_{m,\alpha}}\leq \nu(f)
\end{align}
for some continuous seminorm $\nu$ on $\mathcal{S}(X^\circ(F))$ depending on $p,m,A,B$. 

Assume henceforth that $A+m>\tfrac{n+1}{2}$.   Since the zeros of $p_{m,\alpha}$ are located in $\frac{1}{2}\ZZ$, by slightly decreasing $A$ and increasing $B$ if necessary, we are free to assume that no zeros of $p_{m,\alpha}$ are on the lines $\mathrm{Re}(s)=A$ or $\mathrm{Re}(s)=B$ for all $\alpha$.  
Let
 \begin{align*}
 \Omega :&=\{s \in V_{A,B}:|\mathrm{Im}(s)| < 1, A<\mathrm{Re}(s)<B \}.
 \end{align*}
Again using the fact that the zeros of $p_{m,\alpha}$ are located in $\tfrac{1}{2}\ZZ$, we have
\begin{align} \label{pmbound}
\max_{s \in V_{A,B}-\Omega} \frac{1}{p_{m,\alpha}(s)} \ll_{m} 1
\end{align}
 where the implied constant is independent of $\alpha$.  
  Assume now that $p_{w_0}$ satisfies \eqref{pw}.
Since $M_{w_0}f_{\mu^{\alpha}_s}$ is a good section \cite[Lemmas 1.2 and 1.3]{Ikeda:poles:triple}, the maximum modulus principal implies
 \begin{align*} 
 \mathrm{sup}_{s \in V_{A,B}}|p_{w_0}(s) M_{w_0}f_{\mu^\alpha_s}(w_0)| &\leq \mathrm{sup}_{s \in V_{A,B}-\Omega}|p_{w_0}(s) M_{w_0}f_{\mu^\alpha_s}(w_0)|\\
 &\leq |M_{w_0}f_{\mu^\alpha_s}(w_0)|_{A,B,p_{w_0}p_{m,\alpha}}\max_{s \in  V_{A,B}-\Omega}\frac{1}{p_{m,\alpha}(s)} \\& \leq \nu(f)
 \end{align*}
 for some continuous seminorm $\nu$ on $\mathcal{S}(X^\circ(F))$ depending on $p_{w_0},m,A,B$.  Here in the last inequality have used \eqref{bound:stuff} and \eqref{pmbound}. This implies \eqref{seminorm:stuff}.
\end{proof}

\section{Groups and orbits}
\label{sec:groups:orbits} 
 
For this section, $F$ is a field of characteristic zero.  For $1\leq i \leq 3$, let $V_i=\GG_a^{d_i}$ where $d_i$ is even and let $\mathcal{Q}_i$ be a nondegenerate quadratic form on $V_i(F)$. Let $\mathcal{Q}:=\mathcal{Q}_1+\mathcal{Q}_2+\mathcal{Q}_3.$ We put
\begin{align}\label{vcirc}
V^{\circ}_i:=V_i-\{0\} \textrm{ and }V^{\circ}:=V_1^\circ \times V_2^\circ \times V_3^{\circ}
\end{align}
and we let
$V' \subset V$
be the open subscheme consisting of $(v_1,v_2,v_3)$ such that no two $v_i$ are zero.    
For an $F$-algebra $R$, recall that
\begin{align} \label{Y}
Y(R)&=\{ (y_1,y_2,y_3) \in V(R):\mathcal{Q}_1(y_1)=\mathcal{Q}_2(y_2)=\mathcal{Q}_3(y_3)\}.
\end{align}
We observe that $Y^{\mathrm{sm}}=Y \cap V'.$
We let 
\begin{align} \label{Yani}
Y^{\mathrm{ani}}\subset Y
\end{align}
be the open complement of the vanishing locus of $\mathcal{Q}_i$ (it is independent of $i$).  

We let $\mathrm{GO}_{\mathcal{Q}_i}$ be the similitude group of $(V_i,\mathcal{Q}_i)$ and let $\nu:\mathrm{GO}_{\mathcal{Q}_{i}} \to \GG_m$ be the similitude norm.  We then set
\begin{align} \label{H}
H(R):=\left\{(h_1,h_2,h_3) \in \mathrm{GO}_{\mathcal{Q}_1}(R) \times \mathrm{GO}_{\mathcal{Q}_2}(R) \times \mathrm{GO}_{\mathcal{Q}_3}(R): \nu(h_1)=\nu(h_2)=\nu(h_3) \right\},
\end{align}
and define
\begin{align} \label{lambda} \begin{split}
\lambda:H(R) &\lto R^\times\\
(h_1,h_2,h_3) &\longmapsto \nu(h_1). \end{split}
\end{align}

Let
\begin{align} \label{tildeY0}
\widetilde{Y}_{0}(R):=\{ (y_1,y_2,y_3) \in V^\circ(R): \mathcal{Q}_1(y_1)=\mathcal{Q}_2(y_2)=\mathcal{Q}_3(y_3)=0\}
\end{align}
and let $Y_{0}$
be the (quasi-affine) quotient of $\widetilde{Y}_{0}$ by $\GG_m \times \GG_m$, acting via the restriction of the action
\begin{align} \label{act:Gm2}
\GG_m(R) \times \GG_m(R) \times V(R) &\lto V(R)\\
(a_1,a_2,(v_1,v_2,v_3)) &\longmapsto (a_1 v_1, a_2 v_2, (a_1a_2)^{-1}v_3). \nonumber
\end{align}
This quotient can be constructed by taking the affine closure $\overline{Y}_0$ of $\widetilde{Y}_0$ in $V$ and viewing $Y_0$ as an open subscheme of the GIT quotient of $\overline{Y}_0$ by $\GG_m \times \GG_m$.  We observe that $Y_0$ is a geometric quotient of $\widetilde{Y}_0$.  

For $1 \leq i \leq 3$, we define the scheme
\begin{align} \label{tildeYi}
\widetilde{Y}_{i}(R):&=\{(y_1,y_2,y_3) \in V^\circ(R):\mathcal{Q}_{i-1}(y_{i-1})=\mathcal{Q}_{i+1}(y_{i+1}) \textrm{ and } \mathcal{Q}_i(y_i)=0\}.
\end{align}
Here the indices are taken modulo $3$ in the obvious sense. Let $Y_1$ be the quotient of $\widetilde{Y}_1$ by $\GG_m$ acting via the restriction of the action
\begin{align}\label{quotient1}
    \GG_m(R)\times V(R)&\lto V(R)\\
    (a,(v_1,v_2,v_3))&\longmapsto (v_1,av_2,av_3) \nonumber.
\end{align}
This is nothing but the product over $F$ of the quasi-affine scheme cut out by $\mathcal{Q}_1$ in $V_1^\circ$ and the quasi-projective scheme cut out of $\mathbb{P}(V_2^\circ \times V_3^\circ)$ by $\mathcal{Q}_2=\mathcal{Q}_3$. 
The schemes $Y_2$ and $Y_3$ are defined similarly.
Thus
\begin{align}\label{quotient0}
Y_0:=\widetilde{Y}_0/\GG_m^2\quad \textrm{ and } \quad Y_i:=\widetilde{Y_i}/\GG_m
\end{align}
where the quotients are defined as above.  
Using Hilbert's theorem 90, we deduce the following lemma:
\begin{lem} \label{lem:Hilbert}
The maps
 $\widetilde{Y}_{0}(F)/(F^\times)^2 \to Y_{0}(F)$  and $\widetilde{Y}_i(F)/F^\times\to Y_i(F)$ are bijective. \qed
\end{lem}
We often 
identify $\SL_2^3(R)$ with the subgroup
$G(R)< \mathrm{Sp}_6(R)$ defined as follows:
\begin{align} \label{Gembed}
G(R)=\left\{\left(\begin{smallmatrix} a_1 & & &b_1& & \\ & a_2 & && b_2 & \\ & & a_3 & & & b_3\\ c_1 & & & d_1 & & \\ & c_2 & & & d_2 & \\ & & c_3 & & & d_3\end{smallmatrix}\right) \in \GL_6(R) :a_id_i-b_ic_i=1 \textrm{ for }1 \leq i \leq 3 \right\}.
\end{align}
We give a set of representatives for
\begin{align*}
X^{\circ}(F)/G(F)
\end{align*}
and the corresponding stabilizers. 
Let
\begin{align}\label{eq:gammas} \begin{split}
\gamma_b:&=\left(\begin{smallmatrix} 0 & 0 & 0 & -1 & 0& 0\\ 0 & 1 & 0 & 0 & 0 &0\\
0 & 0 & 1 & 0& 0 & 0\\
1 & 1 & 1 & 0 & 0 & 0\\ 0 & 0 & 0 &-1 & 1 &  0\\
0 & 0 & 0 & -1 & 0 & 1  \end{smallmatrix}\right),\\
\gamma_{i}:&=\left(\begin{smallmatrix} 1 & 0 & 0 & 0 & 0 & 0\\
0 & 1 & 0 & 0 & 0 &0\\
0 & 0 & 0 & 0 & 0 & 1\\
0 & 0 &0 &1 & 0 & 0
\\
0& 0 & 0 & 0 & 1 & 1\\
0 & 1 & -1 & 0 & 0 & 0\end{smallmatrix}\right) \left(\begin{smallmatrix} 0 & 1 & 0 & 0 & 0 & 0\\ 0 & 0 & 1 & 0 & 0 & 0 \\ 1 &0  & 0 & 0 & 0 & 0 \\0  & 0 & 0 & 0 & 1 & 0 \\ 0& 0 & 0 & 0 & 0 & 1 \\ 0 & 0 & 0 & 1 & 0 & 0 \end{smallmatrix} \right)^{i-1} \textrm{ for }1 \leq i \leq 3. \end{split}
\end{align}
All four matrices are in $\mathrm{Sp}_6(\ZZ)$. By \cite[Lemmas 2.1 and 2.2]{Getz:Liu:Triple}, the matrices $\gamma_i$ together with the identity matrix, denoted by $\gamma_0=\mathrm{Id}$, form a complete set of representatives of $X^{\circ}(F)/G(F)$ (strictly speaking, we have chosen different representatives for the $\gamma_i$ orbits than in \cite{Getz:Liu:BK}, but this does not affect the validity of \cite[Lemmas 2.1 and 2.2]{Getz:Liu:Triple}).  For $\gamma\in X^{\circ}(F)$, let $G_{\gamma} \leq G$ be the stabilizer of $\gamma$ under the right action. 
\begin{lem}\cite[Lemma 2.3]{Getz:Liu:Triple} \label{lem:stab} One has
\begin{align*}
G_{\gamma_b}(R):&=\left\{\left(\left(\begin{smallmatrix} 1 & t_1 \\ & 1 \end{smallmatrix} \right),\left(\begin{smallmatrix}1 & t_2 \\ & 1 \end{smallmatrix} \right),\left(\begin{smallmatrix} 1 & t_3 \\ & 1 \end{smallmatrix} \right) \right):t_1,t_2,t_3 \in R, t_1+t_2+t_3=0\right\},\\
G_{\mathrm{Id}}(R):&=\left\{\left(\left(\begin{smallmatrix} b_1^{-1} & t_1 \\ & b_1 \end{smallmatrix} \right),\left(\begin{smallmatrix}b_2^{-1} & t_2 \\ & b_2 \end{smallmatrix} \right),\left(\begin{smallmatrix} b_3^{-1} & t_3 \\ & b_3 \end{smallmatrix} \right) \right):t_1,t_2,t_3 \in R, b_1,b_2,b_3 \in R^\times, b_1b_2b_3=1\right\},\\
G_{\gamma_{1}}(R):&=\left\{\left(\left(\begin{smallmatrix} 1 & t \\ & 1 \end{smallmatrix} \right),g,\left(\begin{smallmatrix}1 & \\ & -1 \end{smallmatrix} \right)g\left(\begin{smallmatrix}1 & \\ & -1 \end{smallmatrix} \right)\right) :t \in R, g \in \SL_2(R)\right\},\\
G_{\gamma_{2}}(R):&=\left\{\left(\left(\begin{smallmatrix}1 & \\ & -1 \end{smallmatrix} \right)g\left(\begin{smallmatrix}1 & \\ & -1 \end{smallmatrix} \right),\left(\begin{smallmatrix} 1 & t \\ & 1 \end{smallmatrix} \right),g\right) :t \in R, g \in \SL_2(R)\right\},
\\
G_{\gamma_{3}}(R):&=\left\{\left(g,\left(\begin{smallmatrix}1 & \\ & -1 \end{smallmatrix} \right)g\left(\begin{smallmatrix}1 & \\ & -1 \end{smallmatrix} \right),\left(\begin{smallmatrix} 1 & t \\ & 1 \end{smallmatrix} \right)\right) :t \in R, g \in \SL_2(R)\right\}.
\end{align*}
\qed
\end{lem}

\section{Local functions} \label{sec:loc:func}

In this section, we define the local integrals required to state our summation formula and prove some of their basic properties. 
Let $F$ be a local field of characteristic zero.  We use the conventions on Schwartz spaces explained in \S \ref{sec:gen:Schwartz}.   For each of the $5$ orbits of $G(F)$ in $X^{\circ}(F)$
given in Lemma \ref{lem:stab}, we will define a family of integrals.

For $f=f_1\otimes f_2 \in \mathcal{S}(X(F)) \otimes \mathcal{S}(V(F))$, let
\begin{align} \label{Is}\begin{split}
I(f)\left(y \right)&:=\int_{G_{\gamma_b}(F) \backslash G(F)} f_1\left(\gamma_bg\right) \rho\left(g\right)f_2(y)dg, \quad y \in Y^{\mathrm{sm}}(F),\\
I_{0}(f)\left(y \right)&:=\int_{N_2^3(F) \backslash G(F)} f_1\left(g\right) \rho\left(g\right)f_2(y)dg, \quad y \in \widetilde{Y}_0(F).
\end{split}
\end{align}
Here the stabilizer $G_{\gamma_b}$ is computed in
Lemma \ref{lem:stab}.
These are integrals attached to the $G(F)$-orbit of $\gamma_b$ and $\gamma_0=\mathrm{Id}$, respectively.  

Let
$$
\Delta_i:\SL_2 \lto G
$$
be defined by 
\begin{align} \label{Deltai}
\Delta_i(h):=\begin{cases}\left(I_2,h,\left(\begin{smallmatrix} 1 & \\ & -1 \end{smallmatrix} \right)h\left(\begin{smallmatrix} 1 & \\ & -1 \end{smallmatrix} \right)\right) &\textrm{ for }i=1,\\
\left(\left(\begin{smallmatrix} 1 & \\ & -1 \end{smallmatrix} \right)h\left(\begin{smallmatrix} 1 & \\ & -1 \end{smallmatrix} \right),I_2,h\right) &\textrm{ for }i=2,\\
\left(h,\left(\begin{smallmatrix} 1 & \\ & -1 \end{smallmatrix} \right)h\left(\begin{smallmatrix} 1 & \\ & -1 \end{smallmatrix} \right),I_2\right) &\textrm{ for }i=3.
\end{cases}
\end{align}
Moreover let
\begin{align} \label{pi} \begin{split}
p_i:G(R) &\lto \SL_2(R)\\
(g_1,g_2,g_3) &\longmapsto g_{i+1} \end{split}
\end{align}
where the indices are taken modulo $3$ in the obvious sense.  

We need one more piece of data to define the integrals attached to the 
other orbits.  Let $\Phi \in \mathcal{S}(F^2)$. For $y \in \widetilde{Y}_i(F)$, $1 \leq i \leq 3$ and $s\in \CC$ with $\mathrm{Re}(s)>0$, we define
\begin{align} \label{Ii} \begin{split}
&I_{i}(f\otimes \Phi)(y,s)\\
&:=\int_{G_{\gamma_i}(F) \backslash G(F)}f_1\left(\gamma_{i}g\right)
\int_{N_2(F) \backslash \SL_2(F)}\int_{F^\times}
\rho\left(\Delta_i(h)g\right)f_2( y)\Phi(x(0,1)hp_i(g) )|x|^{2s}dx^{\times}dhdg.\end{split}
\end{align}
We point out that all of the integrals above can be defined directly for a general $f \in \mathcal{S}(X(F) \times \mathcal{S}(V(F))$ (not just a pure tensor) but the notation is more confusing.  One can also define them indirectly for all $f \in \mathcal{S}(X(F) \times V(F))$ using the definition for pure tensors give above.  
Indeed, in the non-Archimedean case $I(f)$, $I_0(f)$, and $I_i(f \otimes \Phi)$ are defined for all $f \in \mathcal{S}(X(F) \times V(F))$ by bilinearity.
In the Archimedean case, the estimates in \S \ref{sec:bound:arch} imply that for a given $f\in \mathcal{S}(X(F)\times V(F))$ and any sequence $f_n\in \mathcal{S}(X(F))\otimes \mathcal{S}(V(F))$ converging to $f$, the functions
 \begin{align*}
     I(f):=\lim_{n\to\infty} I(f_n), \quad  I_0(f):=\lim_{n\to\infty} I_0(f_n), \quad I_i(f \otimes \Phi):=\lim_{n \to \infty}I_i(f_n \otimes \Phi)
 \end{align*}
 are well-defined (via pointwise convergence). 

In \S \ref{sec:unr} we will compute the integrals defined in this section in the unramified setting. We prove that these integrals are absolutely convergent and bound them in the non-Archimedean case in  \S\ref{sec:bound:na} and in the Archimedean case in \S\ref{sec:bound:arch}.

 \subsection{The Schwartz space of $Y$} \label{ssec:Schwartz:Y} 
In Propositions \ref{prop:na:smooth} and \ref{prop:a:smooth}, we will show that $I(f)$ is a smooth function on $Y^{\mathrm{sm}}(F)$.  With this in mind, we define
\begin{align} \label{SY}
\mathcal{S}(Y(F)):=\mathrm{Im}(I:\mathcal{S}(X(F) \times V(F)) \to C^\infty(Y^{\mathrm{sm}}(F))).
\end{align}
This is the \textbf{Schwartz space of $Y(F)$}.  
We observe that \cite[Lemma 4.3]{Getz:Liu:Triple} implies in particular that the natural action of $H(F)$ on $C^\infty(Y^{\mathrm{sm}}(F))$  preserves $\mathcal{S}(Y(F))$.  

\begin{lem} \label{lem:Iker} Let $F$ be an Archimedean local field. The kernel of the map 
$$
I:\mathcal{S}(X(F) \times V(F)) \lto C^\infty(Y^{\mathrm{sm}}(F))
$$
is closed. 
\end{lem}

\begin{proof} For any $N \geq 0,$ the Cauchy-Schwarz inequality implies that $|I(f_1\otimes f_2)|(y)$ is bounded by the product of the square-roots of the following two integrals:
\begin{align} \label{left}
& \int_{G_{\gamma_b}(F)\backslash G(F)} |f_1|^2(\gamma_b g)\max(|\gamma_bg|,1)^{2N}dg,
\\
\label{right}&\int_{G_{\gamma_b}(F)\backslash G(F)}\max(|\gamma_bg|,1)^{-2N}|\rho(g)f_2|^2(y)dg.
\end{align}
Now $G_{\gamma_b}(F) \backslash G(F)$ is dense in $X^{\circ}(F)$ and hence the right $\mathrm{Sp}_6(F)$-invariant positive Radon measure on $X^{\circ}(F)$ agrees with $dg,$ at least after scaling by a positive real constant.  We continue to denote this measure on $X^{\circ}(F)$ by $dg.$ Thus \eqref{left} is equal to 
\begin{align} \label{use:seminorm}
    \int_{X^{\circ}(F)} |f_1|^2( g)\max(|g|,1)^{2N}dg.
\end{align}
Let $\nu_{\mathrm{Id},N+1,0}$ be defined as Lemma \ref{lem:seminorm}, where $\mathrm{Id}$ is the identity in $U(\mathfrak{g}).$  Using the decomposition of the measure $dg$ afforded by the Iwasawa decomposition and Lemma \ref{lem:seminorm}, we see that \eqref{use:seminorm} is bounded by 
 $\norm{f_1}^2_2+c^2\nu_{\mathrm{Id},N+1,0}(f_1)^2$, where $c$ is a positive constant independent of $f_1$.  
 
 On the other hand, by the Iwasawa decomposition \eqref{right} equals 
\begin{align*}
    &\int_{(F^\times)^3\times F} \max(m(t,a),1)^{-2N}\int_{K^3}|\rho(k)f_2|^2(a^{-1}y)dk\left(\prod_{i=1}^3 |a_i|^{2-d_i}\right) d^\times a dt.
\end{align*}
 Here
\begin{align}\label{eq:gamma0gnorm}
    m(t,a):=\max(|ta_1a_2a_3|,|a_1|,|a_2|,|a_3|,|a_1^{-1}a_2a_3|,|a_2^{-1}a_1a_3|,|a_3^{-1}a_1a_2|)
\end{align}
(see the proof of \cite[Proposition 7.1]{Getz:Liu:Triple} for more details). 
Taking $N$ sufficiently large and applying Lemma \ref{lem:universal} in the special case $D=\mathrm{Id}$, $r=e_i=0$ we deduce that the linear form $f \mapsto I(f)(y)$ is continuous for every $y \in Y^{\mathrm{sm}}(F)$.  
The kernel in the statement of the lemma is the intersection of the kernels of these continuous linear forms.
\end{proof}
\noindent We endow $\mathcal{S}(Y(F))=\mathcal{S}(X(F) \times V(F))/\ker I$ with the quotient topology (which is Fr\'echet).  

In this Archimedean setting, there is a family of seminorms $\{\nu\}$ such that $\mathcal{S}(X(F) \times V(F))$ is the set of smooth functions $f:X^{\circ}(F) \times V(F) \to \CC$ satisfying $\nu(f)<\infty$ for all $\nu.$  The seminorms $\nu$ we have in mind are tensor products of the seminorms \eqref{seminorm} and the usual seminorms on $\mathcal{S}(V(F)).$  Since $\mathcal{S}(Y(F))$ is a topological quotient space of $\mathcal{S}(X(F) \times V(F)),$ it is then also a space of smooth functions on $Y^\circ(F)$ on which a family of seminorms are finite.   

The integrals $I(f)$ depend on the choice of additive character $\psi$ used to define the Weil representation $\rho_{\psi}$. We write
$I_{\psi}(f)$
    for $I(f)$ defined using the Weil representation  $\rho_{\psi}$. Let $\gamma(\mathcal{Q},\psi):=\prod_{i=1}^3\gamma(\mathcal{Q}_i,\psi)$ be the product of the Weil indices.

\begin{lem} \label{lem:Ipsi}
 Let $c\in F^\times$ and $\psi_c(x):=\psi(cx)$. Then
 \begin{align*}
     I_{\psi_c}(f_1\otimes f_2)(y)=\frac{\gamma(\mathcal{Q},\psi_c)}{\gamma(\mathcal{Q},\psi)}|c|^{-2+\sum_{i=1}^3 d_i/2} I_\psi\left(L(m(c^{-1})) R(\begin{smallmatrix} cI_3& \\
     & I_3
     \end{smallmatrix})f_1\otimes f_2\right)(cy).
 \end{align*}
 In particular, the Schwartz space $\mathcal{S}(Y(F))$ is independent of the choice of $\psi$. 
\end{lem}

\begin{proof} Let $B_2 < \SL_2$ be the Borel subgroup of upper triangular matrices and let 
$$
w_0=\left(\begin{psmatrix} & 1\\ -1\end{psmatrix},\begin{psmatrix} & 1\\ -1\end{psmatrix},\begin{psmatrix} & 1\\ -1\end{psmatrix}\right) \in \SL_2^3(F).
$$
Since $N_2^3(F)w_0 B_2^3(F)$ is dense in $\SL_2^3(F)$, we have 
\begin{align*}
    &I_{\psi_c}(f_1\otimes f_2)(y)\\
    &=\int_{F\times F^3\times (F^\times)^3} f_1(\gamma_b\begin{psmatrix}1 & t \\ & 1 \end{psmatrix}  w_0 \begin{psmatrix} 1 & x \\ & 1 \end{psmatrix} \begin{psmatrix} a & \\ & a^{-1} \end{psmatrix})\rho_{\psi_c}(\begin{psmatrix}1 & t \\ & 1 \end{psmatrix}  w_0 \begin{psmatrix} 1 & x \\ & 1 \end{psmatrix} \begin{psmatrix} a & \\ & a^{-1}\end{psmatrix})f_2(y) dt dx d^\times a . 
\end{align*}
Observe that
\begin{align*} 
    &\rho_{\psi_c}(\begin{psmatrix}1 & t \\ & 1 \end{psmatrix}  w_0 \begin{psmatrix} 1 & x \\ & 1 \end{psmatrix} \begin{psmatrix} a & \\ & a^{-1}\end{psmatrix})f_2(y)\\
    &=\psi(ct\mathcal{Q}(y))\gamma(\mathcal{Q},\psi_c)\int_{V(F)}\rho_{\psi_c}(\begin{psmatrix} 1 & x \\ & 1 \end{psmatrix} \begin{psmatrix} a & \\ & a^{-1}\end{psmatrix})f_2(u)\prod_{i=1}^3\psi(cu^t_iJ_i y_i)du_i\\
    &= \psi(c^{-1}t\mathcal{Q}(cy))\gamma(\mathcal{Q},\psi_c)\int_{V(F)}\rho_{\psi}\begin{psmatrix} a & \\ & a^{-1}\end{psmatrix}f_2(u)\prod_{i=1}^3\psi(cx_i\mathcal{Q}_i(u_i))\psi(u_i^tJ_i cy_i)du_i\\
    &=  |c|^{\sum_{i=1}^3 d_i/2}\gamma(\mathcal{Q},\psi_c)\gamma(\mathcal{Q},\psi)^{-1}\rho_{\psi}(\begin{psmatrix}1 & c^{-1}t \\ & 1 \end{psmatrix}  w_0 \begin{psmatrix} 1 & cx \\ & 1 \end{psmatrix} \begin{psmatrix} a & \\ & a^{-1}\end{psmatrix})f_2(cy).
\end{align*}
Here $J_i$ is the matrix of $\mathcal{Q}_i$.  The factor of $|c|^{\sum_{i=1}^3 d_i/2}$ appears because we have to renormalize the self-dual Haar measures with respect to $\psi_c$ so that they are self-dual with respect to $\psi$.
Taking a change of variables $t\mapsto ct$, $x_i\mapsto c^{-1}x_i$, we see that $I_{\psi_c}(f_1\otimes f_2)(y)$  is $|c|^{\sum_{i=1}^3 d_i/2}\gamma(\mathcal{Q},\psi_c)\gamma(\mathcal{Q},\psi)^{-1}$ times
\begin{align*}
    &\int_{F\times F^3\times (F^\times)^3} f_1(\gamma_b\begin{psmatrix}1 & ct \\ & 1 \end{psmatrix}  w_0 \begin{psmatrix} 1 & c^{-1}x \\ & 1 \end{psmatrix} \begin{psmatrix} a & \\ & a^{-1} \end{psmatrix})\rho_{\psi}(\begin{psmatrix}1 & t \\ & 1 \end{psmatrix}  w_0 \begin{psmatrix} 1 & x \\ & 1 \end{psmatrix} \begin{psmatrix} a & \\ & a^{-1}\end{psmatrix})f_2(cy)\frac{dt dx  d^\times a}{|c|^2} \\
    &=\int_{F\times F^3\times (F^\times)^3} f_1(\gamma_b\begin{psmatrix}1 &  \\ & c^{-1} \end{psmatrix}\begin{psmatrix}1 & t \\ & 1 \end{psmatrix}  w_0 \begin{psmatrix} 1 & x \\ & 1 \end{psmatrix} \begin{psmatrix} a & \\ & a^{-1} \end{psmatrix}\begin{psmatrix}c &  \\ & 1 \end{psmatrix})\\
    &\hskip1.5in  \times \rho_{\psi}(\begin{psmatrix}1 & t \\ & 1 \end{psmatrix}  w_0 \begin{psmatrix} 1 & x \\ & 1 \end{psmatrix} \begin{psmatrix} a & \\ & a^{-1}\end{psmatrix})f_2(cy) \frac{dt dx  d^\times a}{|c|^2} \\
    &= |c|^{-2} I_\psi\left(L(m(c^{-1})) R(\begin{smallmatrix} cI_3& \\
     & I_3
     \end{smallmatrix})f_1\otimes f_2\right)(cy).
\end{align*}
The fact that the Schwartz space is preserved now follows from Lemma \ref{lem:BKF} and \cite[Lemma 4.3]{Getz:Liu:Triple}.
\end{proof}

For $F$ Archimedean or non-Archimedean, 
let
$$
\mathcal{S}:=\mathrm{Im}(\mathcal{S}(V(F))\lto C^\infty(Y^{\mathrm{sm}}(F)))
$$
where the implicit map is restriction of functions.  We observe that $C_c^\infty(Y^{\mathrm{sm}}(F) ) <\mathcal{S}.$
Moreover, we have the following result: 

\begin{lem} \label{lem:smooth}   
One has 
\begin{align*}
    \mathcal{S}&=\mathrm{Im}\left(I:C^\infty_c(\gamma_bG(F))\otimes \mathcal{S}(V(F))\lto C^\infty(Y^\mathrm{sm}(F))\right)\\
    &=\mathrm{Im}\left(I:\mathcal{S}(\gamma_bG(F) \times V(F))\lto C^\infty(Y^\mathrm{sm}(F))\right),
\end{align*}
where the tensor product is algebraic.  In particular, $\mathcal{S} <\mathcal{S}(Y(F))$.
\end{lem}

\begin{proof}  
Let $\gamma_b G(F)\cong G_{\gamma_b}(F) \backslash G(F)$ be the orbit of $\gamma_b$ in $X^\circ(F)$; it is open and of full measure in $X^{\circ}(F).$ We have a commutative diagram
$$
\begin{tikzcd}
    \mathcal{S}(G(F) \times V(F))\arrow[r] \arrow[d] &\mathcal{S}(V(F)) \arrow[d,"|_{Y^{\mathrm{sm}}(F)}"]\\
    \mathcal{S}(\gamma_bG(F) \times V(F)) \arrow[r,"I"]& \mathcal{S}(Y(F))
\end{tikzcd}
$$
where the top horizontal arrow is the unique (continuous) linear map sending a pure tensor 
$$
\Phi \otimes f \in \mathcal{S}(G(F)) \otimes \mathcal{S}(V(F))
$$
to $\int_{G(F)}\Phi(g)\rho(g)fdg$ and the left vertical arrow sends $f(g,v) $ to $\int_{G_{\gamma_b}(F)}f(ng,v)dn.$  The right vertical arrow is given by restriction, and hence its image is $\mathcal{S}$ by definition.

We claim the upper horizontal map is surjective. In the non-Archimedean case, this follows easily from the smoothness of the Weil representation. In the Archimedean case, it is a consequence of a well-known theorem of Dixmier-Malliavin, that the map is surjective even if the domain is restricted to $C^\infty_c(G(F))\otimes \mathcal{S}(V(F))$. Therefore, $I$ is surjective even if the domain is restricted to $C^\infty_c(\gamma_bG(F))\otimes \mathcal{S}(V(F))$ and the lemma follows.
\end{proof}

\noindent This lemma provides a robust supply of elements of the Schwartz space $\mathcal{S}(Y(F)).$

We now revert to the adelic setting, bearing in mind the conventions on Schwartz spaces explained in \S \ref{sec:gen:Schwartz}.  Let $F$ be a number field.  The obvious global analogue of \eqref{Is} yields a map $I:\mathcal{S}(X(\A_F) \times V(\A_F)) \to C^\infty(Y^{\mathrm{sm}}(\A_F))$ and we set
$$
\mathcal{S}(Y(\A_F)):=\mathrm{Im}(I:\mathcal{S}(X(\A_F) \times V(\A_F)) \to C^\infty(Y^{\mathrm{sm}}(\A_F))).
$$
To check that this is well-defined, one uses the computation of the \textbf{basic function for $Y(F_v)$}
\begin{align*}
b_{Y,v}:=I(b_{X,v} \otimes \one_{V(\OO_{v})})
\end{align*}
in Proposition \ref{prop:unram:comp} below.  
We define
$$
\mathcal{S}(Y(F_\infty)):=\mathrm{Im}(I:\mathcal{S}(X(F_\infty) \times V(F_\infty)) \to C^\infty(Y^{\mathrm{sm}}(F_\infty))).
$$
The map $I$ has closed kernel by a trivial modification of the proof of Lemma \ref{lem:Iker}  and we give $\mathcal{S}(Y(F_\infty))$ the quotient topology.  Hence $\mathcal{S}(Y(F_\infty))$ is the (completed) projective topological tensor product $\widehat{\otimes}_v \mathcal{S}(Y(F_v)).$  We then have
$$
\mathcal{S}(Y(\A_F))=\widehat{\otimes}_{v|\infty} \mathcal{S}(Y(F_v)) \otimes \bigotimes_{v\nmid \infty}{}' \mathcal{S}(Y(F_v))
$$
where the restricted tensor product is taken with respect to the $b_{Y,v}$. Indeed, we have
\begin{align*}
    \mathcal{S}(Y(\mathbb{A}_F))&=I(\mathcal{S}(X(\mathbb{A}_F)\times V(\mathbb{A}_F)))\\
    &=\bigcup_{\infty \subset S} I\left(\mathcal{S}(X(F_S)\times V(F_S))\otimes \bigotimes_{v\not\in S} (b_{X,v}\otimes \one_{V(\OO_v)})\right) \\
    &=\bigcup_{\infty \subset S} \mathcal{S}(Y(F_S))\otimes \bigotimes_{v\not\in S} b_{Y,v}\\
    &=\widehat{\otimes}_{v|\infty} \mathcal{S}(Y(F_v)) \otimes \bigotimes_{v\nmid \infty}{}' \mathcal{S}(Y(F_v)).
\end{align*}

\section{The summation formula} \label{sec:summation}

  Let $F$ be a number field. Our goal in this section is to prove our main summation formula, Theorem \ref{thm:main:intro}, modulo some convergence statements that we prove later in the paper.   
We require the following assumptions on $f=f_1\otimes f_2 \in \mathcal{S}(X(\A_F) \times V(\A_F))$: There are finite places $v_1,v_2$ of $F$ (not necessarily distinct) such that
\begin{align*}
    f_1&=f_{v_1}f_{v_2}f^{v_1v_2} \textrm{ and } f_{v_1} \in C_c^\infty(X^{\circ}(F_{v_1})),\textrm{ } \mathcal{F}_X(f_{v_2}) \in C_c^\infty(X^{\circ}(F_{v_2})),\\
    \rho(g)f_2(v)&=0\,\,\textrm{for $v\not\in V^\circ(F)$,\, for all $g\in G(\A_F)$.}
\end{align*}
We will also require that $\Phi \in \mathcal{S}(\A_F^2)$ satisfies $\widehat{\Phi}(0)\neq 0$, where 
\begin{align} \label{2dFT}
\widehat{\Phi}(x,y)=\int_{\A_F^2}\Phi(t_1,t_2)\psi(xt_1+yt_2)dt_1dt_2.
\end{align}

We prove Theorem \ref{thm:main:intro} in this section assuming the
absolute convergence statements given in Propositions \ref{prop:Id:AC} and \ref{prop:other:AC}.  We will indicate precisely when these propositions are used below.  After this section, much of the remainder of the paper is devoted to proving these convergence statements.  

Computing formally one has
\begin{align} \label{formal:comp}
\nonumber &\int_{G(F) \backslash G(\A_F)} \sum_{\gamma \in X^\circ (F)} f_1(\gamma g)\Theta_{f_2}(g)dg \\
\nonumber &=\sum_{\gamma \in X^\circ(F)/G(F)} \int_{G_{\gamma}(F) \backslash G(\A_F)} f_1(\gamma g)\Theta_{f_2}(g)dg\\
&=\sum_{\gamma \in X^\circ(F)/G(F)} \int_{G_{\gamma}(\A_F) \backslash G(\A_F)}f_1(\gamma g)\int_{[G_{\gamma}]}\Theta_{f_2}(g_1g)dg_1dg.
\end{align}
 The set $X^\circ(F)/G(F)$ has $5$ elements represented by $\gamma_b, \gamma_i$, $1 \leq i \leq 3$ and $\gamma_0=\mathrm{Id}$ in the notation of \eqref{eq:gammas}.  The stabilizers are given explicitly by Lemma \ref{lem:stab}, and we will use this lemma without further comment below.

We start with the $\gamma_b$ contribution.  It is computed as in the proof of \cite[Theorem 5.3]{Getz:Liu:Triple}:
\begin{align*}
\int_{G_{\gamma_b}(\A_F)\backslash G(\A_F)}f_1(\gamma_bg)\int_{[G_{\gamma_b}]}\Theta_{f_2}(g_1g)dg_1dg = \sum_{\xi \in Y^{\mathrm{sm}}(F)}I(f)(\xi).
\end{align*}
Strictly speaking, the proof of  \cite[Theorem 5.3]{Getz:Liu:Triple} assumed $f_{1}$ was finite under a maximal compact subgroup of $\mathrm{Sp}_{6}(F_\infty)$, but the same proof is valid given our work in \S \ref{sec:BK}.

We now turn to the $\mathrm{Id}$ term.
Using the definition of the Weil representation, we have that this term is 
\begin{align*}
\int_{G_{\mathrm{Id}}(\A_F) \backslash G(\A_F)} &f_1(
g)\int_{[\GG_m \times \GG_m]}
\sum_{\substack{\xi \in V(F)\\ \mathcal{Q}_1(\xi_1)=\mathcal{Q}_2(\xi_2)=\mathcal{Q}_3(\xi_3)=0}}\\& 
\rho\left(\left( \left(\begin{smallmatrix} a_1 & \\ & a_1^{-1}\end{smallmatrix} \right),\left(\begin{smallmatrix} a_2 & \\ & a_2^{-1} \end{smallmatrix} \right),\left(\begin{smallmatrix} (a_1a_2)^{-1} & \\ & a_1a_2 \end{smallmatrix} \right)\right)g\right) f_2( \xi)da_1^\times da_2^\times dg
\\=\int_{G_{\mathrm{Id}}(\A_F) \backslash G(\A_F)} &f_1(g)\int_{\A_F^\times \times \A_F^\times}
\sum_{\substack{\xi \in \widetilde{Y}_{0}(F)/(F^\times)^2}}\\& 
\rho\left(\left( \left(\begin{smallmatrix} a_1 & \\ & a_1^{-1}\end{smallmatrix} \right),\left(\begin{smallmatrix} a_2 & \\ & a_2^{-1} \end{smallmatrix} \right),\left(\begin{smallmatrix} (a_1a_2)^{-1} & \\ & a_1a_2 \end{smallmatrix} \right)\right)g\right) f_2( \xi )da_1^\times da_2^\times dg.
\end{align*}
Here $(F^\times)^2$ acts as in \eqref{act:Gm2}. Thus using Lemma \ref{lem:Hilbert} we conclude that 
the above is equal to 
\begin{align*}
&\int_{G_{\mathrm{Id}}(\A_F) \backslash G(\A_F)} f_1(g)\int_{\A_F^\times \times \A_F^\times}
\sum_{\substack{\xi \in Y_{0}(F)}}\\& \hspace{2.5 cm}
\rho\left(\left( \left(\begin{smallmatrix} a_1 & \\ & a_1^{-1}\end{smallmatrix} \right),\left(\begin{smallmatrix} a_2 & \\ & a_2^{-1} \end{smallmatrix} \right),\left( \begin{smallmatrix} (a_1a_2)^{-1}& \\ & a_1a_2  \end{smallmatrix} \right)\right)g\right) f_2( \xi)da_1^\times da_2^\times dg\\
&=\int_{N_2^3(\A_F) \backslash G(\A_F)} f_1(g)
\sum_{\substack{\xi \in Y_{0}(F)}} 
\rho\left(g\right) f_2( \xi) dg\\
&=\sum_{\xi \in Y_{0}(F)}I_{0}(f)(\xi)\,.
\end{align*}
This formal computation is justified by Proposition \ref{prop:Id:AC}.

We finally turn to the $\gamma_i$, $1 \leq i \leq 3$, terms. Let $\Phi \in \mathcal{S}(\A_F^2)$ be a function satisfying $\widehat{\Phi}(0)\neq 0$. 
We prove in Proposition \ref{prop:other:AC} below that the sum 
\begin{align*}
\sum_{\xi \in Y_{i}(F)}
I_i(f \otimes \Phi)(\xi,s)
\end{align*}
converges absolutely and defines a holomorphic function of $s$ for $\mathrm{Re}(s)$ sufficiently large.  Moreover, it admits a meromorphic continuation to the $s$ plane and its residue at $s=1$ is  
\begin{align*}
&\frac{\widehat{\Phi}(0)}{\kappa_F} \int_{G_{\gamma_{i}}(\A_F) \backslash G(\A_F)} f_1(\gamma_{i}g)\int_{[G_{\gamma_{i}}]}
\sum_{\xi \in V(F)}
\rho(hg)f_2( \xi)dh dg.
\end{align*}
where $\kappa_F:=\frac{2}{\mathrm{Vol}(F^\times \backslash (\A_F^\times)^1)}.$

Thus altogether we have shown that 
\begin{align*}
&\int_{G(F) \backslash G(\A_F)} \sum_{\gamma \in X^\circ(F)} f_1(\gamma g)\Theta_{f_2}(g)dg\\&=\sum_{\xi \in Y^{\mathrm{sm}}(F)} I(f)(\xi)
+\sum_{\xi \in Y_0(F)}I_0(f)(\xi)+\frac{\kappa_F}{\widehat{\Phi}(0)}
\mathrm{Res}_{s=1}\sum_{i=1}^3 \sum_{\xi \in Y_i(F)}I_i(f\otimes \Phi)(\xi,s).
\end{align*}
On the other hand by Theorem \ref{thm:BK:PS} 
\begin{align*}
\int_{G(F) \backslash G(\A_F)} \sum_{\gamma \in X^\circ(F)} f_1(\gamma g)\Theta_{f_2}(g)dg=\int_{G(F) \backslash G(\A_F)} \sum_{\gamma \in X^\circ(F)} \mathcal{F}_X(f_1)(\gamma g)\Theta_{f_2}(g)dg.
\end{align*}
Replacing $f_1$ by $\mathcal{F}_X(f_1)$ in the argument above we see that this is 
\begin{align*}
&\sum_{\xi \in Y^{\mathrm{sm}}(F)} I(\mathcal{F}_X(f))(\xi)
+\sum_{\xi \in  Y_0(F)}I_0(\mathcal{F}_X(f))(\xi)+
\frac{\kappa_F}{\widehat{\Phi}(0)}\mathrm{Res}_{s=1}\sum_{i=1}^3 \sum_{\xi \in Y_i(F)}I_i(\F_X(f)\otimes \Phi)(\xi,s).
\end{align*}
Thus assuming the absolute convergence statements in Propositions \ref{prop:Id:AC} and \ref{prop:other:AC}  we have proved Theorem \ref{thm:main:intro}.

\section{Bounds on integrals in the non-Archimedean case} \label{sec:bound:na}

Throughout this section $F$ is a  non-Archimedean local field of characteristic zero.  For $v_i \in V_i(F)=F^{d_i}$,
we let
\begin{align} \label{tensor}
\mathrm{ord}(v_i) \quad (\textrm{resp}.~|v_i|) \end{align}
be the minimum of the orders (resp.~ maximum of the  norms) of the entries of $v_i$ with respect to the standard basis on $V_i(F).$  Thus $|v_i|=q^{-\mathrm{ord}(v_i)}.$ We have natural induced bases on $V_i(F) \otimes V_j(F)$ and $V_1(F) \otimes V_2(F) \otimes V_3(F)$ and we define $|v_i \otimes v_j|,$ etc., similarly.

Fix functions
$$
(f=f_1\otimes f_2,\Phi) \in \mathcal{S}(X(F)\times V(F))  \times \mathcal{S}(F^2).
$$
We bound the integrals attached to these functions that appeared in the proof of Theorem \ref{thm:main:intro}.  These bounds will be used to deduce the absolute convergence statements of Propositions \ref{prop:Id:AC} and \ref{prop:other:AC} below.
All implicit constants in this section are allowed to depend on $f\otimes \Phi$.

First we pause to justify an assertion made in \S \ref{ssec:Schwartz:Y}:

\begin{prop}\label{prop:na:smooth} We have
$\mathcal{S}(Y(F)) <C^{\infty}(Y^\mathrm{sm}(F)).$
\end{prop}
\begin{proof}
Fix $v=(v_1,v_2,v_3)\in Y^{\mathrm{sm}}(F)$ and let $v'\in Y^{\mathrm{sm}}(F)$.  By symmetry we can assume that $|v_2||v_3| \neq 0$.  It suffices to show 
$$
\int_{G_{\gamma_b}(F) \backslash G(F)} |f_1(\gamma_bg)||\rho(g)f_2(v)-\rho(g)f_2(v')|dg
$$
is zero for $|v-v'|$ sufficiently small.  
We can choose $\kappa_v \in \RR_{>0}$ such that if $|v-v'|<\kappa_v$ then $|v'_i|=|v_i|$ for $2\le i\le 3$, and $|v_1|=|v_i'|$ if $v_1 \neq 0$.  For the remainder of the proof we assume $|v-v'|<\kappa_v$.  
By the Cauchy-Schwarz inequality and  Lemma \ref{lem:seminorm}, the integral above is bounded by
\begin{align*}
    \norm{f_1}_2\left(\int_{G_{\gamma_b}(F)\backslash G(F)}\one_{\ge c}(\gamma_b g)|\rho(g)f_2(v)-\rho(g)f_2(v')|^2dg\right)^{1/2}
\end{align*}
for some $c \in \ZZ.$

By the Iwasawa decomposition, it suffices to show that for all $c \in \ZZ$ the integral
\begin{align} \label{int:above}
    \int_{m(t,a)\le q^{-c}} \left(\int_{K^3} |\rho(k)f_2(a^{-1}v)-\psi(-t\mathcal{Q}(v)+t\mathcal{Q}(v'))\rho(k)f_2(a^{-1}v')|^2dk\right) \frac{d^\times a dt}{\prod_{i=1}^3 |a_i|^{d_i-2}}
\end{align}
is zero for $|v-v'|$ sufficiently small, where $m(t,a)$ is defined as in \eqref{eq:gamma0gnorm}.
The integral \eqref{int:above} is supported in the set of $a_i$ such that $|v_i|q^{-N}\le |a_i|\le q^{-c}$ for each $i$ for some $N$ depending on $f_2$. Since $m(t,a)\le q^{-c}$, we have additionally $q^{-2N+c}|v_2||v_3|\le q^{c}|a_2||a_3|\le |a_1|\le q^{-c}$.  We have assumed that $|v_2||v_3| \neq 0;$ hence the support of the integral, as a function of $a$, lies in a compact subset of $(F^\times)^3$ independent of $v'$ (since $|v-v'|<\kappa_v$).
Thus the integral over $t$ has support in a set that is independent of $v'$.  In particular, if $|v-v'|$ is sufficiently small, then $\psi(-t\mathcal{Q}(v)+t\mathcal{Q}(v'))=1$, so \eqref{int:above} becomes
\begin{align*} 
    \int_{m(t,a)\le q^{-c}} \left(\int_{K^3} |\rho(k)f_2(a^{-1}v)-\rho(k)f_2(a^{-1}v')|^2dk\right)\Big( \prod_{i=1}^3 |a_i|^{2-d_i}\Big)  d^\times a dt.
\end{align*}
Since the vector space $\langle \rho(k)f_2 \rangle_{k\in K }$ is finite dimensional, and the integral over $a$ is supported in a compact set independent of $v'$ (since $|v-v'|<\kappa_v$), for $v'$ close enough to $v$, the integral above vanishes.
\end{proof}

\begin{prop} \label{prop:I1int:na} For $v=(v_1,v_2,v_3) \in V^{\circ}(F)$, one has
\begin{align*}
\int_{N_2^3(F) \backslash G(F)}|f_1(g)\rho(g)f_2(v)|dg   \ll \prod_{i=1}^3 |v_i|^{-d_i/2}.
\end{align*}
The integral is supported in the set of $v$ satisfying $
|v_1\otimes v_2\otimes v_3| \ll 1$.
The function $I_0(f)(v)$ satisfies the same bounds on its magnitude and support.
\end{prop}

\begin{proof}
We decompose the Haar measure using the Iwasawa decomposition to see that the integral in the proposition is equal to 
\begin{align} \label{na:Ib2}
&\int_{N_2^3(F) \backslash G(F)} 
|f_1\left(g\right) \rho\left(g\right)
f_2(v)|dg\nonumber\\
&=\int_{(F^\times)^3 \times K^3} \left| f_1\left(\left(\begin{smallmatrix} a^{-1} & \\ & a\end{smallmatrix}\right)k\right)\rho\left(\left(\begin{smallmatrix} a^{-1} & \\ & a\end{smallmatrix}\right)k\right)f_2(v)\right|\Big(\prod_{i=1}^3 |a_i|^{2} d^\times a_i\Big) dk\nonumber\\
&=\int_{(F^\times)^3 \times K^3} \left| f_1\left(\left(\begin{smallmatrix} a^{-1} & \\ & a\end{smallmatrix}\right)k\right)
\rho\left(k\right)f_2(a^{-1} v)\right|
\Big(\prod_{i=1}^3 |a_i|^{2-d_i/2} d^\times a_i \Big) dk.
\end{align}  
Now
\begin{align*}
\left| \left(\begin{smallmatrix} a^{-1} & \\ & a\end{smallmatrix}\right)k \right| = |a_1 a_2 a_3|. 
\end{align*}
By Lemma \ref{lem:seminorm}, \eqref{na:Ib2} is bounded by a constant times 
 \begin{align} \label{int:supp}
 &  \int_{|a_1a_2a_3|\ll 1} \widetilde{f}_2(a^{-1} v)
\prod_{i=1}^3 |a_i|^{-d_i/2} d^\times a_i ,
 \end{align}
 where 
 \begin{align}\label{overk}
 \widetilde{f}_2(v):=\int_{K^3} |\rho(k)f_2(v)|dk.
 \end{align}
 Since $\widetilde{f}_2$ is compactly supported, we have that $|v_i| \ll_{f_2} |a_i|$ for $1 \leq i \leq 3$. Therefore \eqref{int:supp} is bounded by a constant times 
 \begin{align*}
     &\prod_{i=1}^3\int_{|v_i|\ll |a_i|} |a_i|^{-d_i/2} d^\times a_i \ll  \prod_{i=1}^3 |v_i|^{-d_i/2}.
 \end{align*}
Moreover, the support of \eqref{int:supp} as a function of $v$ satisfies $|v_1||v_2||v_3| \ll |a_1a_2 a_3|\ll 1$, as claimed.
\end{proof}

\begin{prop} \label{prop:other:na:bound}
 For $r>0$, as a function of $v \in V^{\circ}(F)$, the integral
\begin{align*}
&\int_{(N_2 (F) \times \Delta_1(\SL_2)(F)) \backslash G(F)}|f_1\left(\gamma_{1}g\right)|\\
&\times\int_{N_2(F) \backslash \SL_2(F)}\int_{F^\times}
|\rho\left(\left( I_2,h,\left(\begin{smallmatrix} 1 & \\ & -1 \end{smallmatrix} \right)h\left(\begin{smallmatrix} 1 & \\ & -1 \end{smallmatrix} \right)\right)g\right)f_2( v)\Phi(x(0,1)hp_1(g))||x|^{2r}d^\times xdhdg
\end{align*}
has support in $|v_1|\ll 1$.  It is bounded by a constant times
\begin{align*}
&C'^r\zeta(2r)\zeta(2r+d_2/2-1)|v_1|^{-d_1/2}|v_3|^{1-d_3/2}\max(|v_1||v_3|,|v_2|)^{1-d_2/2-2r} \\& \times \left(C+\max\left( 0,\mathrm{ord}(v_1 \otimes v_3))-\mathrm{ord}(v_2)\right)+\zeta(2r+d_2/2-1)\right).
\end{align*}
for some constant $C,C'> 0$. 
If $r=\mathrm{Re}(s)$,
the function $I_1(f\otimes\Phi)(v,s)$ satisfies the same bounds on its magnitude and support.
\end{prop}

\begin{proof}
The integral in the proposition is equal to 
\begin{align*}
&\int_{N_2(F) \backslash \SL_2(F)}\int_{N_2 (F)\backslash \SL_2(F) \times \mathrm{SL}_2(F)}|f_1\left(\gamma_{1}(g_1,g_2,I_2)\right)|\\
&\times\int_{F^\times}
|\rho\left(\left( g_1,hg_2,\left(\begin{smallmatrix} 1 & \\ & -1 \end{smallmatrix} \right)h\left(\begin{smallmatrix} 1 & \\ & -1 \end{smallmatrix} \right)\right)\right)f_2( v)\Phi(x(0,1)hg_2)||x|^{2r}d^\times xdg_1dg_2dh.
\end{align*}
We  change variables $g_2 \mapsto h^{-1}g_2$ to see that this is 
\begin{align*}
    &\int_{N_2(F) \backslash \SL_2(F)}\int_{N_2 (F)\backslash \SL_2(F) \times \mathrm{SL}_2(F)}|f_1\left(\gamma_{1}(g_1,h^{-1}g_2,I_2)\right)|\\
&\times\int_{F^\times}
|\rho\left(\left( g_1,g_2,\left(\begin{smallmatrix} 1 & \\ & -1 \end{smallmatrix} \right)h\left(\begin{smallmatrix} 1 & \\ & -1 \end{smallmatrix} \right)\right)\right)f_2( v)\Phi(x(0,1)g_2)||x|^{2r}d^\times xdg_1dg_2dh.
\end{align*}
Since $\Delta_1(\SL_2(F))$ is in the stabilizer of $\gamma_1$, this is 
\begin{align*}
    &\int_{N_2(F) \backslash \SL_2(F)}\int_{N_2 (F)\backslash \SL_2(F) \times \mathrm{SL}_2(F)}|f_1\left(\gamma_{1}(g_1,g_2,\left(\begin{smallmatrix} 1 & \\ & -1 \end{smallmatrix} \right)h\left(\begin{smallmatrix} 1 & \\ & -1 \end{smallmatrix}\right)\right)|\\
&\times\int_{F^\times}
|\rho\left(\left( g_1,g_2,\left(\begin{smallmatrix} 1 & \\ & -1 \end{smallmatrix} \right)h\left(\begin{smallmatrix} 1 & \\ & -1 \end{smallmatrix} \right)\right)\right)f_2( v)\Phi(x(0,1)g_2)||x|^{2r}d^\times xdg_1dg_2dh.
\end{align*}
Now decomposing the Haar measure using the Iwasawa decomposition, we see that this is 
\begin{align*} 
&\int_{ (F^\times)^3 \times F \times K^3\times F^\times }\left|f_1\left(\gamma_1\left(\left(\begin{smallmatrix} a_1^{-1} & \\ & a_1 \end{smallmatrix} \right)k_1,\left(\begin{smallmatrix} 1& t \\ & 1 \end{smallmatrix} \right)\left(\begin{smallmatrix} a_2^{-1}&  \\ & a_2 \end{smallmatrix} \right)k_2,\left(\begin{smallmatrix}a_3^{-1} & \\ & a_3 \end{smallmatrix} \right)k_3\right)\right)\right|\\
& \times\left|\rho\left(\left(\begin{smallmatrix} a_1^{-1} &\\  & a_1\end{smallmatrix}\right)k_1,\left(\begin{smallmatrix} 1& t \\ & 1 \end{smallmatrix} \right)\left(\begin{smallmatrix} a_2^{-1}&  \\ & a_2 \end{smallmatrix} \right)k_2,\left(\begin{smallmatrix}a_3^{-1} & \\ & a_3 \end{smallmatrix} \right)k_3\right)f_2(v)\Phi((0,xa_2)k_2)\right|\\& \times|x|^{2r}  |a_1a_2a_3|^{2} d^\times x dt da_1^\times da_2^\times da_3^\times dk_1 dk_2dk_3.
\end{align*}
We have
\begin{align*}\begin{split}
&\left| \gamma_1\left(\left(\begin{smallmatrix} a_1^{-1} & \\ & a_1 \end{smallmatrix} \right)k_1,\left(\begin{smallmatrix} 1& t \\ & 1 \end{smallmatrix} \right)\left(\begin{smallmatrix} a_2^{-1}&  \\ & a_2 \end{smallmatrix} \right)k_2,\left(\begin{smallmatrix}a_3^{-1} & \\ & a_3 \end{smallmatrix} \right)k_3\right)\right|\\ &=\max(|a_1|,|a_1a_2a_3^{-1}|, |a_1a_2^{-1}a_3|,|a_1a_2a_3t|)\\
&=:m'(t,a).\end{split}
\end{align*}
Taking a change of variable $x\mapsto xa_2^{-1}$, by Lemma \ref{lem:seminorm} the integral above is bounded by a constant times
\begin{align*}
q^{2nr}\zeta(2r)\int_{\substack{ (F^\times)^3 \times F \\ m'(t,a) \ll  1}}m'(t,a)^{-2} \widetilde{f}_2(a^{-1}v)dt |a_2|^{-2r}\prod_{i=1}^3|a_i|^{2-d_i/2}  d^\times a_i
\end{align*}
for some $n\in \ZZ$, where $\widetilde{f}_2$ is the nonnegative function defined in \eqref{overk}. For some $N \in \ZZ_{\geq 0}$ sufficiently large, we can write the integral here as
\begin{align*}
\sum_{k=-N}^\infty q^{2k} \int \widetilde{f}_2(a^{-1}v)dt|a_2|^{-2r} \prod_{i=1}^3|a_i|^{2-d_i/2}  d^\times a_i
\end{align*}
where the integral is over $t,a$ such that $m'(t,a) = q^{-k}$.  This is bounded by 
\begin{align*}
\sum_{k=-N}^\infty \int q^k  \widetilde{f}_2(a^{-1}v) |a_2|^{-2r}\prod_{i=1}^3|a_i|^{1-d_i/2}  d^\times a_i,
\end{align*}
where the integral is now over $a$ such that $m'(a):=\max(|a_1|,|a_1a_2a_3^{-1}|,|a_1a_2^{-1}a_3|) \leq q^{-k}$.
Taking a change of variables $a_1 \mapsto \varpi^ka_1$, one arrives at
\begin{align}\label{other:int:unr}
\sum_{k=-N}^\infty \int q^{kd_1/2}  \widetilde{f}_2\left( \frac{v_1}{\varpi^ka_1},\frac{v_2}{a_2}, \frac{v_3}{a_3}\right)|a_2|^{-2r} \prod_{i=1}^3|a_i|^{1-d_i/2}  d^\times a_i
\end{align}
where the integral is now over $a_1,a_2,a_3$ such that 
$$
1 \geq \max(|a_1|,|a_1a_2a_3^{-1}|,|a_1a_2^{-1}a_3|).
$$
The bound on the support as a function of $v_1$ is now obvious. 
We also deduce that if $a$ is in the support of the integral for a given $v$, then
$$
|v_3| \ll |a_3| \leq |a_1|^{-1}|a_2| \ll |v_1|^{-1}|a_2|.
$$
Thus for some $C,C'>0$ depending on $f_2$, \eqref{other:int:unr}  is bounded by a constant times
\begin{align*}
&|v_1|^{-d_1/2} \int_{|v_3| \ll |a_3| \ll |v_1|^{-1}|a_2|,\, |v_2|\ll |a_2|}|a_3|^{1-d_3/2}|a_2|^{1-d_2/2-2r}d^\times a_3d^\times a_2 \\ 
 &\ll_{d_3} |v_1|^{-d_1/2}|v_3|^{1-d_3/2} \int_{\max(|v_1||v_3|, |v_2|)\ll |a_2|}(C+\ord(v_1\otimes v_3)-\ord(a_2))|a_2|^{1-d_2/2-2r}d^\times a_2 \\
 & \leq |v_1|^{-d_1/2}|v_3|^{1-d_3/2}C'^{1-d_2/2-2r}\max(|v_1||v_3|,|v_2|)^{1-d_2/2-2r}\\& \times  \int_{1 \leq |a_2|}(C+\mathrm{ord}(v_1 \otimes v_3)-\min(\mathrm{ord}(v_1 \otimes v_3),\mathrm{ord}(v_2))-\mathrm{ord}(a_2))|a_2|^{1-d_2/2-2r}d^\times a_2
\\
&\leq |v_1|^{-d_1/2}|v_3|^{1-d_3/2}C'^{1-d_2/2-2r}\max(|v_1||v_3|,|v_2|)^{1-d_2/2-2r}\zeta(2r+d_2/2-1) \\& \times \left(C+\max\left( 0,\mathrm{ord}(v_1 \otimes v_3)-\mathrm{ord}(v_2)\right)+\zeta(2r+d_2/2-1)\right).
\end{align*}
\end{proof}

\section{The unramified calculation} \label{sec:unr} 
  
For this section, $F$ is a local field unramified over $\QQ_p$ and $\psi:F \to \CC^\times$ is an unramified nontrivial character.
Let 
\begin{align}
\chi_\mathcal{Q}(a_1,a_2,a_3):=\chi_{\mathcal{Q}_1}(a_1)\chi_{\mathcal{Q}_2}(a_2)\chi_{\mathcal{Q}_3}(a_3)
\end{align}
where $\chi_{\mathcal{Q}_i}$ is the (quadratic) character attached to $\mathcal{Q}_i$ as in \cite[\S 3.1]{Getz:Liu:Triple}. We assume that $\chi_\mathcal{Q}$ is unramified and $\one_{V(\OO)}$ is fixed under $\rho(K^3)$ where $K=\SL_2(\OO)$. Recall $\one_c$ defined in \eqref{1k} and the basic function 
$$
b_X:=\sum_{j,k=0}^\infty q^{2j}\one_{k+2j}.
$$
In this section we give formulae for the unramified functions $
I(b_X \otimes \one_{V(\OO)})(v), I_0(b_X\otimes \one_{V(\OO)})(v), 
I_{i}(b_X\otimes \one_{V(\OO)}\otimes\one_{\OO^2})(v,s)
$
for $1 \leq i \leq 3$.

\subsection{The open orbit} 
For the reader's convenience, we state the formula for
\begin{align} \label{bY}
    b_Y:=I(b_X\otimes \one_{V(\OO)})
\end{align}
given by \cite[Proposition 6.3]{Getz:Liu:Triple}:
\begin{prop} \label{prop:unram:comp} 
For $v\in Y^\mathrm{sm}(F)$, one has 
\begin{align*} 
b_Y(v)=\sum_{j=0}^\infty \int \one_{\OO}\left(\frac{\mathcal{Q}(v)}{a_1a_2a_3\varpi^{4j}} \right) \one_{V(\OO)}\left(\frac{v}{a\varpi^{2j}}\right) \chi_{\mathcal{Q}}(a)\prod_{i=1}^3 \left(\frac{|a_i|}{q^{2j}}\right)^{1-d_i/2}d^\times a
\end{align*}
where the integral is over $a_1,a_2,a_3 \in \OO$ satisfying
\begin{align*} 
\max(|a_1^{-1}a_2a_3|,|a_2^{-1}a_1a_3|,|a_3^{-1}a_1a_2|) \leq 1.
\end{align*}  \qed
\end{prop}
One can also write the basic function $b_Y(v)$ as
\begin{align*}
   \one_{Y(\mathcal{O})}(v) \sum_{\substack{0\le k_i\le \ord (v_i)\\ k_i\le k_{i+1}+k_{i-1}\\
   k_1+k_2+k_3\le \ord \mathcal{Q}(v)}}  \frac{1-q^{\left(\min_i\left(\left\lfloor \frac{\ord \mathcal{Q}(v)-\sum_{i=1}^3 k_i}{4}\right\rfloor,\left\lfloor\frac{\ord(v_i)-k_i}{2}\right\rfloor\right)+1\right)(d_1+d_2+d_3-6)}}{(1-q^{d_1+d_2+d_3-6})\prod_{i=1}^3\chi_{\mathcal{Q}_i}(\varpi^{k_i})}q^{\sum_{i=1}^3 k_i(d_i/2-1)}.
\end{align*}

\subsection{The identity orbit}

By a minor modification of the proof of Proposition \ref{prop:I1int:na} above, we obtain the following proposition:

\begin{prop} \label{prop:I1comp}
 Suppose $v=(v_1,v_2,v_3)\in \widetilde{Y_0}(F)$. Assume moreover that $|v_1|=|v_2|=1$.
One has
\begin{align*} 
I_0(b_X\otimes \one_{V(\OO)})(v)=\sum_{k=0}^\infty \sum_{j=0}^\infty  
q^{2j}\int
\one_{V(\OO)}(a^{-1}v)\chi_\mathcal{Q}(a)\prod_{i=1}^3|a_i|^{2-d_i/2}d^\times a_i 
\end{align*}
where the integral is over those $a_1^{-1},a_2^{-1},a_3 \in \OO$ such that $
|a_1a_2a_3|=q^{-k-2j}.$
As a function of $v_3$, the integral is supported in $V_3(\OO)$.
 Let $\epsilon>0$. For $v\in V^\circ (F)$ with $|v_1|=|v_2|=1$, we have that
\begin{align*}
\int_{N_2^3(F) \backslash G(F)} |b_X\left(g\right) \rho\left(g\right)\one_{V(\OO)}(v)|dg \leq C|v_3|^{-d_3/2-\epsilon}\one_{V_3(\OO)}(v_3)
\end{align*}
for some constant $C>0$ depending on $\epsilon$, which equals $1$ for $q$ sufficiently large. \qed
\end{prop}

 Evaluating at $v\in \widetilde{Y}_0(F)$ with $|v_1|=|v_2|=1$, one can rewrite $I_0(b_X\otimes \one_{V(\OO)})(v)$ as
\begin{align*}
    \one_{V_3(\OO)}(v_3)\sum_{\substack{0\le k_i\le \ord(v_3)\\ k_1+k_2\le k_3}} \frac{1-q^{2+2\lfloor \frac{k_3-k_1-k_2}{2}\rfloor}}{1-q^2} q^{k_1(2-d_1/2)+k_2(2-d_2/2)-k_3(2-d_3/2)} \prod_{i=1}^3\chi_{\mathcal{Q}_i}(\varpi^{k_i}).
\end{align*}

\subsection{The other orbits}

The following assertions can be proved by an easy refinement of the argument proving Proposition \ref{prop:other:na:bound}:

\begin{prop} \label{prop:other:comp}
Suppose $v=(v_1,v_2,v_3) \in  \widetilde{Y_1}(F)$. For $\mathrm{Re}(s)>0$ one has
\begin{align*}
&I_{1}(b_X\otimes \one_{V(\OO)}\otimes \one_{\OO^2})(v,s)
\\&=\zeta(2s)\sum_{j=0}^\infty \int \one_{\OO}\left( \frac{\mathcal{Q}_2(v_2)}{a_1a_2a_3} \right)\one_{V(\OO)}\left(\frac{v_1}{\varpi^{2j}a_1},\frac{v_2}{a_2},\frac{v_3}{a_3}\right) \frac{\chi_\mathcal{Q}(a)}{|a_2|^{2s}q^{2j(1-d_1/2)}}\prod_{i=1}^3|a_i|^{1-d_i/2}  d^\times a_i
\end{align*}
where the integral is over $a_1 \in F^\times \cap \OO, a_2,a_3\in F^\times $ such that $
|a_1|^{-1} \geq \max\left( |a_2a_3^{-1}|,|a_2^{-1}a_3| \right).$ \qed
\end{prop}

One can alternatively write $I_{1}(b_X\otimes \one_{V(\OO)}\otimes \one_{\OO^2})(v,s)$ as
\begin{align*}
 \one_{V_1(\OO)}(v_1)\zeta(2s)\sum_{\substack{k_i\le \ord(v_i)\\ k_1\ge |k_2-k_3|\\ k_1+k_2+k_3\le \ord(\mathcal{Q}_2(v_2))}} q^{2k_2s+\sum_{i=1}^3 k_i(d_i/2-1)}\left(\sum_{j=0}^{\lfloor \frac{\ord(v_1)-k_1}{2}\rfloor} q^{j(d_1-2)}\right)\prod_{i=1}^3\chi_{\mathcal{Q}_i}(\varpi^{k_i}).
\end{align*}

\begin{lem} \label{lem:otherbound}
For $v=(v_1,v_2,v_3)\in V^{\circ}(F)$  and $r>0$, the integral 
\begin{align*}
&\int_{(N_2 (F) \times \Delta_1(\SL_2)(F)) \backslash G(F)}|b_X\left(\gamma_{1}g\right)|\\
&\times
\int_{N_2(F) \backslash \SL_2(F)}
|\rho\left(\left( I_2,h,\left(\begin{smallmatrix} 1 & \\ & -1 \end{smallmatrix} \right)h\left(\begin{smallmatrix} 1 & \\ & -1 \end{smallmatrix} \right)\right)g\right)\one_{V(\OO)}( v)|\one_{\OO^2}((0,x)hp_1(g))|x|^{2r}d^\times xdhdg
\end{align*}
vanishes unless $v_1 \in V_1(\OO)$. It is bounded by
\begin{align*}
 &C\zeta(2r)\zeta(2r+d_2/2-1)|v_1|^{-d_1/2}|v_3|^{1-d_3/2}
 \max\left(|v_1\otimes v_3|,|v_2|\right)^{1-d_2/2-2r}\\& \times 
  \left(\max(0,\mathrm{ord}(v_1 \otimes v_3)-\mathrm{ord}(v_2))
+\zeta(2r+d_2/2-1)\right)
\end{align*}
 for some constant $C>0$ which equals $1$ for $q$ sufficiently large.
 Thus if $\mathrm{Re}(s)=r$, the function $I_1(b_X \otimes \one_{V(\OO)} \otimes \one_{\OO^2})(v,s)$ admits the same bounds on its magnitude and support.  \qed
\end{lem}

\noindent An expression for the integrals $I_2$ and $I_3$ and corresponding bounds and supports can be obtained by symmetry.

\section{Bounds on integrals in the Archimedean case}
\label{sec:bound:arch}

In this section $F$ is an Archimedean local field. We estimate the local integrals defined in \S \ref{sec:loc:func}.  The bounds obtained in this section will be used to prove Propositions \ref{prop:Id:AC} and \ref{prop:other:AC}, the absolute convergence statements used in the proof of Theorem \ref{thm:main:intro}.  As usual, the bounds in the archimedian case are slightly harder to prove than in the nonarchimedian case, but the basic outline of the proofs is the same.  We let
\begin{align}
    |(a_1,\dots,a_{d_i})|:=\max\{|a_j|: 1 \leq j \leq d_i\}
\end{align}
for $(a_1,\dots,a_{d_i}) \in V_i(F).$
We moreover fix
$$
(f_1\otimes f_2,\Phi) \in \mathcal{S}(X(F)\times V(F)) \times \mathcal{S}(F^2)
$$
We will bound integrals involving the pure tensor $f_1 \otimes f_2$ in this section.   In each case, the bounds will be continuous in $f_1$ and $f_2$ with respect to the Fr\'echet topologies on $\mathcal{S}(X(F))$ and $\mathcal{S}(V(F)).$  Thus the bounds extend by continuity to all $f \in \mathcal{S}(X(F) \times V(F)).$

 The following is a rephrasing of \cite[Lemma 8.1]{Getz:Liu:Triple}:
 \begin{lem} \label{lem:basic:bound}
 Let $A,B \in \RR_{>0}$, $C \in \RR_{\geq 0}$ and let $x \in F^\times$.  Assume $A>B$ and $A\neq B+C$. One has
 \begin{align*}
 \int_{F^\times} \max(|a^{-1}x|,1)^{-A}|a|^{-B}\max(|a|,1)^{-C}da^\times \ll_{A,B,C}\max(|x|,1)^{-\min(A-B,C)}|x|^{-B}.
 \end{align*}\qed
 \end{lem} 
\noindent This will be used several times below.
 
\subsection{The open orbit}  Recall that $V'(F) \subset V(F)$ is the subset of vectors $(v_1,v_2,v_3)$ such that no two $v_i$ are zero. In order to bound the function $I(f_1 \otimes f_2) \in \mathcal{S}(Y(F))$ and its various derivatives it is convenient to first prove the following bound:  

\begin{lem}\label{lem:universal}  Given $r,e_i,N\in \RR_{\ge 0}$, $D\in U(\mathrm{Lie}(V(F)))$, let $M:(F^{\times})^{3}\times F\times K^3\times V(F)\to \RR$ be the function
\begin{align}
    M(a,t,k,v):=\max(m(t,a),1)^{-2N}|t|^{r}|D\rho(k)f_2|^2(a^{-1}v)\left(\prod_{i=1}^3 |a_i|^{2-d_i-e_i}\right),
\end{align}
where $m(t,a)$ is defined as \eqref{eq:gamma0gnorm}. 
For $v\in V'(F)$, there is a compact neighborhood $U$ of $v$, and a continuous integrable function $M'$ on $(F^{\times})^{3}\times F\times K^3$ such that for all $v'\in U$ 
\begin{align*}
    M(a,t,k,v')\le M'(a,t,k). 
\end{align*}
Moreover, given $N_i\in \ZZ_{\ge 0}$ there exists a continuous seminorm $\nu'$ on $\mathcal{S}(V(F))$ such that
\begin{align*}
     &\int_{(F^\times)^3\times F\times K^3}M(a,t,k,v)dkd^\times a dt\nonumber
     \\ &\le \nu'(f_2)^2\left\{\begin{array}{ll}
      \prod_{i=1}^3 \max(|v_i|,1)^{-2N_i}|v_i|^{1-d_i-e_i-r}    &  \textrm {if $v\in V^{\circ}(F)$,}\\
       \prod_{i\neq j} \max(|v_i|,1)^{-2N_i}|v_i|^{2-d_i-e_i-2r-d_j-e_j}  & \textrm {if $v_j=0$},\\
     \end{array}\right.
\end{align*}
provided $N\ge 5\max_i \{N_i, d_i+e_i+r\}$.
\end{lem}

This bound will be used in the proof of Proposition \ref{prop:a:smooth} below.

\begin{proof}  By the continuity of the Weil representation and compactness of $K$, for any $C_1,C_2,C_3\in \ZZ_{\ge 0}$, there exists a continuous seminorm $\nu_{D,C_1,C_2,C_3}$ on $\mathcal{S}(V(F))$  such that for all $(k,v)\in  K^3\times V(F)$ we have
\begin{align*}
    |D\rho(k)f_2|(v)\le \nu_{D,C_1,C_2,C_3}(f_2)\left(\prod_{i=1}^3\max(|v_i|,1)^{-C_i}\right).
\end{align*}
Let $U$ be a compact neighborhood of $v$ such that for $v'\in U,$ if $v_i\neq 0$ then $v'_i\neq 0$. Choose $v'\in U$ with minimum norm. 
Put
\begin{align} \label{M'}M'(a,t,k):=\nu_{D,C_1,C_2,C_3}(f_2)^2\max(m(t,a),1)^{-2N}|t|^{r}\prod_{i=1}^3 \max(|a_i^{-1}v'_i|,1)^{-2C_i}|a_i|^{2-d_i-e_i}.
\end{align}
Then $M(a,t,k,v) \leq M'(a,t,k)$ for all $v \in U$.  Thus to prove the lemma it suffices to show that for all $v \in V'(F)$ one has
\begin{align}\label{eq:open:a:universal} \begin{split}
   & \int_{(F^\times)^3 \times F} \max(m(t,a),1)^{-2N}|t|^{r} \left(\prod_{i=1}^3 \max(|a_i^{-1}v_i|,1)^{-2C_i}|a_i|^{2-d_i-e_i}d^\times a_i\right)dt \\
   &\le   \left\{\begin{array}{ll}
      \prod_{i=1}^3 \max(|v_i|,1)^{-2N_i}|v_i|^{1-d_i-e_i-r}    &  \textrm {if $v\in V^{\circ}(F)$,}\\
       \prod_{i\neq j} \max(|v_i|,1)^{-2N_i}|v_i|^{2-d_i-e_i-2r-d_j-e_j}  & \textrm {if $v_j=0$},\\
     \end{array}\right. \end{split}
\end{align}
provided that $N\ge 5\max_i \{N_i, d_i+e_i+r\}$.
We break the integral into $m(t,a)\le 1$ and $m(t,a)> 1$. Suppose $v\in V^{\circ}(F)$. In the range $m(t,a)\le 1$, we have $|t|\le |a_1a_2a_3|^{-1}$ and $|a_i|\le 1$ for all $i.$ Therefore the integral is bounded by
\begin{align}\label{eq:ultbound:a}
    &\int_{|a|\le 1}\prod_{i=1}^3 \max(|a_i^{-1}v_i|,1)^{-2C_i}|a_i|^{1-d_i-e_i-r} d^\times a\nonumber \\
    &\le \prod_{i=1}^3 \int_{F^\times}\max(|a_i|,1)^{-2N_i}\max(|a_i^{-1}v_i|,1)^{-2C_i}|a_i|^{1-d_i-e_i-r} d^\times a_i.
\end{align}
In the range $m(t,a)> 1$, applying the inequality
\begin{align*}
m(t,a)^4\ge \max(|ta_1a_2a_3|,1)\max(|a_1|,1)\max(|a_2|,1)\max(|a_3|,1),
\end{align*}
the contribution of this part of the integral is bounded by
\begin{align*}
    &\int_{(F^\times)^3\times F} \max(|ta_1a_2a_3|,1)^{-N/2}|ta_1a_2a_3|^{r}\\
    &\times \left(\prod_{i=1}^3\max(|a_i|,1)^{-N/2}\max(|a_i^{-1}v_i|,1)^{-2C_i}|a_i|^{2-d_i-e_i-r} \right) d^\times a dt\\
    &\ll_{N,r} \prod_{i=1}^3 \int_{F^\times} \max(|a_i|,1)^{-N/2}\max(|a_i^{-1}v_i|,1)^{-2C_i}|a_i|^{1-d_i-e_i-r}  d^\times a_i,
\end{align*}
  provided $N>2r+2$.
  Since $N\ge 4\max_i N_i$ by assumption the integral is bounded by \eqref{eq:ultbound:a}. 
The assertion then follows from Lemma \ref{lem:basic:bound} by setting $A=2C_i, B=r+d_i+e_i-1, C=2N_i$ and choosing $C_i$ so that $A-B>C$ for each $i$.

Now assume $v \in V'(F)-V^{\circ}(F)$.  By symmetry we may assume $v_1=0$. Let 
$$
|a|:=\max_{1 \leq i \leq 3}(|a_i|).
$$
If $m(t,a)\le 1$ then $|t|\le |a_1a_2a_3|^{-1}, |a|\le 1, |a_2a_3|\le |a_1|$. Therefore the contribution of $|m(t,a)| \le 1$ to the integral \eqref{eq:open:a:universal} is bounded by
\begin{align}\label{eq:ultbound:a:zero}
    &\int_{|a|\le 1, |a_2a_3|\le |a_1|} |a_1|^{1-d_1-e_1-r}\prod_{i=2}^3 \max(|a_i^{-1}v_i|,1)^{-2C_i}|a_i|^{1-d_i-e_i-r}  d^\times a_1d^\times a_2d^\times a_3\nonumber \\
    &\le  \prod_{i=2}^3\int_{F^\times} \max(|a_i|,1)^{-2N_i}\max(|a_i^{-1}v_i|,1)^{-2C_i}|a_i|^{2-d_i-e_i-2r-d_1-e_1} d^\times a_i.
\end{align}
For $m(t,a)>1$ we have the inequality
$$ m(t,a)^5\ge \max(|ta_1a_2a_3|,1)\max(|a_1|,1)\max(|a_2|,1)\max(|a_3|,1)\max(|a_1^{-1}a_2a_3|,1).$$ Thus the contribution of  $m(t,a)> 1$ to \eqref{eq:open:a:universal} is
bounded by a constant depending on $N$ times
\begin{align} \label{the:int}
     &\int_{(F^\times)^3} \max(|a_1^{-1}a_2a_3|,1)^{-2N/5} \left(\prod_{i=1}^3\max(|a_i|,1)^{-2N/5}\max(|a_i^{-1}v_i|,1)^{-2C_i}|a_i|^{1-d_i-e_i-r} \right) d^\times a
\end{align}
since $N>5r+5$.
The contribution of $|a_2a_3|\le |a_1|$ to \eqref{the:int} is dominated by  \eqref{eq:ultbound:a:zero} since $N\ge 5\max_i N_i$. In the range $|a_2a_3|\ge |a_1|$, one has that
\begin{align*}
    \int_{|a_1|\le |a_2a_3|} \max(|a_1|,1)^{-2N/5}|a_1|^{1-d_1-e_1-r+2N/5}d^\times a_1\ll_{N,e_1,r} \min(|a_2a_3|,1)^{1-d_1-e_1-r+2N/5}.
\end{align*}
Since $2N/5\ge d_1+e_1+r$, we deduce that the integral \eqref{the:int} is also dominated by \eqref{eq:ultbound:a:zero}. The assertion now follows from Lemma \ref{lem:basic:bound} by setting $A=2C_i, B=2r+d_1+e_1+d_i+e_i-2, C=2N_i$ and choosing $C_i$ so that $A-B>C$ for $i=2,3$.
\end{proof}

\begin{prop}\label{prop:a:smooth}
We have
$\mathcal{S}(Y(F))< C^{\infty}(Y^{\mathrm{sm}}(F))$. Moreover, for $f\in \mathcal{S}(Y(F))$ and $D\in U(\mathrm{Lie}(V(F)))$,
    $$ |Df(v)|\left(\prod_{i=1}^3\max(|v_i|,1)^{N_i}|v_i|^{(d_i+[F:\RR]^{-1}\deg D-1)/2}\right)$$
is bounded on $Y^{\mathrm{ani}}(F)$ for all $N_i \in \ZZ$.
\end{prop}

\begin{proof}
    Let $v_0 \in Y^{\mathrm{sm}}(F)$ and
      $D\in U(\mathrm{Lie}(V(F)))$.
      Let $\Delta:F \to F^3$ be the diagonal embedding.  Using the notation of Lemma \ref{lem:universal}  there is a neighborhood $U$ of $v_0$ such that for $v \in U$ the expression
      \begin{align*}
      |f_1(\gamma_b\begin{psmatrix} 1 & \Delta(t)\\ & 1 \end{psmatrix}\begin{psmatrix}a^{-1} & \\ & a\end{psmatrix}k)D\rho(\begin{psmatrix} 1 & \Delta(t)\\ & 1 \end{psmatrix}\begin{psmatrix}a^{-1} & \\ & a\end{psmatrix}k)f_2(v)| 
      \end{align*}
      is dominated by a finite sum of functions of the form
      $$
      |f_1(\gamma_b\begin{psmatrix} 1 & \Delta(t)\\ & 1 \end{psmatrix}\begin{psmatrix}a^{-1} & \\ & a\end{psmatrix}k)|\max(m(t,a),1)^{N} M(a,t,k,v)^{1/2}|a|^{-1} 
      $$
      where $M(a,t,k,v)$ is defined using various parameters $f_2$, $e_i$, $r$ depending on $D$. We recall that 
      $$
      m(t,a)=|\gamma_b\begin{psmatrix} 1 & \Delta(t)\\ & 1 \end{psmatrix}\begin{psmatrix}a^{-1} & \\ & a\end{psmatrix}k|
      $$
      by \eqref{eq:gamma0gnorm}.
      Thus applying the Cauchy-Schwarz inequality we have
      \begin{align*}
         & \int_{F \times (F^\times)^3 \times K}|f_1(\gamma_b\begin{psmatrix} 1 & \Delta(t)\\ & 1 \end{psmatrix}\begin{psmatrix}a^{-1} & \\ & a\end{psmatrix}k)\max(m(t,a),1)^{N}| M(a,t,k,v)^{1/2}|a|^{-1} |a|^2 d^\times a dtdk\\
          &\le  \left(\int_{G_{\gamma_b}(F) \backslash G(F)}|f_1(\gamma_bg)\max(|\gamma_b g|,1)^{N}|^2dg \right)^{1/2} \left(\int_{F \times (F^\times)^3 \times K}M(a,t,k,v) d^\times a dtdk\right)^{1/2}.
          \end{align*}
 The left integral converges by the argument in the proof of Lemma \ref{lem:Iker}, and the right converges by Lemma \ref{lem:universal}.  To obtain the bound in the lemma one simply keeps track of which parameters $r$ and $e_i$ are required in the argument above in terms of $\deg D$.  
 
 To prove that $\mathcal{S}(Y(F)) <C^\infty(Y^{\mathrm{sm}}(F))$ we apply the Leibniz integral rule.  
 To justify its application, we require a bound on $\int_{G_{\gamma_b}(F) \backslash G(F)}|f_1(\gamma_bg)  D\rho(g)f_2(v)|dg$ that is uniform in a small neighborhood of a given $v \in Y^{\mathrm{sm}}(F).$  Choose a compact neighborhood $U$ of $v$ as in the proof of Lemma \ref{lem:universal}.  Then by Lemma \ref{lem:universal} for $v \in U$ one has $M(a,t,k,v) \leq M'(a,t,k),$ defined as in \eqref{M'}.  
It suffices to show 
 $$
 \int_{F \times (F^\times)^3 \times K}|f_1(\gamma_b\begin{psmatrix} 1 & \Delta(t)\\ & 1 \end{psmatrix}\begin{psmatrix}a^{-1} & \\ & a\end{psmatrix}k)\max(m(t,a),1)^{N}| M'(a,t,k)^{1/2}|a|^{-1} |a|^2 d^\times a dtdk<\infty.
 $$
 This follows from \eqref{eq:open:a:universal} and the argument above.
 \end{proof}

\begin{rem}
By mimicking the proof above one can also bound $Df(v)$ when $v_i=0$ for some $i$.
\end{rem}

\subsection{The identity orbit}

\begin{prop} \label{a:prop:I1int}  Let $v \in V^{\circ}(F)$. Given a positive integer $N'$ and $\epsilon>0$, there are continuous seminorms $\nu$ on $\mathcal{S}(X(F))$ and $\nu'$ (depending on $N',\epsilon$) on $\mathcal{S}(V(F))$ such that one has the bound
\begin{align*}
 \int_{N_2^3(F) \backslash G(F)}|f_1(g)\rho(g)f(v)|dg  \le \nu(f_1)\nu'(f_2) \max\{|v_1||v_2||v_3|,1\}^{-N'} \prod_{i=1}^3 |v_i|^{-d_i/2-\epsilon}.
\end{align*} 
The function $I_0(f)(v)$ admits the same bound.
\end{prop}

\begin{proof}
By symmetry, we may assume $d_1\ge d_2\ge d_3$. Recall the seminorms $\nu_{D,N,\beta}$ mentioned in Lemma \ref{lem:seminorm}.  Arguing as in the proof of Proposition \ref{prop:I1int:na}, we see that the integral in the proposition is bounded by $\max(\nu_{\mathrm{Id},N,0}(f_1),\nu_{\mathrm{Id},0,0}(f_1))$ times
\begin{align*}
 & \  \int_{(F^{\times})^3} |a_1 a_2 a_3|^{-2}\max(|a_1a_2a_3|,1)^{-N}  \widetilde{f}_2(a^{-1} v)
\prod_{i=1}^3 |a_i|^{2-d_i/2} d^\times a_i,
 \end{align*}
 where $\widetilde{f}_2$ is defined as in \eqref{overk}.
 By the continuity of Weil representation and compactness of $K$, for any $N_1,N_2,N_3 \in \ZZ_{\geq 0}$, there exists a continuous seminorm $\nu'$ depending on $N_1,N_2,N_3$ such that the integral above is bounded by $\nu'(f_2)$ times 
\begin{align*}
\begin{split}
& \ \int_{(F^\times)^3 }\max(|a_1a_2a_3|,1)^{-N} \prod_{i=1}^3\max(|a_i^{-1}v_i|,1)^{-N_i}|a_i|^{-d_i/2}da_i^\times\\
&=\ \int_{(F^\times)^3}\max(|a_1|,1)^{-N}\max(|a_1^{-1}(a_2a_3)v_1|,1)^{-N_1}|a_1|^{-d_1/2} |a_2a_3|^{d_1/2}
\\
 &\hskip1in\times \prod_{i=2}^3\max(|a_i^{-1}v_i|,1)^{-N_i}|a_i|^{-d_i/2}da_i^\times.
\end{split}
\end{align*}
Here we have taken a change of variables $a_1 \mapsto (a_2a_3)^{-1}a_1$.  For the remainder of the proof all implicit constants are allowed to depend on $N_1,N_2,N_3,N,$ and we assume
\begin{align*} 
    N_i-d_i/2>N_{i-1}-d_{i-1}/2
\end{align*}
for each $i$ where $N_0=d_0=0$.

Taking $N>N_1+d_1/2$ and applying Lemma \ref{lem:basic:bound} with $A=N_1$, $B=d_1/2$ and $C=N$ to the $a_1$ integral we see that the above is bounded by
\begin{align} \nonumber
&\int_{(F^\times)^2}\max(|a_2a_3v_1|,1)^{-N_1+d_1/2}|a_2a_3v_1|^{-d_1/2}
\prod_{i=2}^3\max(|a_i^{-1}v_i|,1)^{-N_i}|a_i|^{(d_1-d_i)/2}d^\times a_i\\ \begin{split}
&=|v_1|^{-d_1/2}\int_{(F^\times)^2}\max(|a_2v_1|,1)^{-N_1+d_1/2}|a_2|^{-d_2/2}\\ &\hspace{2cm}\times
\max(|a_2^{-1}a_3v_2|,1)^{-N_2}\max(|a_3^{-1}v_3|,1)^{-N_3} |a_3|^{(d_2-d_3)/2}da_2^\times da_3^\times.\end{split} \label{tag}
\end{align}
Here we have taken a change of variables $a_2 \mapsto a_3^{-1}a_2$. 
The integral 
\begin{align} \label{a3:int}
& \int_{F^\times}
\max(|a_2^{-1}a_3v_2|,1)^{-N_2}\max(|a_3^{-1}v_3|,1)^{-N_3}|a_3|^{(d_2-d_3)/2} da_3^\times
\end{align}
breaks into the sum of four integrals
\begin{align*}
    \int_{|a_3|\le \min\left(\frac{|a_2|}{|v_2|},|v_3|\right)}+ \int_{ \min\left(\frac{|a_2|}{|v_2|},|v_3|\right)< |a_3|\le \frac{|a_2|}{|v_2|}}+\int_{\frac{|a_2|}{|v_2|} < |a_3|\le \max\left(\frac{|a_2|}{|v_2|},|v_3|\right)}+ \int_{ \max\left(\frac{|a_2|}{|v_2|},|v_3|\right)\le |a_3|}
\end{align*}
and this is bounded by a constant (depending only on $N_2,N_3$) times
\begin{align*}
    & |v_3|^{-N_3}\min\left(\tfrac{|a_2|}{|v_2|},|v_3|\right)^{N_3+(d_2-d_3)/2}+G_{d_2,d_3}(a_2,v_2,v_3)
    + \left(\frac{|a_2|}{|v_2|}\right)^{N_2}\\\times& \Bigg(|v_3|^{-N_3}\left(\max\left(\tfrac{|a_2|}{|v_2|},|v_3|\right)^{N_3-N_2+(d_2-d_3)/2}-\left(\tfrac{|a_2|}{|v_2|}\right)^{N_3-N_2+(d_2-d_3)/2}\right)\\
    &+\max\left(\tfrac{|a_2|}{|v_2|},|v_3|\right)^{-N_2+(d_2-d_3)/2}\Bigg)
\end{align*}
where
\begin{align*}
    G_{d_2,d_3}(a_2,v_2,v_3)=\left\{\begin{array}{lc}
    \left(\tfrac{|a_2|}{|v_2|}\right)^{(d_2-d_3)/2}-\min\left(\tfrac{|a_2|}{|v_2|},|v_3|\right)^{(d_2-d_3)/2} &  \textrm{if $d_2\neq d_3$},\\
    \log{\left(\tfrac{|a_2|}{|v_2|}\right)} -\log{\min\left(\tfrac{|a_2|}{|v_2|},|v_3|\right)}\  & \textrm{if $d_2=d_3$}.
    \end{array}\right.
\end{align*}
Thus \eqref{a3:int} is bounded by a constant times 
\begin{align*}
    F(a_2,v_2,v_3):=\left\{\begin{array}{lc}
    |v_3|^{(d_2-d_3)/2}\left(\frac{|a_2|}{|v_2||v_3|}\right)^{N_2}  & \textrm{if $|a_2|<|v_2||v_3|,$} \\|v_3|^{(d_2-d_3)/2}
   \left( \tfrac{|a_2|}{|v_2||v_3|}\right)^{(d_2-d_3)/2}  & \textrm{if $|a_2|\ge |v_2||v_3|,$} \textrm{ and }d_2 \neq d_3,\\
   1+\log\left(\tfrac{|a_2|}{|v_2||v_3|}\right)  & \textrm{if $|a_2|\ge |v_2||v_3|$} \textrm{ and }d_2 = d_3.
    \end{array}\right.
\end{align*}

Thus the original integral \eqref{tag} is bounded by a constant times
\begin{align*}
    &|v_1|^{-d_1/2}\int_{F^\times} \max(|a_2v_1|,1)^{-N_1+d_1/2}|a_2|^{-d_2/2}F(a_2,v_2,v_3)da_2^\times\\
    &=\left(\prod_{i=1}^3|v_i|^{-d_i/2}\right) \int_{F^\times} \max(|a_2|c,1)^{-N_1+d_1/2}|a_2|^{-d_2/2}\min(|a_2|,1)^{N_2}F'(a_2) da_2^\times,
\end{align*}
where $c=|v_1||v_2||v_3|$, and
\begin{align*}
    F'(a_2):=\left\{\begin{array}{ll}
       \max(|a_2|,1)^{(d_2-d_3)/2}&  \textrm{if $d_2\neq d_3,$}\\
        \log(\max(|a_2|,1))+1&   \textrm{if $d_2=d_3$}.
    \end{array}\right.
\end{align*}
Here we have changed variables $a_2\mapsto a_2(|v_2||v_3|)$. The assertion of the proposition now follows from taking a change of variable $a_2\mapsto a_2^{-1}$ and applying Lemma \ref{lem:basic:bound} with $A=N_1-d_1/2,B=\epsilon<1/2, C=N_2-d_2/2-\epsilon$.
\end{proof}

\subsection{The other orbits}

\begin{prop} \label{prop:arch:other}
Let $r=\mathrm{Re}(s)>0$ and $N \in \ZZ_{> 0}$, and assume  $N>\max(2r+d_2/2-2,
d_3/2-2)$.
For $v \in V^{
\circ}(F)$, there are continuous seminorms $\nu$ on $\mathcal{S}(X(F))$ and $\nu'$  on $\mathcal{S}(V(F))$, depending on $N$, such that 
\begin{align*}
&\int_{(N_2 (F) \times \Delta_1(\SL_2)(F)) \backslash G(F)}|f_1\left(\gamma_{1}g\right) |\\
&\times\int_{N_2(F) \backslash \SL_2(F)}\int_{F^\times}
|\rho\left(\left( I_2,h,\left(\begin{smallmatrix} 1 & \\ & -1 \end{smallmatrix} \right)h\left(\begin{smallmatrix} 1 & \\ & -1 \end{smallmatrix} \right)\right)g\right)f_2( v)\Phi((0,x)hp_1(g))||x|^{2r}d^\times xdhdg\\
&\le \Psi(2r)\nu(f_1)\nu'(f_2)|v_1|^{-N-d_1/2}\max\left(|v_2|,|v_3|\right)^{-2r-d_2/2-d_3/2+2}
\end{align*}
 where $\Psi:\RR_{>0} \to \RR$ is an analytic function.
 The function $I_1(f\otimes \Phi)(v,s)$ admits the same bound.
\end{prop}

\begin{proof}  
By Lemma \ref{lem:seminorm} one has $|f_1(g)| \le \nu_{\mathrm{Id,N,0}}(f_1)|g|^{-2-N}$ for any $N \geq 0$.  Thus  arguing as in Proposition \ref{prop:other:na:bound}, we see that the 
integral is bounded by
\begin{align}\label{eq:firsteq}\begin{split}
&\nu_{\mathrm{Id},N,0}(f_1)\int_{(F^\times)^3 \times F }m'(t,a)^{-2-N}\left(\int_{K^3\times F^{
\times}}\left|\rho(k_1,k_2,k_3)f_2(a^{-1}v)\Phi((0,x)k_2)\right||x|^{2r}d^\times xdk_1dk_2dk_3\right)\\
&\hspace{2cm}\times  dt |a_2|^{-2r}\prod_{i=1}^3|a_i|^{2-d_i/2}  d^\times a_i, \end{split}
\end{align}
where $m'(t,a)=\max(|a_1|,|a_1a_2a_3^{-1}|, |a_1a_2^{-1}a_3|,|a_1a_2a_3t|)$. For simplicity we assume $f_2=\otimes_{i=1}^3 f_{2i}$.  The general case merely requires more annoying notation. Applying the Cauchy-Schwarz inequality on the second copy of $K$, the inner integral is bounded by
\begin{align*}
    &\left(\prod_{i=1,3}\int_{K}|\rho_i(k_i)f_{2i}(a_i^{-1}v_i)|dk_i\right)\left(\int_{K}|\rho(k_2)f_{22}(a_2^{-1}v_2)|^2dk_2\right)^{1/2}\\& \times \int_{F^\times}\left(\int_{K}|\Phi((0,x)k_2)|^2 dk_2\right)^{1/2}|x|^{2r}d^\times x.
\end{align*}
The last factor is $\Psi(2r)$ for an appropriate analytic function $\Psi:\RR_{>0} \to \RR$. By the continuity of the Weil representation and compactness of $K$,
for any $N_1,N_2,N_3 \in \ZZ_{> 0}$, there exists a continuous seminorm $\nu'_{N_1,N_2,N_3}$ on $\mathcal{S}(V(F))$ such that the integral in \eqref{eq:firsteq} is bounded by $\Psi(2r)$ times
\begin{align*} 
&\nu'_{N_1,N_2,N_3}(f_2)\int_{(F^\times)^3 \times F} m'(t,a)^{-N-2}|a_2|^{-2r}\prod_{i=1}^3\max(|a_i^{-1}v_i|,1)^{-N_i}|a_i|^{2-d_i/2}da^\times dt.
\end{align*}
From now on all implicit constants are allowed to depend on $N_1,N_2,N_3,N$.  
 Let 
 $$
 m''(t,a):=\max(1,|a_2a_3^{-1}|,|a_2^{-1}a_3|,|a_2a_3t|)=\max (|a_2a_3^{-1}|,|a_2^{-1}a_3|,|a_2a_3t|).
 $$
We assume without loss of generality that $N_1>N+d_1/2$. We then write the integral above  as the product of 
 \begin{align*}
     \int_{F^\times} |a_1|^{-N-d_1/2}\max(|a_1^{-1}v_1|,1)^{-N_1}da_1^\times \ll |v_1|^{-N-d_1/2}
 \end{align*}
and
\begin{align} \label{int23}
\int_{(F^\times)^2 \times F}m''(t,a)^{-N-2}|a_2|^{-2r}\prod_{i=2}^3\max(|a_i^{-1}v_i|,1)^{-N_i}|a_i|^{2-d_i/2}da_2^\times da_3^\times dt.
\end{align}

Now consider \eqref{int23}.  We write it as the sum of 
\begin{align} \label{231}
&\int_{|a_2| \geq |a_3|}|a_2|^{-N-d_2/2-2r}\max(|a_2^{-1}v_2|,1)^{-N_2}
\\& \times 
\int_F \max(|a_3^{-1}|,|a_3t|)^{-N-2}\max(|a_3^{-1}v_3|,1)^{-N_3}|a_3|^{2-d_3/2}da_3^\times da_2^\times dt \nonumber
\end{align}
and
\begin{align} \label{232}
&\int_{|a_3|>|a_2|}|a_3|^{-N-d_3/2} \max(|a_3^{-1}v_3|,1)^{-N_3} \\&\times\int_{F} \max(|a_2^{-1}|,|a_2t|)^{-N-2}\max(|a_2^{-1}v_2|,1)^{-N_2}|a_2|^{2-d_2/2-2r}da_3^\times da_2^\times dt. \nonumber
\end{align}
Executing the $t$ integral in \eqref{231}, we see that it is bounded by a constant times
\begin{align} \label{231prime}\begin{split}
&\int_{|a_2| \geq |a_3|}|a_2|^{-N-d_2/2-2r}\max(|a_2^{-1}v_2|,1)^{-N_2}
\max(|a_3^{-1}v_3|,1)^{-N_3}|a_3|^{2+N-d_3/2}da_3^\times da_2^\times\\
&= \int_{1 \geq |a_3|}|a_2|^{2-d_2/2-d_3/2-2r}\max(|a_2^{-1}v_2|,1)^{-N_2}
\max(|a_3^{-1}a_2^{-1}v_3|,1)^{-N_3}|a_3|^{2+N-d_3/2}da_3^\times da_2^\times,
\end{split}
\end{align}
where the latter equation is obtained by taking a change of variables $a_3\mapsto a_2a_3$.  
 Similarly \eqref{232} is bounded by a constant times 
\begin{align} \label{232prime}
\begin{split}
&\int_{|a_3|>1}|a_3|^{-N-d_3/2} \max(|a_2^{-1}a_3^{-1}v_3|,1)^{-N_3} 
\max(|a_2^{-1}v_2|,1)^{-N_2}|a_2|^{2-d_2/2-d_3/2-2r}da_3^\times da_2^\times .
\end{split}
\end{align}
Carrying out the integral over $a_3$ directly in \eqref{231prime} we see that it is bounded by a constant times
\begin{align*}
\int_{F^\times }|a_2|^{2-d_2/2-d_3/2-2r}\max(|a_2^{-1}v_2|,1)^{-N_2}\max(|a_2^{-1}v_3|,1)^{-N_3}da_2^{\times}
\end{align*}
provided $N-d_3/2+2>0$. This bound is also valid for \eqref{232prime} provided $N+d_3/2>N_3$. The integral above is 
bounded by a constant times
\begin{align*}
  \max\left(|v_2|,|v_3|\right)^{-2r-d_2/2-d_3/2+2}
\end{align*}
provided $N_i>2r+d_2/2+d_3/2-2$ for $i=2,3$.
\end{proof}

\section{Absolute convergence} \label{sec:AC}

In this section, we prove the absolute convergence statements 
that make the proof of the summation formula in \S \ref{sec:summation} rigorous. Fix a number field $F.$
For the remainder of the section, we fix 
\begin{align*}
(f=f_1\otimes f_2,\Phi) \in \mathcal{S}(X(\A_F)\times V(\A_F)) \times \mathcal{S}(\A_F^2).
\end{align*}
All implicit constants are allowed to depend on $f\otimes \Phi$. For $y_i \in V^\circ_i(\A_F)$, we let
\begin{align*}
    |y_i|:=\prod_v|y_i|_v.
\end{align*}

\begin{lem}\label{lem:globalIdbound}
Let $1/2>\epsilon>0$ and a finite set of places $S$ containing the infinite places be given. For $y\in V^{\circ}(\A_F)$ such that  $|y_1|_v=|y_2|_v=1$ for all $v\not\in S$, there exists a Schwartz function $\Psi\in \mathcal{S}((V_1\otimes V_2\otimes V_3)(\A_F))$ (depending on $S,\epsilon$) such that
\begin{align*}
\int_{(N_2)^3(\A_F) \backslash G(\A_F)} 
\left|f_1(g)
\rho\left(g\right) f_2(y) \right|dg  \le \Psi(y_1\otimes y_2\otimes y_3)\prod_{i=1}^3|y_i|^{-d_i/2-\epsilon}.
\end{align*}
The function  $I_0(f)$ satisfies the same bound.
\end{lem}

\begin{proof} This follows from the local bounds in Propositions \ref{prop:I1int:na},  \ref{prop:I1comp}, and  \ref{a:prop:I1int}.
\end{proof}

Let $\GG_m^2$ act on $V^\circ$ via the restriction of the action \eqref{act:Gm2}.

\begin{prop} \label{prop:Id:AC}
The sum
\begin{align*}
\sum_{\xi \in V^\circ(F)/(F^\times)^2}\int_{N_2^3(\A_F) \backslash G(\A_F)} 
\left|f_1(g)
\rho\left(g\right) f_2( \xi) \right|dg
\end{align*}
is finite.
\end{prop}

\begin{proof}
Let 
$$
\{\mathfrak{a}_j \subset \OO:1 \leq j \leq k\}
$$ 
be a set of representatives for the ideal classes of $\OO,$ the ring of integers of $F.$  For every $\xi_i \in V_i(F)$, we can choose an $\alpha \in F^\times$ such that $\alpha \xi_i \in V_i(\OO)$ and
the greatest common denominator $\mathrm{gcd}(\alpha\xi_i)$ of the coefficients of $\alpha \xi_i$ is $\mathfrak{a}_j$ for some $1 \leq j \leq k$.  Using this observation, we see that the sum in the proposition is bounded by a constant times
\begin{align*}
\sum_{j_1,j_2=1}^k\sum_{\substack{\xi \in (V_1(\OO) \times V_2(\OO) \times V_3^\circ(F))/ (\OO^\times)^2\\ \gcd(v_1)=\mathfrak{a}_{j_1},\,\gcd(v_2)=\mathfrak{a}_{j_2}}}\int_{N_2^3(\A_F) \backslash G(\A_F)} 
\left|f_1(g)
\rho\left(g\right) f_2( \xi) \right|dg.
\end{align*}
Here $(\OO^\times)^2 < (F^\times)^2$ acts via the action \eqref{act:Gm2}.
Thus it suffices to fix a pair of ideals $\mathfrak{b}_1$ and $\mathfrak{b}_2$ and prove convergence of the sum
\begin{align*}
\sum_{\substack{\xi \in (V_1(\OO) \times V_2(\OO) \times V_3^\circ(F))/ (\OO^\times)^2\\ \gcd(\xi_1)=\mathfrak{b}_{1},\,\gcd(\xi_2)=\mathfrak{b}_{2}}}\int_{N_2^3(\A_F) \backslash G(\A_F)} 
\left|f_1(g)
\rho\left(g\right) f_2( \xi) \right|dg.
\end{align*}
Let $S$ be a finite set of places including the 
infinite places such that $\mathfrak{b}_i\widehat{\OO}^S=\widehat{\OO}^S$ for each $i$.  Then by Lemma \ref{lem:globalIdbound}, there exists $\Psi \in \mathcal{S}((V_1 \otimes V_2 \otimes V_3)(\A_F))$ such that the sum above is bounded by 
\begin{align*}
\sum_{\substack{\xi \in (V_1(\OO) \times V_2(\OO) \times V_3^{\circ}(F))/ (\OO^\times)^2\\ \gcd(\xi_1)=\mathfrak{b}_{1},\,\gcd(\xi_2)=\mathfrak{b}_{2}}}\Psi(\xi_1\otimes \xi_2\otimes \xi_3)\prod_{i=1}^3|\xi_i|^{-d_i/2-\epsilon}\le \sum_{\xi\in (V_1 \otimes V_2\otimes V_3)(F)} \Psi(\xi).
\end{align*}
Here we have used the fact that, by the product rule, 
\begin{align} \label{weak:prod}
|\xi_i|\geq 1.
\end{align}
for $\xi_i \in V^{\circ}_i(F)$.
\end{proof}

\begin{lem}\label{lem:globalotherbound}
Let constants $c>1/2>\epsilon>0$ be given. For $1/2+\epsilon<r<c$ and an integer 
$$
N>\max(d_1/2+d_3/2-2,2r+d_1/2+d_2/2-2),
$$
there exists a Schwartz function $\Psi\in \mathcal{S}(V_1(\A_F))$ (depending on $\epsilon,c$) such that
\begin{align*}
\begin{split}
    &\int_{G_{\gamma_1}(\A_F)\backslash G(\A_F)} |f_1(\gamma_1 g)|\int_{N_2(\A_F)\backslash \SL_2(\A_F)}\left|
\rho\left(\left( I_2,h,\left(\begin{smallmatrix} 1 & \\ & -1 \end{smallmatrix} \right)h\left(\begin{smallmatrix} 1 & \\ & -1 \end{smallmatrix} \right)\right)g\right)f_2( y)\Phi((0,x)hp_1(g))|x|^{2r}\right|d^\times xdhdg\\ 
&\le\Psi(y_1)|y_1|^{-N}|y_3|^{1-d_3/2-\epsilon} \prod_{v}\max\left(|y_2|_v,|y_3|_v\right)^{1-2r-d_2/2+\epsilon}.
\end{split}
\end{align*}
The function $I_1(f\otimes \Phi)(y,s)$ defines a holomorphic function of $s$ in the strip $\tfrac{1}{2}+\epsilon <\mathrm{Re}(s)<c$ for each $y$ and admits the same bound with $r=\mathrm{Re}(s)$.
\end{lem}

\begin{proof} Let $S$ be a finite set of places including the infinite places  such that $f_1^{S}=b_X^{S}$, $f_2^{S}=\one_{V(\widehat{\OO}^{S})}$ is fixed by $\rho(\SL_2^3(\widehat{\OO}^S))$ and $\Phi^{S}=\one_{(\widehat{\OO}^S)^{2}}$.  Assume moreover that $\psi_v$ is unramified for $v \not \in S$ and $F/\QQ$ is unramified at places of $\QQ$ not dividing places of $S.$ Using Lemma \ref{lem:otherbound} and Propositions \ref{prop:other:na:bound} and \ref{prop:arch:other}, for any given integers $N_1> 0$, $N> \max(d_1/2+d_3/2-2,2r+d_2/2+d_1/2-2)$, there exist $\Psi^\infty\in \mathcal{S}(V_1(\A_F^\infty))$ and a positive constant $C$ depending on $\epsilon$ and $c$, such that the integral is bounded by a constant depending on $N_1,N,\epsilon,c$ times
\begin{align}  \label{to:bound} \begin{split}
\Psi^{\infty}(y_1)&\prod_{v|\infty} \max(|y_1|_v,1)^{-N_1}|y_1|_v^{-N}\max\left(|y_2|_v,|y_3|_v\right)^{-2r-d_2/2-d_3/2+2}\\
    \times&\prod_{v\nmid\infty} \zeta_v(2r)\zeta_v(2r+d_2/2-1)|y_1|_v^{-d_1/2}|y_3|_v^{1-d_3/2}\max\left(|y_1\otimes y_3|_v,|y_2|_v\right)^{1-d_2/2-2r}\\
    \times& (C_v+\max(0,\ord_v(y_1\otimes y_3)-\ord_v(y_2))+\zeta_v(2r+d_2/2-1)) \end{split}
\end{align}
where $C_v\in \{C,0\}$ and $C_v=0$ for almost all $v$. Since $1/2+\epsilon <r $, we have
\begin{align*}
    &C_v+\max(0,\ord_v(y_1\otimes y_3)-\ord_v(y_2))+\zeta_v(2r+d_2/2-1)\\
    &\ll  \min(|y_1\otimes y_3|_v/|y_2|_v,1)^{-\epsilon}\zeta_v(2r+d_2/2-1)\\&=|y_1|_v^{-\epsilon}|y_3|_v^{-\epsilon}\max(|y_1\otimes y_3|_v,|y_2|_v)^{\epsilon}\zeta_v(2r+d_2/2-1).
\end{align*}
for all finite $v$.  Here the implied constant is equal to $1$ for $q_v$ sufficiently large in a sense independent of $y$.
 Thus \eqref{to:bound} is bounded by a constant depending on $c$ and $\epsilon$ times
\begin{align} \begin{split} \label{to:bound2}
\Psi^{\infty}(y_1)&\prod_{v|\infty} \max(|y_1|_v,1)^{-N_1}|y_1|_v^{-N}\max\left(|y_2|_v,|y_3|_v\right)^{-2r-d_2/2-d_3/2+2}\\
    \times&\prod_{v\nmid\infty} |y_1|_v^{-d_1/2-\epsilon}|y_3|_v^{1-d_3/2-\epsilon}\max\left(|y_1\otimes y_2|_v,|y_3|_v\right)^{1-d_2/2-2r+\epsilon}.  \end{split}
\end{align}
 For a finite place $v$, if $\Psi^{\infty}(y_1) \neq 0$ then $|y_1|_v\le C'_v$ for some constant $C'_v\ge 1$, which is $1$ for almost all $v$, and hence 
\begin{align*}
    \max\left(|y_1\otimes y_3|_v,|y_2|_v\right)&\ge {C'}_v^{-1}|y_1|_v\max(|y_3|_v,|y_2|_v).
\end{align*}
Thus \eqref{to:bound2} is bounded by a constant times
\begin{align*} \begin{split} 
\Psi^{\infty}(y_1)&\prod_{v|\infty} \max(|y_1|_v,1)^{-N_1}|y_1|_v^{-N}\max\left(|y_2|_v,|y_3|_v\right)^{-2r-d_2/2-d_3/2+2}\\
    \times&\prod_{v\nmid\infty} |y_1|_v^{1-d_1/2-d_2/2-2r}|y_3|_v^{1-d_3/2-\epsilon}\max\left(| y_2|_v,|y_3|_v\right)^{1-d_2/2-2r+\epsilon}.  \end{split}
\end{align*}
The desired inequality follows from $\max(|y_2|_v,|y_3|_v)\ge |y_3|_v$.
\end{proof}

\begin{prop} \label{prop:other:AC}
If  $f_2$ satisfies  \eqref{awayfromVcirc0} and $r\gg 1$, the sum
\begin{align*} 
&\sum_{\xi \in V_1^\circ(F) \times \mathbb{P}(V_2 \times V_3)(F)} \int_{G_{\gamma_1}(\A_F) \backslash G(\A_F)}|f_1(\gamma_1g)|\\
&\times\int_{N_2(\A_F) \backslash \SL_2(\A_F)}\int_{\A_F^\times}\left|
\rho\left(\left( I_2,h,\left(\begin{smallmatrix} 1 & \\ & -1 \end{smallmatrix} \right)h\left(\begin{smallmatrix} 1 & \\ & -1 \end{smallmatrix} \right)\right)g\right)f_2( \xi)\Phi((0,x)hp_1(g))\right||x|^{2r}d^\times xdhdg
\end{align*}
is finite. Therefore,
\begin{align*}
    \sum_{\xi\in Y_1(F)} I_1(f\otimes \Phi)(\xi,s) 
\end{align*}
defines a holomorphic function for $\mathrm{Re}(s)\gg 1$. Moreover, it extends to a meromorphic function of $\CC$, holomorphic except for possible simple poles at $s=0$ and $s=1$.  One has
\begin{align*}
    &\mathrm{Res}_{s=1} \sum_{\xi\in Y_1(F)}I_1(f\otimes \Phi)(\xi,s)\\&=\frac{\mathrm{Vol}(F^\times \backslash (\A_F^\times)^1)\widehat{\Phi}(0)}{2} \int_{G_{\gamma_{1}}(\A_F) \backslash G(\A_F)} f_1(\gamma_{1}g)\int_{[G_{\gamma_{1}}]}
\sum_{\xi \in V(F)}
\rho(hg)f_2( \xi)dh dg.
\end{align*}
and 
\begin{align} \label{is:abs:conv}
    \int_{G_{\gamma_{1}}(\A_F) \backslash G(\A_F)} \left|f_1(\gamma_{1}g)\right|\int_{[G_{\gamma_{1}}]}
\left|\sum_{\xi \in V(F)}
\rho(hg)f_2( \xi)dh \right|dg
\end{align}
is absolutely convergent.
\end{prop}

\noindent The corresponding assertions for the integrals $I_2(f \otimes \Phi)$ and $I_3(f \otimes \Phi)$ are valid by symmetry. 
\begin{proof} Choose $\epsilon$ sufficiently small. By Lemma \ref{lem:globalotherbound} and \eqref{weak:prod}, there exists $\Psi \in \mathcal{S}(V_1(\A_F))$ such that the sum is  bounded by a constant times
\begin{align*}
\sum_{(\xi_1,\xi) \in  V_1(F)  \times \mathbb{P}(V_2 \times V_3)(F)} \Psi(\xi_1)|\xi|^{1-d_2/2-2r+\epsilon}
 \ll_{\Psi} \sum_{\xi \in \mathbb{P}(V_2 \times V_3)(F)}|\xi|^{1-d_2/2-2r+\epsilon}.
\end{align*}
The right hand side is the height zeta function of $\mathbb{P}(V_2 \times V_3)(F).$  It converges for $r$ sufficiently large \cite[\S 3]{ChambertLoir:Lectures}.
This completes the proof of the first claim of the proposition, and we deduce that $   \sum_{\xi\in Y_1(F)} I_1(f\otimes \Phi)(\xi,s)$
is holomorphic for $\mathrm{Re}(s) \gg 1$.

To obtain the meromorphic continuation, we break down the integral $\sum_{\xi\in Y_1(F)} I_1(f\otimes \Phi)(\xi,s)$ into two sums of the form 
\begin{align} \label{before:break} \begin{split}
    &\int_{G_{\gamma_1}(\A_F)\backslash G(\A_F)}f_1(\gamma_1 g)\int_{[\SL_2]}
\sum_{\substack{\xi \in V(F)\\ \mathcal{Q}_1(\xi_1)=0}}
\rho\left(\left( I_2,h,\left(\begin{smallmatrix} 1 & \\ & -1 \end{smallmatrix} \right)h\left(\begin{smallmatrix} 1 & \\ & -1 \end{smallmatrix} \right)\right)g\right)f_2( \xi)\\
& \times\int\sum_{\delta \in B_2(F) \backslash \SL_2(F) }\Phi(x(0,1)\delta hp_1(g))|x|^{2s}d^\times x dhdg \end{split}
\end{align}
where the unspecified integral is over $|x| \geq 1$ or $|x| \leq 1$.  The contribution of $|x| \geq 1$ converges for $\mathrm{Re}(s)$ large and hence converges for all $s$.  
Using the Poisson summation formula on $F^2$, the contribution of $|x| \leq 1$ equals 
\begin{align*}
    &\int_{G_{\gamma_1}(\A_F)\backslash G(\A_F)}f_1(\gamma_1 g)\int_{[\SL_2]}
\sum_{\substack{\xi \in V(F)\\ \mathcal{Q}_1(\xi_1)=0}}
\rho\left(\left( I_2,h,\left(\begin{smallmatrix} 1 & \\ & -1 \end{smallmatrix} \right)h\left(\begin{smallmatrix} 1 & \\ & -1 \end{smallmatrix} \right)\right)g\right)f_2( \xi)\\
& \times\int_{|x|\le 1}\Bigg(\sum_{\delta \in B_2(F) \backslash \SL_2(F) }\widehat{\Phi}(x^{-1}(1,0)\delta^{-t} h^{-t}p_1(g)^{-t})|x|^{2s-2}d^\times x +\widehat{\Phi}(0)|x|^{2s-2}-\Phi(0)|x|^{2s}d^{\times }x\Bigg)dh dg.
\end{align*}
An argument similar to the argument proving the holomorphy of the $|x|>1$ contribution implies that the contribution of the sum over $\delta$  defines an entire function of $s.$ For $\mathrm{Re}(s)\gg 1$, the remaining contribution is
$$
\frac{\mathrm{Vol}(F^\times \backslash (\A_F^\times)^1)}{2}\left(\frac{\widehat{\Phi}(0)}{s-1}-\frac{\Phi(0)}{s}\right)\int_{G_{\gamma_{1}}(\A_F) \backslash G(\A_F)} f_1(\gamma_{1}g)\int_{[G_{\gamma_{1}}]}\sum_{\xi \in V(F)}
\rho(hg)f_2( \xi)dh dg.$$
Assuming that \eqref{is:abs:conv} is convergent, this term admits a meromorphic continuation to the $s$ plane, holomorphic except at $s \in \{0,1\}$ with poles and residues as specified.  To obtain the convergence of \eqref{is:abs:conv} one begins with
\begin{align*}
    &\int_{G_{\gamma_1}(\A_F)\backslash G(\A_F)}\left|f_1(\gamma_1 g)\right|\int_{[G_{\gamma_1}]}
\left|\sum_{\substack{\xi \in V(F)}}
\rho\left(hg\right)f_2( \xi)\right|\\
& \times\int\sum_{\delta \in B_2(F) \backslash \SL_2(F) }\Phi(x(0,1)\delta hp_1(g))|x|^{2s}d^\times x dhdg
\end{align*}
instead of \eqref{before:break},
argues as before, and then observes that one obtains an equality between \eqref{is:abs:conv} times 
$\frac{\mathrm{Vol}(F^\times \backslash (\A_F^\times)^1)}{2}\left(\frac{\widehat{\Phi}(0)}{s-1}-\frac{\Phi(0)}{s}\right)$
and a sum that converges for $\mathrm{Re}(s)$ large.  The absolute convergence statement follows.
\end{proof}

\section{The $L^2$-theory}\label{sec:l2theory}
 We now discuss the $L^2$-theory. Let $F$ be a local field of characteristic zero. We assume throughout this section that $Y^{\mathrm{sm}}(F) \subset Y(F)$ is nonempty, and hence is dense in the Hausdorff topology \cite[Remark 3.5.76]{Poonen:Rational}.
 
 We first improve the bound in \cite[Propositions 7.1 and 8.2]{Getz:Liu:Triple}:
 
 \begin{prop}  \label{prop:betabound} 
Assume $\frac{1}{2}>\beta\ge 0$ and $v \in V^{\circ}(F)$. Let 
$$
f=f_1 \otimes f_2 \in \mathcal{S}(X(F)) \otimes\mathcal{S}(V(F)).
$$
If $F$ is non-Archimedean then
\begin{align*} 
\int_{G_{\gamma_b}(F) \backslash G(F)} |f_1(\gamma_bg)\rho(g)f_2(v)|dg \ll_{f_1,f_2} 
  \prod_{i=1}^3|v_i|^{\beta/3-d_i/2+2/3}.
\end{align*}
The integral as a function of $v$ has support in $\omega^{-N}V(\OO)$ for some $N\in \ZZ$.  If $F$ is Archimedean then, given $N>0$, there is a continuous seminorm $\nu_{\beta,N}$ on $\mathcal{S}(X(F)\times V(F))$ such that
\begin{align*}
&\int_{G_{\gamma_b}(F)\backslash G(F)} |f_1(\gamma_bg)\rho(g)f_2(v)|dg \le  \nu_{\beta,N}(f_1\otimes f_2)
      \prod_{i=1}^3 \max(|v_i|,1)^{-N}|v_i|^{\beta/3-d_i/2+2/3}.
\end{align*}
 The function $I(f)$ satisfies the same bound and support constraint.
\end{prop}

\begin{proof} 
Assume for the moment that $F$ is non-Archimedean.  The bound on the support of the integral is part of \cite[Proposition 7.1]{Getz:Liu:Triple}, so we only require the bound on the magnitude.  By Lemma \ref{lem:seminorm} and Iwasawa decomposition, the integral is bounded by a constant depending on $\beta$ and $f_1$ times 
\begin{align} \label{nonarch1}
&\int_{m(t,a)\le c}m(t,a)^{-2+\beta} \widetilde{f}_2(a^{-1}v)\left(\prod_{i=1}^3 |a_i|^{2-d_i/2}\right) d^\times a dt
\end{align}
for some constant $c>0$, where $\widetilde{f}_2$
and $m(t,a)$ are defined as in \eqref{overk} and \eqref{eq:gamma0gnorm} respectively. Observe that
\begin{align*}
    m(t,a)&= |a_1a_2a_3|\max\left(|t|,|a_1|^{-2},|a_2|^{-2},|a_3|^{-2}\right).
\end{align*}
Thus \eqref{nonarch1} is equal to 
\begin{align} \label{former2}
&\int_{m(t,a)\le c}\max\left(|t|,|a_1|^{-2},
    |a_2|^{-2},|a_3|^{-2}\right)^{\beta-2}\widetilde{f}_2(a^{-1}v)\left(\prod_{i=1}^3 |a_i|^{\beta-d_i/2}\right) d^\times a dt.
\end{align}
 Since 
 \begin{align}\label{t:int}
     &\int_F \max(|t|,|a_1|^{-2},|a_2|^{-2},|a_3|^{-2})^{-2+\beta}dt \ll_\beta \min(|a_1|,|a_2|,|a_3|)^{2-2\beta}\leq  |a_1a_2a_3|^{2/3-2\beta/3},
 \end{align}
 the integral \eqref{former2} is bounded by a constant times
\begin{align*}
\int_{(F^\times)^3} \widetilde{f}_2(a^{-1}v)\left(\prod_{i=1}^3 |a_i|^{\beta/3-d_i/2+2/3}\right)d^\times a.
\end{align*}
Taking a change of variables
$a_i\mapsto a_i\varpi^{\ord(v_i)}$, the above is bounded by a constant times 
\begin{align*}
\prod_{i=1}^3 |v_i|^{\beta/3-d_i/2+2/3}\int_{q^{-N}\le |a_i|} \left(\prod_{i=1}^3 |a_i|^{\beta/3-d_i/2+2/3}\right)d^\times a
\end{align*}
for some $N$ depending on $f_2$.
The integral converges as $d_i\geq 2$ for all $i$.

Now assume $F$ is Archimedean. 
By Lemma \ref{lem:seminorm} and the argument above, for any $N' \in \RR$ there is a continuous seminorm $\nu_{\beta,N'}$ on $\mathcal{S}(X(F))$  such that the integral is bounded by $\nu_{\beta,N'}(f_1)$ times
\begin{align} \label{arch:case}
\int_{(F^\times)^3\times F} \max\left(|t|,
    |a_1|^{-2},|a_2|^{-2},|a_3|^{-2}\right)^{\beta-2}\max(m(t,a),1)^{-N'}\widetilde{f_2}(a^{-1}v)\left(\prod_{i=1}^3 |a_i|^{\beta-d_i/2}\right)dt d^\times a.
\end{align}
 Choose $N'$ so that $N'>3N$, and observe \begin{align*}
m(t,a)^3 \geq \max(|a_1|,1)\max(|a_2|,1)\max(|a_3|,1).
\end{align*}
Then for any $M>0$, there is a continuous seminorm $\nu'_{M}$ on $\mathcal{S}(V(F))$ such that \eqref{arch:case} is bounded by $\nu'_{M}(f_2)$ times
\begin{align*}
    &\int_{(F^\times)^3\times F} \max\left(|t|,|a_1|^{-2},|a_2|^{-2},|a_3|^{-2}\right)^{\beta-2}\\
    &\times \left(\prod_{i=1}^3 \max(|a_i|,1)^{-N}\max(|a_i^{-1}v_i|,1)^{-M} |a_i|^{\beta-d_i/2}\right)dt d^\times a.
\end{align*}
The proposition then follows from \eqref{t:int} (which is still valid for $F$ Archimedean) and Lemma \ref{lem:basic:bound}.
\end{proof}
 
Let $\Omega_V$ be a top degree form on $V(F)$ such that $|\Omega_V|$ is the Haar measure. 
We endow $Y^{\mathrm{sm}}(F)$ with the unique positive measure $dy=|\Omega_Y|$ such that $d(\mathcal{Q}_1-\mathcal{Q}_2) \wedge d(\mathcal{Q}_2-\mathcal{Q}_3) \wedge \Omega_Y$ is  $\Omega_V$ on $V(F)$. Since $Y^{\mathrm{sm}}(F) \subset Y(F)$ is dense, we can consider the $L^p$ space
\begin{align*}
L^p(Y(F)):=L^p(Y(F),dy):=L^p(Y^{\mathrm{sm}}(F),dy).
\end{align*}
We observe that for $r \in F^\times$, one has
\begin{align*}
r^{d_1+d_2+d_3}\Omega_V&=\Omega_V(rv)\\&=
d(\mathcal{Q}_1-\mathcal{Q}_2)(rv) \wedge d(\mathcal{Q}_2-\mathcal{Q}_3)(rv) \wedge \Omega_Y(rv)\\
&=r^4d(\mathcal{Q}_1-\mathcal{Q}_2) \wedge d(\mathcal{Q}_2-\mathcal{Q}_3) \wedge \Omega_Y(rv).
\end{align*}
Thus
\begin{align} \label{homog}
d(ry)=|r|^{d_1+d_2+d_3-4}dy.
\end{align}

\begin{prop}\label{prop:L2} Let $0<p \leq 2$.  One has $\mathcal{S}(Y(F)) < L^p(Y(F))$ and the inclusion is continuous if $F$ is Archimedean.
\end{prop}

\begin{proof} 
Let $f \in \mathcal{S}(Y(F))$. Since $\mathcal{S}(Y(F)) < C^\infty(Y^{\mathrm{sm}}(F))$ by Propositions \ref{prop:na:smooth} and \ref{prop:a:smooth}, we have
$$
\int_{Y^{\mathrm{sm}}(F)}|f(y)|^pdy=\int_{Y^{\mathrm{ani}}(F)}|f(y)|^pdy.
$$
We will therefore bound the integral on the right. 

Fix $\frac{1}{2}>\beta >0$.
  Assume first that $F$ is non-Archimedean.
By Proposition \ref{prop:betabound}, for some $c\in \ZZ$  one has
\begin{align*}
&\int_{Y^{\mathrm{ani}}(F)} |f(y)|^p dy\ll_f \int_{Y^{\mathrm{ani}}(F)\cap \varpi^{-c} V(\OO)}\left(\prod_{i=1}^3|y_i|^{p(\beta/3-d_i/2+2/3)}\right)  dy.
\end{align*}
Let $\alpha:=p\beta+2p-4+\sum_{i=1}^3(1-\tfrac{p}{2})d_i\ge 2\beta >0$.  
Using the homogeneity property
\eqref{homog}, the integral above is bounded by a constant depending on $c$ times
\begin{align*}
& \zeta(\alpha)\int_{\left\{y \in Y^{\mathrm{ani}}(F):1\le |y| <2\right\}}  \left(\prod_{i=1}^3|y_i|^{p(\beta/3-d_i/2+2/3)}\right) dy.
\end{align*}
Here we could just write $|y|=1$, but we have written $1\le |y| <2$ so that we can use the same formula in both Archimedean and non-Archimedean cases.

Now assume that $F$ is Archimedean. Fix $N>\alpha$.   Then by Proposition \ref{prop:betabound} there is a continuous seminorm $v_\beta$ on $\mathcal{S}(Y(F))$ such that
\begin{align*}
    \int_{Y^{\mathrm{ani}}(F)} |f(y)|^pdy \le  &\nu_\beta(f)^p \sum_{j=1}^\infty \int_{\{y \in Y^{\mathrm{ani}}(F):2^{-j}\leq |y| < 2^{1-j}\}} 
    \left( \prod_{i=1}^3|y_i|^{p(\beta/3-d_i/2+2/3)}\right) dy \\
    +&\nu_\beta(f)^p\sum_{j=0}^\infty \int_{\{y \in Y^{\mathrm{ani}}(F):2^j\leq  |y| <2^{j+1}\}
    } |y|^{-N}\left(\prod_{i=1}^3|y_i|^{p(\beta/3-d_i/2+2/3)}\right) dy.
    \end{align*}
Using the homogeneity property \eqref{homog} again, we see that this is
    \begin{align*}
    &\nu_\beta(f)^p\left(\sum_{j=1}^\infty 2^{-\alpha j}+\sum_{j=0}^\infty 2^{(\alpha-N)j}\right)\int_{\{y \in Y^{\mathrm{ani}}(F):1\le |y| <2\}} \left( \prod_{i=1}^3|y_i|^{p(\beta/3-d_i/2+2/3)} \right)dy.
\end{align*}

Thus for any $F$, we are reduced to showing that 
\begin{align*} 
    \int_{\{y \in Y^{\mathrm{ani}}(F):1\le |y| <2\}} \left( \prod_{i=1}^3|y_i|^{p(\beta/3-d_i/2+2/3)}\right) dy
\end{align*}
is finite.  By symmetry, it suffices to show that the integral
\begin{align} \label{is:finite}
    \int_{\{y \in Y^{\mathrm{ani}}(F):
    \max(|y_1|,|y_2|)\leq |y_3|,
    1\le |y_3| <2\}}  \left( \prod_{i=1}^3|y_i|^{p(\beta/3-d_i/2+2/3)}\right) dy
\end{align}
is finite. After a change of variables, we can assume that $\mathcal{Q}_i(v_i)=v^tc_iv$
where 
$$
c_i= \begin{psmatrix} c_{i1} & & \\ & \ddots & \\ & & c_{id_i}\end{psmatrix}$$
is diagonal. 
 Write $v_i=(v_{i1},\ldots,v_{id_i})$. For any $1 \leq j \leq d_1$, $1 \leq k \leq d_2$, we have
\begin{align*}
    \left|\det \begin{psmatrix}
        \partial_{v_{1j}} (\mathcal{Q}_1-\mathcal{Q}_2) & \partial_{v_{2k}} (\mathcal{Q}_1-\mathcal{Q}_2) \\
        \partial_{v_{1j}} (\mathcal{Q}_2-\mathcal{Q}_3) & \partial_{v_{2k}} (\mathcal{Q}_2-\mathcal{Q}_3) 
    \end{psmatrix}\right|= \left|\det \begin{psmatrix}
        2c_{1j}v_{1j} & -2c_{2k}v_{2k}\\
        &  2c_{2k}v_{2k}
    \end{psmatrix}\right|=|4c_{1j}c_{2k}||v_{1j}||v_{2k}|.
\end{align*}
Then for any $j$ and $k$ as above we have
\begin{align*}
    \left( \prod_{i=1}^3|y_i|^{p(\beta/3-d_i/2+2/3)}\right) dy=\frac{\left( \prod_{i=1}^3|y_i|^{p(\beta/3-d_i/2+2/3)}\right)}{|4c_{1j}c_{2k}||y_{1j}y_{2k}|}\frac{dy_1 dy_2 dy_3}{dy_{1j}dy_{2k}}
\end{align*}
outside a set of measure zero with respect to $dy$.  
Here $\frac{dy_i}{dy_{i1}}:=dy_{i2}\dots dy_{id_1}$, etc. and the values of $|y_{1j}|$ and $|y_{2k}|$ are given implicitly in terms of the other entries of $y$.  We can assume that $j$ and $k$ are chosen so that $|y_{1j}|=|y_1|$ and $|y_{2k}|=|y_{2}|$.
Therefore, setting $t_1=|y_1|$ and $t_2=|y_2|$, we see that \eqref{is:finite} is bounded by a constant times
\begin{align*}
&\int_{0}^2\int_{0}^2\left(\int_{(x_1,x_2)\in F^{d_1-1}\times F^{d_2-1}, |x_1|\le t_1,|x_2|\le t_2}  t_1^{p(\beta/3-d_1/2+2/3)-1}t_2^{p(\beta/3-d_2/2+2/3)-1}dx_1 dx_2\right) dt_1dt_2,
\end{align*}
which is finite.
\end{proof}

 Suppose $F$ is non-Archimedean. Let $K \leq \mathrm{Sp}_6(\OO)$ be a compact open subgroup and let $L^2(X^{\circ}(F))^K$ denote the space of functions on $X^{\circ}(F)$ that are right $K$-invariant and square-integrable.

\begin{lem} \label{lem:L2:est}
For $f \in L^2(X^{\circ}(F))^K$ we have 
$$
|f(x)| \leq \frac{\norm{f}_2}{|x|^2\mathrm{meas}(K)^{1/2}}
$$
for any $x\in X^{\circ}(F)$.
\end{lem}
\begin{proof}
For $f \in L^2(X^{\circ}(F))^K$ we  have 
\begin{align*}
    \norm{f}_2^2=\sum_{\gamma \in X^{\circ}(F)/K}\frac{|f(\gamma)|^2\mathrm{meas}(K)}{\delta_P(\gamma)}=\sum_{\gamma \in X^{\circ}(F)/K}|f(\gamma)|^2\mathrm{meas}(K)|\gamma|^4.
\end{align*}
\end{proof}

\begin{prop} \label{prop:I:Lloc}
For $F$ non-Archimedean and $f_2 \in \mathcal{S}(V(F))$, the map $I(\cdot \otimes f_2)$ extends to a continuous map
$$
I(\cdot \otimes f_2):L^2(X^{\circ}(F))^K\lto L^{2}_{\mathrm{loc}}(Y^{\mathrm{ani}}(F)).
$$
\end{prop}

\begin{proof} Assume $y \in Y^{\mathrm{ani}}(F)$.
By Lemma \ref{lem:L2:est}, for any $f_1 \in \mathcal{S}(X(F))$ the integral $I(f_1 \otimes f_2)(y)$
is bounded by $\mathrm{meas}(K)^{-1/2}$ times
\begin{align*}
    \norm{f_1}_2\int_{G_{\gamma_b}(F)\backslash G(F)} |\gamma_bg|^{-2}
   |\rho(g)f_2|(y)dg.
\end{align*}
The proposition thus follows from the argument proving Proposition \ref{prop:betabound} in the special case $\beta=0$.
\end{proof}

\section{The Fourier transform} \label{sec:appendix}

Let $F$ be a number field.

\begin{thm} \label{thm:FY}  Let $v$ be a place of $F$ and assume that $Y^\mathrm{sm}(F_v)$ is nonempty. 
There is a unique $\CC$-linear isomorphism $\mathcal{F}_Y:\mathcal{S}(Y(F_v)) \to \mathcal{S}(Y(F_v))$ such that $ \mathcal{F}_Y \circ I=I \circ \mathcal{F}_X$.  It is continuous if $v$ is Archimedean.  In particular there is a commutative diagram
\medskip
\begin{center}
\begin{tikzcd}[column sep=large]
\mathcal{S}(X(F_v)\times V(F_v))
\arrow[d, two heads, "I"]
\arrow[r, "\mathcal{F}_X"] &\mathcal{S}(X(F_v) \times V(F_v)) \arrow[d,two heads, "I"]\\
\mathcal{S}(Y(F_v)) \arrow[r,"\mathcal{F}_Y"]
& \mathcal{S}(Y(F_v)).
\end{tikzcd}
\end{center}
\medskip
\end{thm}

 We should pause to explain why this theorem is not obvious. Let
$$
C:=\mathcal{S}(X(F_v) \times V(F_v))_{\SL_2^3(F_v)}
$$
denote the space of coinvariants.  It is clear that the map $I$ factors through $C$ and yields a surjection $C\to \mathcal{S}(Y(F_v))$.  Since $\mathcal{F}_X$ is equivariant under the action of $\mathrm{SL}_2^3(F_v)< \mathrm{Sp}_6(F_v)$, it is clear that $\mathcal{F}_X$ descends to define an automorphism of $C$.  However, it is not clear that the map $C \to \mathcal{S}(Y(F_v))$ is injective.  
For instance, there are several orbits of $\SL_2^3(F_v)$ on $X(F_v)$, but the map $I$ depends only on the restriction of a function in $\mathcal{S}(X(F_v) \times V(F_v))$ to one of these orbits. Moreover, $\mathcal{S}(X(F_v))/\mathcal{S}(X^{\circ}(F_v))$ is infinite-dimensional as a representation of $\mathrm{Sp}_6(F_v)$ and not even of finite length in the Archimedean case.  The situation is even more complicated when we restrict to $\mathrm{SL}_2^3(F_v).$  Finally, based on the example of \cite{Getz:Quadrics} and its appendix, we expect that there are more complicated subquotients of $C$ that are the local analogues of the hypothetical global boundary terms that we have excluded from our treatment using our assumption \eqref{awayfromVcirc0} (see the paragraph containing \eqref{Theta}).  The injectivity of the map $C \to \mathcal{S}(Y(F_v))$ is more or less equivalent to the assertion that all of these  complicated subquotients can be recovered from $\mathcal{S}(Y(F_v)).$  
Fortunately, with the global-to-local proof we give below, we can completely avoid the issue of describing the subquotients of $C.$

  Using Theorem \ref{thm:FY}, many prior results can be stated more transparently.  
    For example, by \cite[Lemma 4.3]{Getz:Liu:Triple} we have
\begin{cor} For any place $v$ of $F$, $f \in \mathcal{S}(Y(F_v))$, and $h \in H(F_v)$ one has that
\begin{align*}
    \mathcal{F}_Y(L(h)f)=|\lambda(h)|^{\sum_{i=1}^3d_i/2-2}L\left(\frac{h}{\lambda(h)} \right)\mathcal{F}_Y(f).
\end{align*} \qed
\end{cor}
\noindent Unlike in other sections in this paper, we have not abbreviated $F_v$ by $F$.  We really require both $F$ and $F_v$ in this section because we will use a global-to-local argument to prove Theorem \ref{thm:FY}.    The global-to-local argument is fairly simple, and we invite the reader to skip to the proof of Theorem \ref{thm:FY} to see the basic idea.  

There is a somewhat hidden assumption on the base change $Y_{F_v}$ of $Y$ to $F_v$ in the statement of Theorem \ref{thm:FY}. Namely, we are assuming that {$Y_{F_v}$}
is the base change to $F_v$ of the scheme cut out of a triple of quadratic spaces over the number field $F$ by the simultaneous values of three quadratic forms.  This is no loss of generality since every characteristic zero local field is a localization of a number field  \cite[§25, Theorem 2]{Lorenz}, and every quadratic form over a local field is equivalent to the localization of a quadratic form over the corresponding number field.  Indeed, the latter assertion follows from the fact that every quadratic form over $F_v$ may be diagonalized \cite[\S I.2]{Lam} together with the fact that the natural map $F^\times \to F_v^\times/(F_v^\times)^2$ is surjective since $F^\times$ is dense in $F_v^\times.$

We claim, moreover, that upon replacing $F$ by another number field if necessary, we can assume that $Y^{\mathrm{sm}}(F) \neq \emptyset$.
By a change of basis, we may assume $\mathcal{Q}_i$ is associated to the diagonal matrix $\mathrm{diag}(c_{i1},\ldots,c_{id_i})$.  In the Archimedean case, we can assume $c_{ij}\in\{\pm 1\}$. If $F_v=\mathbb{C},$ take $F=\mathbb{Q}(i)$ and if $F_v=\RR,$ take $F=\mathbb{Q};$ in either case one checks that $Y^{\mathrm{sm}}(F) \neq \emptyset.$ Now suppose $F_v$ is non-Archimedean.   Let $\mathfrak{p}$ be the prime ideal corresponding to $v$.  We may assume $c_{ij}\in \mathcal{O}_{v},$ the ring of integers of $F_v,$ for all $i,j$. As $Y^{\mathrm{sm}}(F_v)$ is nonempty,  $Y^{\mathrm{sm}}(F_v)\cap V(\mathcal{O}_{v})$ is nonempty. Note that $A:=F^{\mathrm{sep}}\cap \mathcal{O}_{v}$ is the henselization of an excellent discrete valuation ring $\mathcal{O}_{\mathfrak{p}}$ whose completion is $\mathcal{O}_v$ (see e.g., \cite[tag 07QS]{stacks-project},
    \cite[Example 8.3.34]{LiuAG}).  Thus $A$ has the approximation property by \cite[Theorem 1.10]{Artin}. In particular, there exists $(y_1,y_2,y_3)\in Y^{\mathrm{sm}}(F_v)$ such that each coordinate of $y_i$ is algebraic over $F$. Let $E$ be the field extension of $F$ obtained by adjoining the coordinates of the $y_i$. Then $E_{w}=F_v$ for some $w|v$ and $Y^{\mathrm{sm}}(E)\neq \emptyset$. This justifies our claim.  
    
    Thus in proving Theorem \ref{thm:FY}, we can and do assume $Y^{\mathrm{sm}}(F) \neq \emptyset$.

\begin{thm} \label{thm:CTS} If $Y^{\mathrm{sm}}(F_v) \neq \emptyset$ for all $v$, then 
$Y^{\mathrm{sm}}(F)$ is nonempty and has dense image in $Y^{\mathrm{sm}}(F_v)$ for all $v.$ 
\end{thm}

\begin{proof}
    Since we have assumed $\dim V_i\ge 2$ and $\mathcal{Q}_i$ is nondegenerate for each $i$, this is a direct consequence of \cite[Corollaire in \S 4]{Colliot1982}.
\end{proof}

 For a place $v$ of $F$, consider the linear map
\begin{align*}
    T:\mathcal{S}(V_i(F_v))&\lto C^\infty(F_v)\\
   f &\longmapsto T(f)(\alpha):=\int_{V_i(F_v)} f(v)\psi(\alpha \mathcal{Q}_i(v)) dv.
\end{align*}
Let 
\begin{align}\label{eq:Tvanish}
    \mathcal{S}_{iv}=\{f\in \mathcal{S}(V_i(F_v))\,|\, \textrm{$T(f)=0$ and $f(0)=0$}\}.
\end{align}
Note that $\eqref{eq:Tvanish}$ coincides with $\eqref{intro:Siv}$ by the definition of the Weil representation. We set $\mathcal{S}_{0v}=\mathcal{S}_{1v}\otimes \mathcal{S}_{2v}\otimes \mathcal{S}_{3v}$. Clearly, $\mathrm{supp}\left(\rho(g)f\right)< V^\circ(F_v)$ for all $f\in \mathcal{S}_{0v}$ and $g\in \SL_2^3(F_v)$. The following lemma implies, in particular, $\mathcal{S}_{0v}$ is nontrivial when $v$ is a finite place above an odd prime.

\begin{lem}\label{lem:nontriv}
    If $v$ is a finite place lying above an odd prime then the kernel of $T$ is infinite dimensional.
\end{lem}

\begin{proof}By diagonalizing the quadratic form $\mathcal{Q}_i$ \cite[Corollary 2.4]{Lam}, we see it suffices to show the kernel of the linear map 
\begin{align*}
T':\mathcal{S}(F_v)&\lto C^\infty(F_v)\\
f &\longmapsto T'(f)(\alpha):=\int_{F_v} f(x)\psi_v(\alpha x^2)dx
\end{align*}
is infinite dimensional. Observe that if a nonzero $f$ lies in the kernel of $T'$, then so does the infinite dimensional vector space spanned by
$$
\{x \mapsto f(x \varpi_v^n):n \in \ZZ\} .
$$
 Therefore, it suffices to show $T'$ has nontrivial kernel. For $a\in \OO_v^\times$, let $U_{a}:=a+\varpi_v\OO_v$. Since $2$ does not divide the residual characteristic of $F_v$, the map $x\mapsto x^2$ induces a bijection $U_{a}\to U_{a^2}$. Therefore, 
 $$
 T'(\one_{U_{a}})(\alpha)=\frac{dx(\OO_{v})}{q_v}\psi_v(a^2\alpha ) \one_{\varpi_v^{N-1}\OO_v}(\alpha),
 $$
 where $N$ is the smallest integer such that $\psi_v$ is trivial on $\varpi_v^N\OO_v$. In particular, $T'(\one_{U_{1}})=T'(\one_{U_{-1}})$. Since $U_{1}$ and $U_{-1}$ are disjoint, the function $\one_{U_{1}}-\one_{U_{-1}}$ is nonzero and lies in the kernel of $T'$.   
\end{proof}

 Recall $Y^{\mathrm{ani}} \subset Y$ defined in \eqref{Yani}.

\begin{lem}\label{lem:totvanish}
Let $v$ be a place where $\mathcal{Q}_{iv}$ splits for all $1 \leq i \leq 3$. Suppose there exists $y_0\in Y^{\mathrm{ani}}(F_v)$ such that $I(\mathcal{F}_X(f_1)\otimes f_2)(y_0)=0$ for all $f_1\otimes f_2\in C^\infty_c( \gamma_b G(F_{v})) \otimes \mathcal{S}_{0v}$. Then 
\begin{equation*}
    I(\F_X(f_1)\otimes f_2)=0
\end{equation*}
for any $f_1\otimes f_2\in C^\infty_c( \gamma_b G(F_{v})) \otimes \mathcal{S}_{0v}.$ 
\end{lem}

\begin{proof}
Given $y\in Y^{\textrm{ani}}(F_{v})$, choose $h\in H(F_{v})$ such that $\lambda(h) h^{-1}y_0=y$. Then for $f_1\otimes f_2\in C^\infty_c( \gamma_b G(F_{v})) \otimes \mathcal{S}_{0v}$,
\begin{align*}
    I(\F_X(f_1)\otimes f_2)(y)&=L\left(\frac{h}{\lambda(h)}\right)I(\F_X(f_1)\otimes f_2)(y_0)=|\lambda(h)|^{-\sum_{i=1}^2 d_i/2}I(\F_X(f))(y_0)
\end{align*}
for some $f\in C^\infty_c( \gamma_b G(F_{v})) \otimes \mathcal{S}_{0v}$ by \cite[Lemma 4.3]{Getz:Liu:Triple}. Thus our hypothesis implies $I(\F_X(f_1)\otimes f_2)(y)=0$.  Since $I(\F_X(f_1)\otimes f_2)$ is continuous on $Y^{\mathrm{sm}}(F_v)$ by Propositions \ref{prop:na:smooth} and \ref{prop:a:smooth}, and  $Y^{\mathrm{ani}}(F_v) \subset Y^{\mathrm{sm}}(F_v)$ is dense by \cite[Remark 3.5.76]{Poonen:Rational}, we deduce the lemma.
\end{proof}

\begin{lem}\label{lem:nonzero}
Let $v$ be a finite place where $\mathcal{Q}_{iv}$ splits for $1 \leq i \leq 3$ and $\mathcal{S}_{0v}$ is nontrivial. For a given $y\in Y^{\mathrm{ani}}(F_v)$, there exists $f_1 \otimes f_2 \in C_c^\infty(\gamma_bG(F_v)) \otimes \mathcal{S}_{0v}$ such that  $I(\mathcal{F}_X(f_1)\otimes f_2)(y) \neq 0$.
\end{lem}

\begin{proof} 
Choose $f_1' \otimes f_2 \in \mathcal{S}(X(F_v)) \otimes \mathcal{S}_{0v}$ such that $I(f_1'\otimes f_2)\neq 0$.  For example, we could take $f_1'$ to be the characteristic function of a sufficiently small neighborhood of $\gamma_b$ in $\gamma_bG(F_v)$.  
Choose a compact open subgroup $K \leq \mathrm{Sp}_{6}(\OO_v)$ such that $f_1'$ is fixed by $K$.  
Finally, choose $f_{1n} \in C_c^\infty(\gamma_bG(F_v))^K$ indexed by $n \in \ZZ_{>0}$ such that
$$
\lim_{n \to \infty} f_{1n}  = \mathcal{F}_X^{-1}(f_1')
$$
in $L^2(X^{\circ}(F_v))^K$.  Then since $\mathcal{F}_X$ is an isometry of $L^2(X^{\circ}(F_v))^K$, we have $\mathcal{F}_X(f_{1n}) \to f_1'$ in $L^2(X^{\circ}(F_v))^K$.  Since
$$
I(\cdot \otimes f_2):L^2(X^{\circ}(F_v))^K \lto L^{2}_{\mathrm{loc}}(Y^{\mathrm{ani}}(F_v))
$$
is well-defined and continuous by Proposition \ref{prop:I:Lloc}, we deduce that
$$
I(\mathcal{F}_X(f_{1n}) \otimes f_2) \to I(f'_1 \otimes f_2) 
$$
in $L^2_{\mathrm{loc}}(Y^{\mathrm{ani}}(F_v))$ and hence $I(\mathcal{F}_X(f_{1n}) \otimes f_2) \neq 0$ for $n$ large enough. The statement thus follows from Lemma \ref{lem:totvanish}.
\end{proof}

\begin{proof}[Proof of Theorem \ref{thm:FY}]
We first prove that if $I(f_{v})=0$ then $I(\mathcal{F}_X(f_{v}))=0$.  Choose finite places $v_1$ and $2\nmid v_2$ distinct from $v$ such that $\mathcal{Q}_{iv_2}$ splits for $1\le i\le 3$. Suppose that $f_{v_1} \in \mathcal{S}(X(F_{v_1}) \times V(F_{v_1}))$ is chosen so that $\mathcal{F}_X(f_{v_1}) \in C_c^\infty(\gamma_b G(F_{v_1}) \times V(F_{v_1}))$ and $I(\mathcal{F}_X(f_{v_1})) \in C_c^\infty(Y^{\mathrm{sm}}(F_{v_1}))$ and that $f_{v_2} \in C_c^\infty(\gamma_bG(F_{v_2})) \otimes \mathcal{S}_{0v_2}$.
 Moreover, choose $f^{v_1v_2v} \in \mathcal{S}(X(\A_F^{v_1v_2v}) \times V(\A_F^{v_1v_2v}))$. Then applying Theorem \ref{thm:main:intro}, we obtain
$$
0=\sum_{y \in Y^{\textrm{sm}}(F)}I(\mathcal{F}_X(f_{v}f_{v_1}f_{v_2}f^{vv_1v_2}))(y).
$$
In particular, since $\mathcal{F}_X(f_{v_1}) \in C_c^\infty(\gamma_b G(F_{v_1}) \times V(F_{v_1}))$ and $f_{v_2} \in C_c^\infty(\gamma_bG(F_{v_2})) \otimes \mathcal{S}_{0v_2}$, all of the boundary terms in the formula vanish.  

We observe that $Y(F)$ is discrete in $Y(\A_F)$.  Let $y_0 \in Y^{\mathrm{ani}}(F).$ We claim that we can choose $f_{v_1}f_{v_2}f^{vv_1v_2}$ so that the right hand side is equal to $I(\F_X(f_{v_1}f_{v_2}f^{v_1v_2}))(y_0)$  where $I(\F_X(f_{v_1}f_{v_2}f^{vv_1v_2}))(y_0) \neq 0$. 
Indeed, by Lemma \ref{lem:smooth}, we can choose $f_{v_1}$ so that $I(\mathcal{F}_X(f_{v_1}))$ is any function in $C_c^\infty(Y^{\mathrm{sm}}(F_{v_1}))$.
Combining this with Lemma \ref{lem:smooth}, the computation of the basic function  in Proposition \ref{prop:unram:comp}, and Lemma \ref{lem:nonzero}, we deduce the claim.  

The claim implies that $I(\mathcal{F}_X(f_v))(y)=0$ for all $y \in Y^{\mathrm{ani}}(F)$.  
Since $Y^{\mathrm{ani}}(F)$ is dense in $Y^{\mathrm{sm}}(F_v)$ by \cite[Remark 3.5.76]{Poonen:Rational} and Theorem \ref{thm:CTS}, we can use the continuity of $I(\mathcal{F}_X(f_v))$ (Propositions \ref{prop:na:smooth} and \ref{prop:a:smooth}) to deduce that $I(\mathcal{F}_X(f_v))=0$.

We have shown that $\mathcal{F}_X(\ker I) \leq \ker I.$ On the other hand $\mathcal{F}_X \circ \mathcal{F}_X=\mathrm{Id}$ by Proposition \ref{prop:cont}, so  $\ker I=\mathcal{F}_X \circ \mathcal{F}_X(\ker I) \leq \mathcal{F}_X (\ker I),$ hence $\mathcal{F}_X(\ker I)=\ker I.$  This implies the theorem.
\end{proof}

The Fourier transform $\mathcal{F}_{X,\psi}$ and $I:=I_\psi$ depend on a choice of additive character $\psi$.  The dependence of $I$ on $\psi$ is through its dependence on the Weil representation $\rho=\rho_\psi.$   Thus $\mathcal{F}_Y$ also depends on $\psi.$  We write $\mathcal{F}_{Y,\psi}$ when we need to indicate this dependence.  Thus $\mathcal{F}_{Y,\psi}$ is determined by the relation
    \begin{align}
        \mathcal{F}_{Y,\psi} \circ I_{\psi}=I_{\psi}\circ \mathcal{F}_{X,\psi}.
    \end{align}

\begin{cor}
For $f\in \mathcal{S}(Y(F_v))$, we have
    \begin{align}\label{Fourierinversion}
        \F_Y^2(f)(v)&=f(v),\\
\label{Conjugation}
        \overline{\F_{Y,\psi}(f)}&=\F_{Y,\overline\psi}(\overline{f}),\\
        \label{same} \F_{Y,\psi}&=\F_{Y,\overline{\psi}}.
    \end{align}
\end{cor}

\begin{proof}
The first equation \eqref{Fourierinversion} is immediate from Proposition \ref{prop:cont}. As for \eqref{Conjugation},  by the explicit formula for $\mathcal{F}_{X,\psi}$ given in \cite[Corollary 6.11]{Getz:Hsu:Leslie}, for any $f_1 \in \mathcal{S}(X(F_v))$ one has
    \begin{align} \label{last:tag}
    \overline{\mathcal{F}_{X,\psi}(f_1)}=\mathcal{F}_{X,\overline{\psi}}(\overline{f}_1).
    \end{align}
Moreover, we claim that for $f_2 \in \mathcal{S}(V(F_v))$ one has 
\begin{align} \label{weil:claim}
\overline{\rho_{\psi}(g)f_2}=\rho_{\overline{\psi}}(g)\overline{f}_2
\end{align}
for all $g \in \SL_2^3(F_v)$. By the second corollary to \cite[Th\'eor\`eme 2]{Weil:Certains:groupes}, the Weil index $\gamma(\mathcal{Q}_i,\psi)$ satisfies the relation $\overline{\gamma(\mathcal{Q}_i,\psi)}=\gamma(\mathcal{Q}_i,\overline{\psi})$. Using this fact, one checks \eqref{weil:claim} by checking it on the same set of generators for $\SL_2^3(F_v)$ traditionally used to define the Weil representation (see \cite[\S 3.1]{Getz:Liu:Triple}, for example).
 Thus \eqref{weil:claim} is valid. Hence for $f_1 \otimes f_2 \in \mathcal{S}(X(F_v) \times V(F_v))$ one has 
    \begin{align*}
    \overline{I_{\psi}(\mathcal{F}_{X,\psi}(f_1) \otimes f_2)}=I_{\overline{\psi}}(\overline{\mathcal{F}_{X,\psi}(f_1) \otimes f_2})=I_{\overline{\psi}}(\mathcal{F}_{X,\overline{\psi}}(\overline{f}_1) \otimes \overline{f}_2).
    \end{align*}
    This implies \eqref{Conjugation}.
The space $\mathcal{S}(Y(F_v))$ is independent of the character $\psi$ by Lemma \ref{lem:Ipsi}.  Thus to show \eqref{same}, by \eqref{Fourierinversion} it suffices to show $\F_{Y,\psi}\circ \F_{Y,\overline{\psi}}(f)=f$ for functions $f$ of the form $I_{\overline{\psi}}(f_1\otimes f_2)$.  We compute
    \begin{align*}
        &\F_{Y,\psi}\circ \F_{Y,\overline{\psi}}(I_{\overline{\psi}}(f_1\otimes f_2))\\
        &= \F_{Y,\psi}(I_{\overline{\psi}}(\F_{X,\overline{\psi}}(f_1)\otimes f_2))  \\
        &= \frac{\gamma(\mathcal{Q},\overline{\psi})}{\gamma(\mathcal{Q},\psi)}\F_{Y,\psi}\circ L(-1)I_{\psi}(L(m(-1))R(\begin{smallmatrix} -I_3& \\
     & I_3
     \end{smallmatrix})\F_{X,\overline{\psi}}(f_1)\otimes f_2) \textrm{ (Lemma \ref{lem:Ipsi})}\\
        &= \frac{\gamma(\mathcal{Q},\overline{\psi})}{\gamma(\mathcal{Q},\psi)}\F_{Y,\psi}\circ I_{\psi}(L(m(-1))R(\begin{smallmatrix} -I_3& \\
     & I_3
     \end{smallmatrix})\F_{X,\overline{\psi}}(f_1)\otimes L(-1)f_2)\\
        &= \frac{\gamma(\mathcal{Q},\overline{\psi})}{\gamma(\mathcal{Q},\psi)}I_{\psi}\circ \F_{X,\psi} (L(m(-1))R(\begin{smallmatrix} -I_3& \\
     & I_3
     \end{smallmatrix})\F_{X,\overline{\psi}}(f_1)\otimes L(-1)f_2)\\
     &=\frac{\gamma(\mathcal{Q},\overline{\psi})}{\gamma(\mathcal{Q},\psi)}I_{\psi}\left(R(\begin{smallmatrix} I_3& \\
     & -I_3
     \end{smallmatrix})\F_{X,\psi} (L(m(-1))\F_{X,\overline{\psi}}(f_1)\otimes L(-1)f_2\right) \textrm{ (Lemma \ref{lem:BKF})}
     \\
        &= \frac{\gamma(\mathcal{Q},\overline{\psi})}{\gamma(\mathcal{Q},\psi)}I_{\psi}(R(\begin{smallmatrix} I_3& \\
     & -I_3
     \end{smallmatrix})\F_{X,\bar{\psi}} \circ \F_{X,\overline{\psi}}(f_1)\otimes L(-1)f_2) \textrm{ (Proposition \ref{prop:cont} and \eqref{last:tag})}\\
        &=\frac{\gamma(\mathcal{Q},\overline{\psi})}{\gamma(\mathcal{Q},\psi)}I_{\psi}(R(\begin{smallmatrix} I_3& \\
     & -I_3
     \end{smallmatrix})f_1\otimes L(-1)f_2) \textrm{ (Proposition \ref{prop:cont})}\\
     &= I_{\overline{\psi}}(L(m(-1))R(-I_6)f_1\otimes f_2) \textrm{ (Lemma \ref{lem:Ipsi})}\\
     &= I_{\overline{\psi}}(f_1\otimes f_2).
    \end{align*}
Here the last equality follows from the fact that
$m(-1)[P,P](F_v) =(-I_6)[P,P](F_v)$. 
\end{proof}

 We now explain how to deduce Theorem \ref{thm:PS} from Theorem \ref{thm:main:intro} and Theorem \ref{thm:FY}.

\begin{proof}[Proof of Theorem \ref{thm:PS}]
Given such $f$, we can choose 
$f_{1v_i} \in C_c^\infty(\gamma_b G(F_{v_i}))$ for $i=1,2$
such that $I(f_{1v_1} \otimes f_{v_1})=f_{v_1}$ and 
$I(f_{1v_2} \otimes f_{2v_2})=\mathcal{F}_Y(f_{v_2})$
where $f_{2v_2}|_{Y^{\mathrm{sm}}(F_{v_2})}=\mathcal{F}_Y(f_{v_2})$.
Indeed,  we can take $f_{1v_i}$ to be a scalar multiple of the characteristic function of a sufficiently small neighborhood of $\gamma_b$ in $\gamma_bG(F_{v_i})$.  

Moreover, choose $f'^{v_1v_2} \in \mathcal{S}(X(\A_F^{v_1v_2})\times V(\A_F^{v_1v_2}))$ such that $
I(f'^{v_1v_2})=f^{v_1v_2}.$
To deduce the theorem, we now apply Theorem \ref{thm:main:intro} to $f'=(f_{1v_1} \otimes f_{v_1})(\mathcal{F}_X^{-1}(f_{1v_2}) \otimes f_{v_2})f'^{v_1v_2}$.  Assumption \eqref{cptsupport0} is clearly valid, and \eqref{awayfromVcirc0} is valid by our hypotheses on $f_{v_1}$ and $\mathcal{F}_Y(f_{v_2})$.
By construction, the boundary terms vanish and the theorem is proved.
\end{proof}

\section*{List of symbols}

\begin{center}
\begin{longtable}{l c r}

$b_X$ &  basic function on $X$ & \eqref{bX} \\
$b_Y$ & basic function on $Y$ & \eqref{bY}\\

$f_{\chi_s}$ & local Mellin transform & \eqref{Mellin}\\
$|f|_{A,B,p}$ & seminorm & \eqref{semi:normf}\\
$|f|_{A,B,w,p_w,\Omega,D}$ & seminorm & \eqref{seminorm}\\
$\F_X$ & Fourier transform on $\mathcal{S}(X(F))$ &\S \ref{ssec:loc:Schwartz}\\
$\F_Y$ & Fourier transform on $\mathcal{S}(Y(F))$ &\S \ref{sec:appendix}\\

$G$ & $\SL_2^3$ &\eqref{Gembed}\\
$\mathfrak{g}$ & Lie algebra of $M^{\mathrm{ab}}\times \mathrm{Sp}_{2n}$ & \S \ref{ssec:loc:Schwartz}\\
$\gamma_i$ &  representatives of $X^{\circ}(F)/G(F)$ & \eqref{eq:gammas}\\
$G_{\gamma_i}$ & stabilizer of $\gamma_i$ in $G$ &\S \ref{sec:groups:orbits}\\
$|g|$ & the norm of $g$ under the Pl\"ucker embedding& \eqref{plucknorm}\\

$H$ & subgroup of a similitude group on $V$ & \eqref{H}\\

$I$ & integral operator attached to the representative $\gamma_b$ & \eqref{Is}\\
$I_0$ & integral  operator attached to the representative $\mathrm{Id}=I_6$ & \eqref{Is}\\
$I_i$ & integral  operator attached to the representative $\gamma_i$ & \eqref{Ii}\\
$L(m)R(g)$ & action of  $M^{\mathrm{ab}}(F) \times \mathrm{GSp}_6(F)$ on $\mathcal{S}(X(F))$& \eqref{LRaction}\\
\quash{$L(h)$ & left translation by $H$ & \\}
$\lambda(h)$ & similitude norm of $h$ & \eqref{lambda}\\

$m$ & an isomorphism $M^{\mathrm{ab}}(F)\to F^\times$ & \eqref{1:def}\\
$M$ & Levi subgroup of $P$ & \eqref{eq:levi}\\

$N$ & unipotent radical of $P$ & \eqref{eq:levi}\\

$N_2$ & standard maximal unipotent subgroup in $\SL_2$ & \S \ref{ssec:bt}\\\

$\omega$ & character of $M$ & \eqref{omega}\\
$\one_k$ & characteristic function of $\one_{[P,P](F)m(\varpi^k) \Sp_6(\OO)}$ & \eqref{1k}\\

$P$ & Siegel parabolic & \S \ref{sec:BK}\\
$\mathrm{Pl}$ & Pl\"ucker embedding & \eqref{pluckerembed}\\ 

$(V,\mathcal{Q})$ & $\prod_{i=1}^3 (V_i,\mathcal{Q}_i)$ & \S \ref{sec:intro}\\

$V^{\circ}$ & $\prod_{i=1}^3(V_i-\{0\})$ & \eqref{vcirc}\\
$(V_i,\mathcal{Q}_i)$ &  quadratic space of even dimension & \S \ref{sec:intro} \\

$X$ & affine closure of $X^{\circ}$ & \eqref{X}\\
$X^{\circ}$ & Braverman-Kazhdan space & \eqref{XP}\\

$Y$ & $\{ (y_1,y_2,y_3) \in V:\mathcal{Q}_1(y_1)=\mathcal{Q}_2(y_2)=\mathcal{Q}_3(y_3)\}$ & \eqref{Y}\\
$Y^{\mathrm{ani}}$ & anisotropic vectors in $Y$ & \S \ref{sec:groups:orbits}\\
$Y^{\mathrm{sm}}$ & smooth locus of $Y$ &\S \ref{sec:groups:orbits} \\
$Y_0$ & $\widetilde{Y_0}/\GG_m^2$& \eqref{quotient0}\\
$Y_i$ & $\widetilde{Y}_i/\GG_m$ &
\eqref{quotient0}\\
$\widetilde{Y_0}$ & vanishing locus of $\mathcal{Q}_1,\mathcal{Q}_2,\mathcal{Q}_3$ in $V^\circ$ &\eqref{tildeY0}\\
$\widetilde{Y_i}$ & $\{(y_1,y_2,y_3) \in V^{\circ}:\mathcal{Q}_{i-1}(y_{i-1})=\mathcal{Q}_{i+1}(y_{i+1}) \textrm{ and } \mathcal{Q}_i(y_i)=0\}$ & \eqref{tildeYi}
\end{longtable}
\end{center}


\bibliography{refs}
\bibliographystyle{alpha}

\end{document}